\documentclass[11pt]{amsart}

\usepackage{amssymb,amsthm,amsmath,url,relsize,etoolbox,cite,nameref,enumerate,fancyhdr}
\usepackage[utf8]{inputenc}
\usepackage{hyperref}

\theoremstyle{plain}
\newtheorem*{theorem*}{Theorem}
\newtheorem{theorem}[subsection]{Theorem}

\newtheorem{lemma}[subsection]{Lemma}

\newtheorem{proposition}[subsection]{Proposition}

\newtheorem{corollary}[subsection]{Corollary}

\newtheorem{definition}[subsection]{Definition}

\newtheorem{remark}[subsection]{Remark}

\DeclareMathOperator{\degree}{deg}
\DeclareMathOperator{\modulus}{mod}
\DeclareMathOperator{\residue}{Res}
\DeclareMathOperator{\rad}{rad}

\renewcommand{\Re}{\operatorname{Re}}

\newcommand*\conj[1]{\overline{#1}}
\newcommand{\sumstar}{\sideset{}{^*} \sum}

\begin{document}

\title{The Fourth Power Mean of Dirichlet $L$-functions in $\mathbb{F}_q [T]$}
\author{J. C. Andrade}
\author{M. Yiasemides}
\date{\today}
\address{Department of Mathematics, University of Exeter, Exeter, EX4 4QF, UK}
\email{j.c.andrade@exeter.ac.uk}
\email{my298@exeter.ac.uk}


\maketitle

\begin{abstract}
\noindent We prove results on moments of $L$-functions in the function field setting, where the moment averages are taken over primitive characters of modulus $R$, where $R$ is a polynomial in $\mathbb{F}_{q}[T]$. We consider the behaviour as $\degree R \rightarrow \infty$ and the cardinality of the finite field is fixed. Specifically, we obtain an exact formula for the second moment provided that $R$ is square-full, and an asymptotic formula for the fourth moment for any $R$. The fourth moment result is a function field analogue of Heath-Brown's result in the number field setting, which was subsequently improved by Soundararajan. Both the second and fourth moment results extend work done by Tamam in the function field setting who focused on the case where $R$ is prime.
\end{abstract}


\tableofcontents

\section{Introduction}

\noindent The study of moments of families $L$-functions is a central theme in analytic number theory. These moments are connected to the famous Lindel\"{o}f hypothesis for such $L$-functions and have many applications in analytic number theory. It is a very challenging problem to establish asymptotic formulas for higher moments of families of $L$-functions and until now we only have asymptotic formulas for the first few moments of any given family of $L$-functions. However, we do have precise conjectures for higher moments of families of $L$-functions due to the work of many mathematicians (see for example \cite{IntegralMomLFunc_CFKRS} and \cite{MultDirSerMomZetaLF_DiaconuEtAl}). In this paper the focus is on the moments of Dirichlet $L$-functions associated to primitive Dirichlet characters.  \\

\noindent In 1981, Heath-Brown \cite{FourthPowerMeanDLF_HeathBrown1981} proved that

\begin{align} \label{Heath-Brown Primitive Fourth Moment}
\sumstar_{\chi (\modulus q)} \Big\lvert L \Big( \frac{1}{2} , \chi \Big) \Big\rvert^4
= \frac{1}{2 \pi^2} \phi^* (q) \prod_{p \mid q} \frac{(1-p^{-1})^3}{(1+p^{-1})} (\log q)^4
+ O \big( 2^{\omega (q)} q (\log q)^3 \big) ,
\end{align}
where, for all positive integers $q$, $\sumstar_{\chi (\modulus q)}$ represents a summation over all primitive Dirichlet characters of modulus $q$, $\phi^* (q)$ is the number of primitive characters of modulus $q$, and $\omega (q)$ is the number of distinct prime divisors of $q$ and $L ( s , \chi )$ is the associated Dirichlet $L$-function. \\

\noindent In the equation above (\ref{Heath-Brown Primitive Fourth Moment}), in order to ensure that the error term is of lower order than the main term, we must restrict $q$ to

\begin{align*}
\omega (q) \leq \frac{\log \log q - 7 \log \log \log q }{\log 2} .
\end{align*}

\noindent Soundararajan \cite{FourthMomDLF_Sound2007} addressed this by proving that

\begin{align*}
\sumstar_{\chi (\modulus q)} \Big\lvert L \Big( \frac{1}{2} , \chi \Big) \Big\rvert^4
= \frac{1}{2 \pi^2} \phi^* (q) \prod_{p \mid q} \frac{(1-p^{-1})^3}{(1+p^{-1})} (\log q)^4 \bigg( 1 + O \Big( \frac{\omega (q)}{\log q} \sqrt{\frac{q}{\phi (q)}} \Big) \bigg)
+ O \big( q (\log q)^{\frac{7}{2}} \big) .
\end{align*}
Here, the error terms are of lower order than the main term without the need to have any restriction on $q$. \\

\noindent In a breakthrough paper, Young \cite{FourthMomDirLFunc_Young_2011} obtained explicit lower order terms for the case where $q$ is an odd prime and was able to establish the full polynomial expansion for the fourth moment of the associated Dirichlet $L$-functions. In other words, he proved that

\begin{align*}
\frac{1}{\phi^* (q)} \sumstar_{\chi (\modulus q)} \Big\lvert L \Big( \frac{1}{2} , \chi \Big) \Big\lvert^4
= \sum_{i=0}^{4} c_i (\log q)^i + O \big( q^{-\frac{5}{512} + \epsilon} \big) ,
\end{align*}
where the constants $c_i$ are computable. The error term was subsequently improved by Blomer \textit{et al.} \cite{MomTwistLF_BlomerFouvryEtAl} who proved that
\begin{align*}
\frac{1}{\phi^* (q)} \sumstar_{\chi (\modulus q)} \Big\lvert L \Big( \frac{1}{2} , \chi \Big) \Big\lvert^4
= \sum_{i=0}^{4} c_i (\log q)^i + O_{\epsilon} \big( q^{-\frac{1}{32} + \epsilon} \big) .
\end{align*}

\noindent In the function field setting Tamam \cite{FourthMoment_Tamam} established that

\begin{align*}
\frac{1}{\phi (Q)} \sum_{\substack{\chi (\modulus Q) \\ \chi \neq \chi_0}} \Big\lvert L \Big( \frac{1}{2} , \chi \Big) \Big\rvert^2
= \degree Q + \frac{1}{(q^{\frac{1}{2}} -1)^2} \bigg( 1 - \frac{2}{\lvert Q \rvert^{\frac{1}{2}} +1} \bigg) 
\end{align*}
and

\begin{align*}
\frac{1}{\phi (Q)} \sum_{\substack{\chi (\modulus Q) \\ \chi \neq \chi_0}} \Big\lvert L \Big( \frac{1}{2} , \chi \Big) \Big\rvert^4
= \frac{q-1}{12q} (\degree Q)^4 + O\big( (\degree Q)^3 \big) 
\end{align*}
as $\degree Q \rightarrow \infty$. Here, $Q$ is an irreducible, monic polynomial in $\mathbb{F}_{q}[T]$ with $\mathbb{F}_{q}$ a finite field with $q$ elements; $\chi_0$ is the trivial character (in this case, of modulus $Q$); and, for non-trivial characters of modulus $Q$, 

\begin{align*}
L \Big( \frac{1}{2} , \chi \Big) = \sum_{\substack{A \in \mathcal{M} \\ \degree A < \degree Q}} \frac{\chi (A)}{\lvert A \rvert^{\frac{1}{2}}} ,
\end{align*}
where $\mathcal{M}$ is the set of monic polynomials $\mathbb{F}_{q}[T]$. \\

\noindent In this paper we prove the function field analogue of Soundararajan's result, which is also an extension of Tamam's fourth moment result. In order to accomplish this we prove, along the way, a function field analogue of a special case of Shiu's Brun-Titchmarsh theorem for multiplicative functions \cite{BrunTitchTheoMultFunc_Shiu}. We also obtain an exact formula for the second moment, similar to Tamam's second moment result; our result holds for square-full moduli, $R$, whereas Tamam's result holds for irreducible moduli, $Q$.

\vspace{1cm}

\noindent \textbf{Acknowledgment:} The first author is grateful to the Leverhulme Trust (RPG-2017-320) for the support through the research project grant ``Moments of $L$-functions in Function Fields and Random Matrix Theory". The second author is grateful for an EPSRC Standard Research Studentship (DTP).


\section{Statement of Results}

\noindent Let $q \in \mathbb{N}$ be a prime-power. We denote the finite field of order $q$ by $\mathbb{F}_q$. We denote the ring of polynomials over the finite field $\mathbb{F}_q$ by $\mathcal{A} := \mathbb{F}_q [T]$. Unless otherwise stated, for a subset $\mathcal{S} \subset \mathcal{A}$ we define $\mathcal{S}_n := \{ A \in \mathcal{S} : \degree A = n \}$. We identify $\mathcal{A}_0$ with $\mathbb{F}_q$. Also, if we have some non-negative real number $x$, then range $\degree A \leq x$ is not taken to include the polynomial $A=0$. \\

\noindent The \emph{norm} of $A \in \mathcal{A} \backslash \{ 0\}$ is defined by $\lvert A \rvert := q^{\degree A}$, and for the zero polynomial we define $\lvert 0 \rvert := 0$. \\

\noindent We denote the set of monic polynomials in $\mathcal{A}$ by $\mathcal{M}$. For $a \in (\mathbb{F}_q)^*$ we denote the set of polynomials, with leading coefficient equal to $a$, by $a \mathcal{M}$. Because $\mathcal{A}$ is an integral domain, an element is prime if and only if it is irreducible. We denote the set of prime monic polynomials in $\mathcal{A}$ by $\mathcal{P}$, and all references to primes (or irreducibles) in the function field setting are taken as being monic primes. Also, when indexing, the upper-case letter $P$ always refers to a monic prime. Furthermore, if we range over polynomials $E$ that divide some polynomial $F$, then these $E$ are taken to be the monic divisors only. \\

\begin{definition}[Dirichlet Characters]
Let $R \in \mathcal{M}$. A \emph{Dirichlet character} on $\mathcal{A}$ with modulus $R$ is a function $\chi : \mathcal{A} \longrightarrow \mathbb{C}^*$ satisfying the following properties. For all $A,B \in \mathcal{A}$:

\begin{enumerate}
\item $\chi (AB) = \chi (A) \chi (B)$;
\item \label{Dirichlet character definition, equal mod R point} If $A \equiv B (\modulus R)$, then $\chi (A) = \chi (B)$;
\item $\chi (A) = 0$ if and only if $(A,R) \neq 1$.
\end{enumerate}
\end{definition}

\noindent Due to point \ref{Dirichlet character definition, equal mod R point}, we can view a character $\chi$ of modulus $R$ as a function on $\mathcal{A} \backslash R \mathcal{A}$. This makes expressions such as $\chi (A^{-1})$ well-defined for $A \in \big( \mathcal{A} \backslash R \mathcal{A} \big)^*$. \\

\noindent We can deduce that $\chi (1) =1$ and $\lvert \chi (A) \rvert =1$ when $(A,R)=1$. We say that $\chi$ is the \emph{trivial character} of modulus $R$ if $ \chi (A) =1$ when $(A,R)=1$, and this is denoted by $\chi_0$. Otherwise, we say that $\chi$ is \emph{non-trivial}. Also, there is only one character of modulus $1$ and it simply maps all $A \in \mathcal{A}$ to $1$. \\

\noindent It can easily be seen that the set of characters of a fixed modulus $R$ forms an abelian group under multiplication. The identity element is $\chi_0$. The inverse of $\chi$ is $\conj\chi$, which is defined by $\conj\chi (A) = \conj{\chi (A)}$ for all $A \in \mathcal{A}$. It can be shown that the number of characters of modulus $R$ is $\phi (R)$. \\

\noindent A character $\chi$ is said to be \emph{even} if $\chi (a) = 1$ for all $a \in \mathbb{F}_q$. Otherwise, we say that it is \emph{odd}. The set of even characters of modulus $R$ is a subgroup of the set of all characters of modulus $R$. It can be shown that there are $\frac{1}{q-1} \phi (R)$ elements in this group. \\

\begin{definition}[Primitive Character]
Let $R \in \mathcal{M}$, $S \mid R$ and $\chi$ be a character of modulus $R$. We say that $S$ is an \emph{induced modulus} of $\chi$ if there exists a character $\chi_1$ of modulus $S$ such that

\begin{align*}
\chi (A)
=\begin{cases}
\chi_1 (A) &\text{ if $(A,R)=1$} \\
0 &\text{ otherwise} .
\end{cases}
\end{align*}
$\chi$ is said to be \emph{primitive} if there is no induced modulus of strictly smaller norm than $R$. Otherwise, $\chi$ is said to be \emph{non-primitive}. $\phi^* (R)$ denotes the number of primitive characters of modulus $R$.
\end{definition}

\noindent We note that all trivial characters of some modulus $R \neq 1$ are non-primitive as they are induced by the character of modulus $1$. We also note that if $R$ is prime, then the only non-primitive character of modulus $R$ is the trivial character of modulus $R$. We denote a sum over primitive characters of modulus $R$ by the standard notation $\sumstar_{\chi (\modulus R)}$. \\

\begin{definition}[Dirichlet $L$-functions]
Let $\chi$ be a Dirichlet character. The associated L-function, $L(s,\chi )$, is defined for $\Re (s) > 1$ by

\begin{align*}
L(s,\chi )
:= \sum_{A \in \mathcal{M}} \frac{\chi (A)}{\lvert A \rvert^s} .
\end{align*}
This has an analytic continuation to either $\mathbb{C}$ or $\mathbb{C} \backslash \{ 1 \}$ depending on the character.
\end{definition}

\noindent In this paper, we will prove the following two main results. The first one is an exact formula for the second moment of Dirichlet $L$-functions in function fields.

\begin{theorem} \label{Second primitive moment function fields statement}
Let $R$ be a square-full polynomial. That is, if $P \mid R$ then $P^2 \mid R$. Then,

\begin{align*}
\sumstar_{\chi (\modulus R)} \Big\lvert L \Big( \frac{1}{2}, \chi \Big) \Big\rvert^2 = & \frac{\phi (R)^3}{\lvert R \rvert^2} \degree R + \bigg( \frac{\phi (R)^3}{\lvert R \rvert^2} - \frac{\phi (R)^2}{\lvert R \rvert^2}\bigg) \sum_{P \mid R} \frac{\degree P}{\lvert P \rvert -1} \\
&+ \frac{1}{\big( q^{\frac{1}{2}} -1 \big)^2} \bigg( - \frac{\phi (R)^3}{\lvert R \rvert^2} + 2 \frac{\phi (R)}{\lvert R \rvert^{\frac{1}{2}}} \prod_{P \mid R} \Big(1-\frac{1}{\lvert P \rvert^{\frac{1}{2}}} \Big)^2 \bigg) .
\end{align*}
\end{theorem}

\noindent The next result, is an asymptotic formula for the fourth moment of Dirichlet $L$-functions associated to primitive Dirichlet characters.

\begin{theorem} \label{Fourth primitive moment function fields statement}
Let $R \in \mathcal{M}$. Then,

\begin{align*}
\sumstar_{\chi (\modulus R)} \Big\lvert L \Big( \frac{1}{2} , \chi \Big) \Big\lvert^4
= &\frac{1-q^{-1}}{12} \phi^* (R) \prod_{\substack{P \text{ prime} \\ P \mid R}} \bigg( \frac{ \big( 1 - \lvert P \rvert^{-1} \big)^3}{1 + \lvert P \rvert^{-1}} \bigg) (\degree R)^4 \\
&+ O \Bigg( \phi^* (R) \bigg( \prod_{\substack{P \text{ prime} \\ P \mid R}}  \frac{\big( 1 - \lvert P \rvert^{-1} \big)^3 }{1 + \lvert P \rvert^{-1}} \bigg) (\degree R)^{\frac{7}{2}} \sqrt{\omega (R)} \Bigg) .
\end{align*}
\end{theorem}


\section{Function Field Background}

\noindent We provide some definitions and results relating to function fields that are needed in this paper. Many of these results are well known and so we do not provide a proof. Some proofs can be found in Rosen's book \cite{NumTheoFuncField_Rosen_2002}, particularly chapter 4.

\begin{definition}[M\"obius Function]
We define the M\"obius function, $\mu$, multiplicatively by $\mu (P) =-1$ and $\mu (P^e) =0$ for all primes $P \in \mathcal{A}$ and all integers $e \geq 2$.
\end{definition}

\begin{definition}[$\omega$ Function]
For all $R \in \mathcal{A} \backslash \{ 0 \}$ we define $\omega (R)$ to be the number of distinct prime factors of $R$. 
\end{definition}

\begin{definition}[$\Omega$ Function]
For all $R \in \mathcal{A} \backslash \{ 0 \}$ we define $\Omega (R)$ to be the total number of prime factors of $R$ (i.e. counting multiplicity). 
\end{definition}

\begin{definition}[$\phi$ Function] \label{Definition, totient function}
For $R \in \mathcal{A}$ with $\degree R =0$ we define $\phi (R) := 1$, and for $R \in \mathcal{A}$ with $\degree R \geq 1$ we define

\begin{align*}
\phi (R)
:= \# \{ A \in \mathcal{A} : \degree A < \degree R , (A,R)=1 \} .
\end{align*}
It is not hard to show that

\begin{align*}
\phi (R)
= \lvert R \rvert \prod_{P \mid R} (1 - \lvert P \rvert^{-1}) .
\end{align*}
\end{definition}

\begin{definition}
For all $R \in \mathcal{A}$ with $\degree R \geq 1$ we define $p_{-} (R)$ to be the largest positive integer such that if $P \mid R$ then $\degree P \geq p_{-} (R)$. Similarly, we define $p_{+} (R)$ to be the smallest positive integer such that if $P \mid R$ then $\degree P \leq p_{+} (R)$.
\end{definition}

\begin{proposition}[Orthogonality Relations]
Let $R \in \mathcal{M}$. Then,

\begin{align*}
\sum_{\chi (\modulus R)} \chi (A) \conj\chi (B)
= \begin{cases}
\phi (R) &\text{ if $(AB,R)=1$ and $A \equiv B (\modulus R)$} \\
0 &\text{ otherwise.}
\end{cases}
\end{align*}
and

\begin{align*}
\sum_{\substack{\chi (\modulus R) \\ \chi \text{ even}}} \chi (A) \conj\chi (B)
= \begin{cases}
\frac{1}{q-1} \phi (R) &\text{ if $(AB,R)=1$ and $A \equiv aB (\modulus R)$ for some $a \in \mathbb{F}_q$} \\
0 &\text{ otherwise.}
\end{cases}
\end{align*}
\end{proposition}

\begin{proposition} \label{Primitive chracter sum, mobius inversion}
Let $R \in \mathcal{M}$ and let  $A,B \in \mathcal{A}$. Then,

\begin{align*}
\sumstar_{\chi (\modulus R)} \chi (A) \tilde{\chi} (B) =
\begin{cases}
\sum_{\substack{EF = R \\ F \mid (A-B)}} \mu (E) \phi (F) &\text{ if $(AB,R)=1$}\\
0 &\text{ otherwise}
\end{cases} .
\end{align*}
\end{proposition}

\begin{corollary} \label{Number of primitive characters}
For all $R \in \mathcal{M}$ we have that

\begin{align*}
\phi^* (R) = \sum_{EF=R} \mu (E) \phi (F) .
\end{align*}
\end{corollary}

\begin{proof}
This follows easily from Proposition \ref{Primitive chracter sum, mobius inversion} when we take $A,B=1$.
\end{proof}

\noindent For a character $\chi$ we will, on occasion, write the associated $L$-function as

\begin{align*}
L(s,\chi ) = \sum_{n=0}^{\infty} L_n (\chi ) q^{-ns} ,
\end{align*}
where we define

\begin{align*}
L_n (\chi ) := \sum_{\substack{ A \in \mathcal{M} \\ \degree A = n}} \chi (A)
\end{align*}
for all non-negative integers $n$ and all characters $\chi$. \\

\noindent Suppose $\chi$ is the character of modulus $1$ and $\Re (s) > 1$. Then, $L(s,\chi )$ is simply the zeta-function for the ring $\mathcal{A}$. That is,

\begin{align*}
L(s,\chi )
= \sum_{A \in \mathcal{M}} \frac{1}{\lvert A \rvert^s} =: \zeta_{\mathcal{A}} (s) .
\end{align*}

\noindent We note further that

\begin{align*}
\zeta_{\mathcal{A}} (s)
= \sum_{A \in \mathcal{M}} \frac{1}{\lvert A \rvert^s}
= \frac{1}{1-q^{1-s}}.
\end{align*}
The far-RHS provides a meromorphic extension for $\zeta_{\mathcal{A}}$ to $\mathbb{C}$ with a simple pole at $1$. The following Euler product formula will also be useful

\begin{align*}
\zeta_{\mathcal{A}} (s)
= \prod_{P \in \mathcal{P}} \big( 1- \lvert P \rvert^{-s} \big)^{-1},
\end{align*}
for $\Re (s) > 1$. \\

\noindent Now suppose that $\chi_0$ is the trivial character of some modulus $R$ and $\Re (s) > 1$. It can be shown that

\begin{align*}
L(s,\chi_0 )
= \bigg( \prod_{\substack{P \in \mathcal{P} \\ P \mid R}} 1- \lvert P \rvert^{-s} \bigg) \zeta_{\mathcal{A}} (s) .
\end{align*}
So, again, the far-RHS provides a meromorphic extension for $L(s,\chi_0 )$ to $\mathbb{C}$ with a simple pole at $1$. \\

\noindent Finally, suppose that $\chi$ is a non-trivial character of modulus $R$ and $\Re (s) > 1$. It can be shown that

\begin{align*}
L(s,\chi )
= \sum_{\substack{A \in \mathcal{M} \\ \degree A < \degree R}} \frac{\chi (A)}{\lvert A \rvert^s} .
\end{align*}
This is just a finite polynomial in $q^{-s}$, and so it provides a holomorphic extension for $L(s,\chi )$ to $\mathbb{C}$.

\begin{proposition}[Functional Equation for $L$-functions of Primitive Characters]
Let $\chi$ be a primitive character of some modulus $R \neq 1$. If $\chi$ is even, then $L(s , \chi )$ satisfies the function equation

\begin{align*}
\big( q^{1-s} -1 \big) L(s,\chi )
= W(\chi ) q^{\frac{\degree R}{2}} \big(q^{-s} -1 \big) \big(q^{-s} \big)^{\degree R -1} L(1-s, \conj\chi ) ;
\end{align*}
and if $\chi$ is odd, then $L(s , \chi )$ satisfies the function equation

\begin{align*}
L(s,\chi )
= W(\chi ) q^{\frac{\degree R -1}{2}} \big(q^{-s} \big)^{\degree R -1} L(1-s, \conj\chi ) ;
\end{align*}
where $\lvert W(\chi ) \rvert =1$.
\end{proposition}
\noindent A generalisation of the proposition above appears in Rosen's book \cite[Theorem 9.24 A]{NumTheoFuncField_Rosen_2002}.

\begin{proposition} \label{Proposition, short sum for odd squared L-function}
Let $\chi$ a primitive odd character of modulus $R$. Then, 

\begin{align*}
\Big\lvert L \Big( \frac{1}{2} , \chi \Big) \Big\rvert^2
= 2 \sum_{\substack{A,B \in \mathcal{M} \\ \degree AB < \degree R}} \frac{\chi (A) \conj\chi (B)}{\lvert AB \rvert^{\frac{1}{2}}}
+ c_o (\chi ) ,
\end{align*}
where we define
\begin{align*}
c_o (\chi )
:= - \sum_{\substack{A,B \in \mathcal{M} \\ \degree AB = \degree R -1}} \frac{\chi (A) \conj\chi (B)}{\lvert AB \rvert^{\frac{1}{2}}} .
\end{align*}
\end{proposition}

\begin{proof}
The functional equation for odd primitive characters gives us that

\begin{align*}
\sum_{n=0}^{\degree R -1} L_n (\chi ) q^{-ns}
= &W(\chi ) q^{\frac{\degree R -1}{2}} \big(q^{-s} \big)^{\degree R -1} \sum_{n=0}^{\degree R -1} L_n (\conj\chi ) q^{-n(1-s)} \\
= &W(\chi ) q^{- \frac{\degree R -1}{2}} \sum_{n=0}^{\degree R -1} L_n (\conj\chi ) q^{(1-s) ( \degree R -1 -n )} .
\end{align*}
Taking the squared modulus of both sides gives us that

\begin{align*}
\sum_{n=0}^{2\degree R -2} \bigg(\sum_{\substack{0\leq i,j < \degree R \\ i+j=n}} L_i (\chi ) L_j (\conj\chi ) \bigg) q^{-ns} 
= q^{-\degree R +1} \sum_{n=0}^{2\degree R -2} \bigg(\sum_{\substack{0\leq i,j < \degree R \\ i+j=n}} L_i (\chi ) L_j (\conj\chi ) \bigg) q^{(1-s) ( 2\degree R -2 -n )} .
\end{align*}
We note that both sides are equal to $\lvert L ( s , \chi ) \rvert^2$, and so by the linear independence of powers of $q^{-s}$ we can take the terms $n=0, 1, \ldots , \degree R -1$ on the LHS and the terms $n=0, 1, \ldots , \degree R -2$ on the RHS to give

\begin{align*}
\lvert L ( s , \chi ) \rvert^2 = &\sum_{n=0}^{\degree R -1} \bigg(\sum_{\substack{0\leq i,j < \degree R \\ i+j=n}} L_i (\chi ) L_j (\conj\chi ) \bigg) q^{-ns}\\ 
+& q^{-\degree R +1} \sum_{n=0}^{\degree R -2} \bigg(\sum_{\substack{0\leq i,j < \degree R \\ i+j=n}} L_i (\chi ) L_j (\conj\chi ) \bigg) q^{(1-s) ( 2\degree R -2 -n )} .
\end{align*}
Hence,

\begin{align*}
\Big\lvert L \Big( \frac{1}{2} , \chi \Big) \Big\rvert^2 
= &\sum_{n=0}^{\degree R -1} \bigg(\sum_{\substack{0\leq i,j < \degree R \\ i+j=n}} L_i (\chi ) L_j (\conj\chi ) \bigg) q^{-\frac{n}{2}} 
+ \sum_{n=0}^{\degree R -2} \bigg(\sum_{\substack{0\leq i,j < \degree R \\ i+j=n}} L_i (\chi ) L_j (\conj\chi ) \bigg) q^{-\frac{n}{2}} \\
= &2 \sum_{\substack{A,B \in \mathcal{M} \\ \degree AB < \degree R}} \frac{\chi (A) \conj\chi (B)}{\lvert AB \rvert^{\frac{1}{2}}}
- \sum_{\substack{A,B \in \mathcal{M} \\ \degree AB = \degree R -1}} \frac{\chi (A) \conj\chi (B)}{\lvert AB \rvert^{\frac{1}{2}}} ,
\end{align*}
as required.
\end{proof}

\begin{proposition} \label{Proposition, short sum for even squared L-function}
Let $\chi$ a primitive even character of modulus $R\neq 1$. Then, 

\begin{align*}
\Big\lvert L \Big( \frac{1}{2} , \chi \Big) \Big\rvert^2
= 2 \sum_{\substack{A,B \in \mathcal{M} \\ \degree AB < \degree R}} \frac{\chi (A) \conj\chi (B)}{\lvert AB \rvert^{\frac{1}{2}}}
+ c_e (\chi ) ,
\end{align*}
where

\begin{align*}
c_e (\chi )
:= &-\frac{q}{\big( q^{\frac{1}{2}}-1 \big)^2} \sum_{\substack{A,B \in \mathcal{M} \\ \degree AB = \degree R -2}} \frac{\chi (A) \conj\chi (B)}{\lvert AB \rvert^{\frac{1}{2}}}
-\frac{2q^{\frac{1}{2}}}{ q^{\frac{1}{2}}-1} \sum_{\substack{A,B \in \mathcal{M} \\ \degree AB = \degree R -1}} \frac{\chi (A) \conj\chi (B)}{\lvert AB \rvert^{\frac{1}{2}}} \\
&+ \frac{1}{\big( q^{\frac{1}{2}}-1 \big)^2} \sum_{\substack{A,B \in \mathcal{M} \\ \degree AB = \degree R }} \frac{\chi (A) \conj\chi (B)}{\lvert AB \rvert^{\frac{1}{2}}} .
\end{align*}
\end{proposition}

\begin{proof}
The functional equation for even primitive characters gives us that

\begin{align}
\begin{split} \label{Even function equation replace with L_n (chi)}
\big( q^{1-s} -1 \big) \sum_{n=0}^{\degree R -1} L_n (\chi ) q^{-ns}
= &W(\chi ) q^{\frac{\degree R}{2}} \big(q^{-s} -1 \big) \big(q^{-s} \big)^{\degree R -1} \sum_{n=0}^{\degree R -1} L_n (\conj\chi ) q^{-n(1-s)} \\
= &W(\chi ) q^{- \frac{\degree R}{2}} \big(q^{1-s} -q \big) \sum_{n=0}^{\degree R -1} L_n (\conj\chi ) q^{(1-s) (\degree R -1 -n)}
\end{split}
\end{align}
For any primitive character $\chi_1$ of modulus $R \neq 1$, we define $L_{-1} (\chi_1 ) :=0$ and recall that $L_{\degree R} (\chi_1 ) =0$. If we define

\begin{align*}
M_i (\chi_1 ) := q L_{i-1} (\chi_1 ) - L_i (\chi_1 )
\end{align*}
for $i=0, 1, \ldots , \degree R$, then (\ref{Even function equation replace with L_n (chi)}) gives us that

\begin{align*}
\sum_{n=0}^{\degree R} M_n (\chi ) q^{-ns}
= &-W(\chi ) q^{- \frac{\degree R}{2}} \sum_{n=0}^{\degree R } M_n (\conj\chi ) q^{(1-s) (\degree R -n)} .
\end{align*}
Similarly as in the proof of Proposition \ref{Proposition, short sum for odd squared L-function}, we take the squared modulus of both sides, and we take the terms $n=0, 1, \ldots , \degree R$ from the LHS and the terms $n=0, 1, \ldots , \degree R -1$ from the RHS to give

\begin{align*}
\big( q^{1-s} - 1 \big)^2 \lvert L ( s , \chi ) \rvert^2 = &\sum_{n=0}^{\degree R} \bigg(\sum_{\substack{0\leq i,j \leq \degree R \\ i+j=n}} M_i (\chi ) M_j (\conj\chi ) \bigg) q^{-ns} \\
+& q^{-\degree R} \sum_{n=0}^{\degree R -1} \bigg(\sum_{\substack{0\leq i,j \leq \degree R \\ i+j=n}} M_i (\chi ) M_j (\conj\chi ) \bigg) q^{(1-s) ( 2\degree R -n )} .
\end{align*}
We now take $s =\frac{1}{2}$ and simplify to get

\begin{align*}
\big( q^{\frac{1}{2}} -1 \big)^2 \Big\lvert L \Big( \frac{1}{2} , \chi \Big) \Big\rvert^2 
= 2 \sum_{n=0}^{\degree R -1} \bigg(\sum_{\substack{0\leq i,j \leq \degree R \\ i+j=n}} M_i (\chi ) M_j (\conj\chi ) \bigg) q^{-\frac{n}{2}} 
+ \sum_{\substack{0\leq i,j \leq \degree R \\ i+j= \degree R}} M_i (\chi ) M_j (\conj\chi ) q^{-\frac{\degree R}{2}} .
\end{align*}
Now, 

\begin{align*}
&\sum_{n=0}^{\degree R -1} \bigg(\sum_{\substack{0\leq i,j \leq \degree R \\ i+j=n}} M_i (\chi ) M_j (\conj\chi ) \bigg) q^{-\frac{n}{2}} \\
= & \sum_{n=0}^{\degree R -1} q^2 \bigg( \sum_{\substack{0 \leq i,j \leq \degree R \\ i+j=n}} L_{i-1} (\chi ) L_{j-1} (\conj\chi ) \bigg) q^{-\frac{n}{2}}
- \sum_{n=0}^{\degree R -1} q \bigg( \sum_{\substack{0 \leq i,j \leq \degree R \\ i+j=n}} L_{i-1} (\chi ) L_{j} (\conj\chi ) \bigg) q^{-\frac{n}{2}} \\
&-\sum_{n=0}^{\degree R -1} q \bigg( \sum_{\substack{0 \leq i,j \leq \degree R \\ i+j=n}} L_{i} (\chi ) L_{j-1} (\conj\chi ) \bigg) q^{-\frac{n}{2}}
+\sum_{n=0}^{\degree R -1} \bigg( \sum_{\substack{0 \leq i,j \leq \degree R \\ i+j=n}} L_{i} (\chi ) L_{j} (\conj\chi ) \bigg) q^{-\frac{n}{2}} \\
= & \sum_{n=0}^{\degree R -3} q \bigg( \sum_{\substack{0 \leq i,j \leq \degree R -1 \\ i+j=n}} L_{i} (\chi ) L_{j} (\conj\chi ) \bigg) q^{-\frac{n}{2}}
-\sum_{n=0}^{\degree R -2} q^{\frac{1}{2}} \bigg( \sum_{\substack{0 \leq i,j \leq \degree R -1 \\ i+j=n}} L_{i} (\chi ) L_{j} (\conj\chi ) \bigg) q^{-\frac{n}{2}} \\
&-\sum_{n=0}^{\degree R -2} q^{\frac{1}{2}} \bigg( \sum_{\substack{0 \leq i,j \leq \degree R -1 \\ i+j=n}} L_{i} (\chi ) L_{j} (\conj\chi ) \bigg) q^{-\frac{n}{2}}
+\sum_{n=0}^{\degree R -1} \bigg( \sum_{\substack{0 \leq i,j \leq \degree R -1 \\ i+j=n}} L_{i} (\chi ) L_{j} (\conj\chi ) \bigg) q^{-\frac{n}{2}} \\
= &\big( q^{\frac{1}{2}} -1 \big)^2 \sum_{\substack{A,B \in \mathcal{M} \\ \degree AB < \degree R}} \frac{\chi (A) \conj\chi (B)}{\lvert AB \rvert^{\frac{1}{2}}} 
 - q \sum_{\substack{A,B \in \mathcal{M} \\ \degree AB = \degree R -2}} \frac{\chi (A) \conj\chi (B)}{\lvert AB \rvert^{\frac{1}{2}}}
+ \big(2q^{\frac{1}{2}} -q \big) \sum_{\substack{A,B \in \mathcal{M} \\ \degree AB = \degree R -1}} \frac{\chi (A) \conj\chi (B)}{\lvert AB \rvert^{\frac{1}{2}}} ,
\end{align*}
and similarly,

\begin{align*}
&\sum_{\substack{0\leq i,j \leq \degree R \\ i+j= \degree R}} M_i (\chi ) M_j (\conj\chi ) q^{-\frac{\degree R}{2}} \\
= & q \sum_{\substack{A,B \in \mathcal{M} \\ \degree AB = \degree R -2}} \frac{\chi (A) \conj\chi (B)}{\lvert AB \rvert^{\frac{1}{2}}}
- 2 q^{\frac{1}{2}} \sum_{\substack{A,B \in \mathcal{M} \\ \degree AB = \degree R -1}} \frac{\chi (B) \conj\chi (A)}{\lvert AB \rvert^{\frac{1}{2}}}
+ \sum_{\substack{A,B \in \mathcal{M} \\ \degree AB = \degree R }} \frac{\chi (A) \conj\chi (B)}{\lvert AB \rvert^{\frac{1}{2}}} .
\end{align*}
Hence,

\begin{align*}
\Big\lvert L \Big( \frac{1}{2} , \chi \Big) \Big\rvert^2 
= &2 \sum_{\substack{A,B \in \mathcal{M} \\ \degree AB < \degree R}} \frac{\chi (A) \conj\chi (B)}{\lvert AB \rvert^{\frac{1}{2}}} 
-\frac{q}{\big( q^{\frac{1}{2}}-1 \big)^2} \sum_{\substack{A,B \in \mathcal{M} \\ \degree AB = \degree R -2}} \frac{\chi (A) \conj\chi (B)}{\lvert AB \rvert^{\frac{1}{2}}} \\
&-\frac{2q^{\frac{1}{2}}}{ q^{\frac{1}{2}}-1} \sum_{\substack{A,B \in \mathcal{M} \\ \degree AB = \degree R -1}} \frac{\chi (A) \conj\chi (B)}{\lvert AB \rvert^{\frac{1}{2}}}
+ \frac{1}{\big( q^{\frac{1}{2}}-1 \big)^2} \sum_{\substack{A,B \in \mathcal{M} \\ \degree AB = \degree R }} \frac{\chi (A) \conj\chi (B)}{\lvert AB \rvert^{\frac{1}{2}}} ,
\end{align*}
as required.
\end{proof}

\noindent It is convenient to define

\begin{align} \label{Definition of c (chi)}
c(\chi )
:= \begin{cases}
c_e (\chi ) &\text{ if $\chi$ is even} \\
c_o (\chi ) &\text{ if $\chi$ is odd} .
\end{cases}
\end{align}


\section{Multiplicative Functions on $\mathbb{F}_q [T]$} \label{Section, Results on Well-known Multiplicative Functions}

\noindent In this section we state and prove some results for the functions $\mu , \phi $ and $\omega$ that are required for the proofs of the main theorems.

\noindent We will need the following well-known theorem.

\begin{theorem}[Prime Polynomial Theorem]
We have that
\begin{align*}
\# \mathcal{P}_n
= \frac{q^n}{n} + O \Big( \frac{q^\frac{n}{2}}{n} \Big) ,
\end{align*}
where the implied constant is independent of $q$. We reserve the symbol $\mathfrak{c}$ for the implied constant.
\end{theorem}

\noindent We will also need the following two definitions.

\begin{definition}[Radical of a Polynomial, Square-free, and Square-full]
For all $R \in \mathcal{A}$ we define the \emph{radical} of $R$ to be the product of all distinct monic prime factors that divide $R$. It is denoted by $\rad (R)$. If $R = \rad (R)$, then we say that $R$ is \emph{square-free}. If for all $P \mid R$ we have that $P^2 \mid R$, then we say that $R$ is \emph{square-full}.
\end{definition}

\begin{definition}[Primorial Polynomials]
Let $(S_i)_{i \in \mathbb{Z}_{> 0}}$ be a fixed ordering of all the monic irreducibles in $\mathcal{A}$ such that $\degree S_i \leq \degree S_{i+1}$ for all $i \geq 1$ (the order of the irreducibles of a given degree is not of importance in this paper). For all positive integers $n$ we define

\begin{align*}
R_n
:= \prod_{i=1}^{n} S_i .
\end{align*}
We will refer to $R_n$ as the $n$-th \emph{primorial}. For each positive integer $n$ we have unique non-negative integers $m_n$ and $r_n$ such that

\begin{align} \label{R_n decomposition}
R_n
= \bigg( \prod_{\degree P \leq m_n } P \bigg) \bigg( \prod_{i=1}^{r_n } Q_i \bigg) ,
\end{align}
where the $Q_i$ are distinct monic irreducibles of degree $m_n +1$. This definition of primorial is not standard.
\end{definition}

\noindent Now, before proceeding to prove results on the growth of the $\omega$ and $\phi$ functions, we note that

\begin{align} \label{Sum of mu(E)/E^s over E mid R}
\sum_{E \mid R} \frac{\mu (E)}{\lvert E \rvert^s} = \prod_{P \mid R} 1 - \frac{1}{\lvert P \rvert^s}
\end{align}
and

\begin{align} \label{Sum of mu(E) deg E/E^s over E mid R}
\sum_{E \mid R} \frac{\mu (E) \degree E}{\lvert E \rvert^s} = - \bigg( \prod_{P \mid R} 1 - \frac{1}{\lvert P \rvert^s} \bigg) \bigg( \sum_{P \mid R} \frac{\degree P}{\lvert P \rvert^s -1} \bigg)
\end{align}
for all $R \in \mathcal{A} \backslash \{ 0 \}$ . The first equation holds for all $s \in \mathbb{C}$. The second holds for all $s \in \mathbb{C} \backslash \{ 0 \}$ and is obtained by differentiating the first with respect to $s$. \\

\noindent Also, for all square-full $R \in \mathcal{A} \backslash \{ 0 \}$ we have that

\begin{align}
\begin{split} \label{Sum of mu(E)phi(F)/F^s over EF=R}
\sum_{EF = R} \frac{\mu (E) \phi (F)}{\lvert F \rvert^s}
= &\sum_{EF = R} \mu (E) \frac{\phi (R)}{\lvert E \rvert} \frac{\lvert E \rvert^s}{\lvert R \rvert^s}
= \frac{\phi (R)}{\lvert R \rvert^s} \sum_{EF = R} \frac{\mu (E)}{\lvert E \rvert^{1-s}} \\
= &\frac{\phi (R)}{\lvert R \rvert^s} \prod_{P \mid R} 1 - \frac{1}{\lvert P \rvert^{1-s}}
\end{split}
\end{align}
and

\begin{align} \label{Sum of mu(E)phi(F) deg F/F^s over EF=R}
\sum_{EF = R} \frac{\mu (E) \phi (F) \degree F}{\lvert F \rvert^s}
= &\frac{\phi (R)}{\lvert R \rvert^s} \bigg( \prod_{P \mid R} 1 - \frac{1}{\lvert P \rvert^{1-s}} \bigg) \bigg( \degree R + \sum_{P \mid R} \frac{\degree P}{\lvert P \rvert^{1-s} -1} \bigg) ,
\end{align}
The first equation holds for all $s \in \mathbb{C}$. The second holds for all $s \in \mathbb{C} \backslash \{ 1 \}$ and is obtained by differentiating the first with respect to $s$.

\begin{lemma} \label{Primorial bound}
For all positive integers $n$ we have that

\begin{align*}
\log_q \log_q \lvert R_n \rvert
= m_n + O(1) .
\end{align*}
\end{lemma}

\begin{proof}
By (\ref{R_n decomposition}) and the prime polynomial theorem, we see that

\begin{align*}
\log_q \lvert R_n \rvert 
= \degree R_n 
\leq \bigg( \sum_{i=1}^{m_n +1} q^i + O \Big( q^{\frac{i}{2}} \Big) \bigg)
\ll  q^{m_n +1} + O \Big(  q^{\frac{m_n +1}{2}} \Big)
\end{align*}
and

\begin{align*}
\log_q \lvert R_n \rvert 
= \degree R_n 
\geq \bigg( \sum_{i=1}^{m_n} q^i + O \Big( q^{\frac{i}{2}} \Big) \bigg)
\gg  q^{m_n} + O \Big(  q^{\frac{m_n }{2}} \Big) . 
\end{align*}
By taking logarithms of both equations above, we deduce that

\begin{align*}
\log_q \log_q \lvert R_n \rvert
= m_n + O (1) .
\end{align*}
\end{proof}

\noindent Using this result we can obtain results about the growth of the $\omega$ function.

\begin{lemma} \label{omega function at primorials}
For all primorials $R_n$ with $n > 1$ we have that

\begin{align*}
\omega (R_n )
= \frac{\log_q \lvert R_n \rvert}{\log_q \log_q \lvert R_n \rvert} + O \bigg( \frac{\log_q \lvert R_n \rvert}{( \log_q \log_q \lvert R_n \rvert )^2} \bigg) .
\end{align*}
\end{lemma}

\begin{proof}
The cases when $m_n = 0$ are not difficult. We proceed with the cases where $m_n \geq 1$. Using the prime polynomial theorem, we have that

\begin{align}
\begin{split} \label{omega R_n sum telescoped}
\omega (R_n )
= \sum_{i=1}^{m_n} \bigg( \frac{q^i}{i} + O\Big( \frac{q^{\frac{i}{2}}}{i} \Big)\bigg) + r_n
= \frac{1}{q-1} \bigg(\sum_{i=1}^{m_n} \frac{q^{i+1}}{i+1} - \frac{q^i}{i} + \frac{q^{i+1}}{i(i+1)} \bigg) + O \bigg( \sum_{i=1}^{m_n} \frac{q^{\frac{i}{2}}}{i} \bigg) + r_n .
\end{split}
\end{align}
We see that

\begin{align*}
\sum_{i=1}^{m_n} \frac{q^{\frac{i}{2}}}{i}
\leq \sum_{i=1}^{\frac{m_n}{2}} \frac{q^{\frac{i}{2}}}{i} + \sum_{i=\frac{m_n}{2}}^{m_n} \frac{q^{\frac{i}{2}}}{i}
\leq \sum_{i=1}^{\frac{m_n}{2}} q^{\frac{i}{2}} + \frac{2}{m_n }\sum_{i=\frac{m_n}{2}}^{m_n} q^{\frac{i}{2}}
= O \Big( \frac{q^{\frac{m_n}{2}}}{m_n } \Big) .
\end{align*}
Similarly,

\begin{align*}
\sum_{i=1}^{m_n} \frac{q^{i}}{i(i+1)}
= O \Big( \frac{q^{m_n}}{(m_n )^2} \Big) .
\end{align*}
Whereas, for the main term we apply the more precise calculations

\begin{align*}
\frac{1}{q-1} \sum_{i=1}^{m_n} \frac{q^{i+1}}{i+1} - \frac{q^i}{i}
= \frac{q}{q-1} \Big( \frac{q^{m_n}}{{m_n}+1} -1 \Big) .
\end{align*}
Hence, (\ref{omega R_n sum telescoped}) gives

\begin{align*}
\omega (R_n )
= \frac{q}{q-1} \frac{q^{m_n}}{{m_n}+1} + O \bigg( \frac{q^{m_n}}{(m_n )^2} \bigg) + r_n . 
\end{align*}

\noindent We also have that

\begin{align*}
\log_q \lvert R_n \rvert
= \degree R_n 
= \bigg( \sum_{i=1}^{m_n } q^i + O\Big( q^{\frac{i}{2}} \Big) \bigg) + (m_n+1) r_n
= \frac{q}{q-1} q^{m_n } + O (q^{\frac{m_n }{2}}) + (m_n +1) r_n ;
\end{align*}
and so

\begin{align*}
\frac{\omega ( R_n )}{\log_q \lvert R_n \rvert}
= \frac{1}{m_n +1} + O \bigg( \frac{1}{(m_n + 1)^2} \bigg)
= \frac{1}{ \log_q \log_q \lvert R_n \rvert } + O \bigg( \frac{1}{ ( \log_q \log_q \lvert R_n \rvert )^2 } \bigg) ,
\end{align*}
where we have used Lemma \ref{Primorial bound}. The proof follows.
\end{proof}

\begin{lemma} \label{Omega function upper bound}
We have that

\begin{align*}
\limsup_{\degree R \rightarrow \infty} \omega (R) \frac{\log_q \log_q \lvert R \rvert }{\log_q \lvert R \rvert}
= 1 .
\end{align*}
\end{lemma}

\begin{proof}
This lemma follows from Lemma \ref{omega function at primorials} if we also prove that

\begin{align} \label{omega(R) leq log_q(R)/log_q log_q(R)}
\omega (R)
\leq \frac{\log_q \lvert R \rvert}{ \log_q \log_q \lvert R \rvert } + O \bigg( \frac{\log_q \lvert R \rvert}{ (\log_q \log_q \lvert R \rvert)^2 } \bigg)
\end{align}
for all $R$ with $\degree R \geq 4$.

\noindent To this end, consider the case where $R$ is a square-free monic with $\omega( R) \geq 3$. Suppose $R$ has $n$ prime divisors. Then, 

\begin{align} \label{Omega function at square-frees upper bound}
\omega (R)
= \omega (R_n ) = \frac{\log_q \lvert R_n \rvert}{\log_q \log_q \lvert R_n \rvert} + O \bigg( \frac{\log_q \lvert R_n \rvert}{( \log_q \log_q \lvert R_n \rvert )^2} \bigg)
\leq \frac{\log_q \lvert R \rvert}{\log_q \log_q \lvert R \rvert} + O \bigg( \frac{\log_q \lvert R \rvert}{( \log_q \log_q \lvert R \rvert )^2} \bigg).
\end{align}
For the second relation we used Lemma \ref{omega function at primorials}. For the third relation we used the fact that for all $q$ the function $\frac{x}{\log_q x}$ is increasing at $x >  e$ and that $\log_q \lvert R_n \rvert \geq \omega (R_n ) \geq 3 > e$. \\

\noindent Now consider the more general case where $R$ is monic and $\omega (R) \geq 3$. We have that

\begin{align*}
\omega (R)
= \omega \big( \rad (R) \big)
\leq &\frac{\log_q \big\lvert \rad (R)  \big\rvert}{\log_q \log_q \big\lvert \rad (R)  \big\rvert} + O \bigg( \frac{\log_q \big\lvert \rad (R)  \big\rvert}{\big( \log_q \log_q \big\lvert \rad (R)  \big\rvert \big)^2} \bigg) \\
\leq &\frac{\log_q \lvert R \rvert}{\log_q \log_q \lvert R \rvert} + O \bigg( \frac{\log_q \lvert R \rvert}{( \log_q \log_q \lvert R \rvert )^2} \bigg) ,
\end{align*}
where the second relation follows from (\ref{Omega function at square-frees upper bound}), and, again, the third relation follows from the fact that for all $q$ the function $\frac{x}{\log_q x}$ is increasing at $x > e$. \\

\noindent Finally consider the case where $R$ is monic with $\degree R \geq 4$ and $\omega (R) \leq 2$. We have that

\begin{align*}
\frac{\log_q \lvert R \rvert}{\log_q \log_q \lvert R \rvert}
= \frac{\degree R}{\log_q \degree R}
\geq \frac{4}{\log_q 4}
\geq 2
= \omega (R).
\end{align*}
So, we have proved (\ref{omega(R) leq log_q(R)/log_q log_q(R)}), and this completes the proof. 
\end{proof}

\begin{remark}
We note that there is an analogous result to Lemma \ref{Omega function upper bound}, in the number field setting:

\begin{align*}
\limsup_{r \rightarrow \infty} \omega (r) \frac{ \log \log r}{\log r} = 1 . 
\end{align*}
\end{remark}

\noindent We now work towards obtaining a result on the growth of the $\phi$ function.

\begin{lemma} \label{(1+ 1/n)^n bounds}
For all integers $n \geq 2$ we have that

\begin{align*}
e^{-1 - \frac{1}{n}} \leq \Big( 1 - \frac{1}{n} \Big)^n \leq e^{-1} .
\end{align*}
\end{lemma}

\begin{proof}
\noindent On one hand, we have that

\begin{align*}
\log \bigg( \Big(1 - \frac{1}{n} \Big)^n \bigg) = n \log \Big(1 - \frac{1}{n} \Big) = -1 - \frac{1}{2n} - \frac{1}{3n^2} - \frac{1}{4n^3} - \ldots \leq - 1 .
\end{align*}
Taking the exponential of both sides gives

\begin{align*}
\Big( 1 - \frac{1}{n} \Big)^n \leq e^{-1} .
\end{align*}

\noindent On the other hand, because $n \geq 2$, we have that

\begin{align*}
\frac{1}{2n} + \frac{1}{3n^2} + \frac{1}{4n^3} \ldots \leq \frac{1}{2n} \sum_{i =0}^{\infty} \frac{1}{n^i} = \frac{1}{2(n-1)} \leq \frac{1}{n} ,
\end{align*}
and so

\begin{align*}
\log \bigg( \Big(1 - \frac{1}{n} \Big)^n \bigg) = -1 - \frac{1}{2n} - \frac{1}{3n^2} - \frac{1}{4n^3} - \frac{1}{5n^4} \geq -1 - \frac{1}{n} .
\end{align*}
Taking exponential of both side gives

\begin{align*}
\Big( 1 - \frac{1}{n} \Big)^n \geq  e^{-1 - \frac{1}{n}} .
\end{align*}
\end{proof}

\begin{proposition} \label{Totient function bounds}
For all $R \in \mathcal{A}$ with $\degree R \geq 1$ we have that

\begin{align} \label{Totient function lower bound}
\phi (R) \geq \frac{e^{- \gamma} \lvert R \rvert}{\log_q \log_q \lvert R \rvert + O(1) } e^{-a q^{-\frac{1}{2}}} ,
\end{align}
and for infinitely many $R \in \mathcal{A}$ we have that

\begin{align} \label{Totient function upper bound}
\phi (R) \leq \frac{e^{- \gamma} \lvert R \rvert}{\log_q \log_q \lvert R \rvert + O(1) } e^{b q^{-\frac{1}{2}}} ,
\end{align}
where $a$ and $b$ are positive constants which are independent of $q$ and $R$.
\end{proposition}

\begin{proof}
For (\ref{Totient function lower bound}) we need only prove the case where $R$ is square-free (and $\degree R \geq 1$). Assuming this, we have that

\begin{align*}
\phi (R) 
= \lvert R \rvert \frac{\phi (R)}{\lvert R \rvert}
= \lvert R \rvert \frac{\phi \big( \rad (R) \big)}{\lvert \rad (R) \rvert}
\geq &\lvert R \rvert \frac{e^{- \gamma}}{\log_q \log_q \lvert \rad (R) \rvert + O(1) } e^{-a q^{-\frac{1}{2}}} \\
\geq &\lvert R \rvert \frac{e^{- \gamma}}{\log_q \log_q \lvert R \rvert + O(1) } e^{-a q^{-\frac{1}{2}}} .
\end{align*}
In fact, we need only prove the inequality for the case where $R$ is a primorial, as the square-free case follows from this. Indeed, if $R$ is square-free with $n \geq 1$ prime factors, then

\begin{align*}
\phi (R) 
\geq \lvert R \rvert \frac{\phi (R_n )}{\lvert R_n \rvert} 
\geq \lvert R \rvert \frac{e^{- \gamma}}{\log_q \log_q \lvert R_n \rvert + O(1) } e^{-a q^{-\frac{1}{2}}} 
\geq \lvert R \rvert \frac{e^{- \gamma}}{\log_q \log_q \lvert R \rvert + O(1) } e^{-a q^{-\frac{1}{2}}} . 
\end{align*}

\noindent So, it suffices to prove (\ref{Totient function lower bound}) for the primorials. It is clear that it is also sufficient to prove (\ref{Totient function upper bound}) for the primorials as there are infinitely many of them, and this is exactly what we will do. We first proceed with (\ref{Totient function lower bound}).

\noindent We have that

\begin{align*}
\frac{\phi (R_n )}{\lvert R_n \rvert}
= \prod_{P \mid R_n} \Big( 1 - \frac{1}{\lvert P \rvert} \Big)
\geq \prod_{\degree P \leq m_n+1} \Big( 1 - \frac{1}{\lvert P \rvert} \Big)
\geq \prod_{n=1}^{m_n+1} \Big( 1 - \frac{1}{q^n} \Big)^{\frac{q^n}{n} + \mathfrak{c} \frac{q^{\frac{n}{2}}}{n}} .
\end{align*}
By Lemma \ref{(1+ 1/n)^n bounds}, we see that 

\begin{align*}
\prod_{n=1}^{m_n+1} \Big( 1 - \frac{1}{q^n} \Big)^{\frac{q^n}{n}} 
= &\prod_{n=1}^{m_n+1} \bigg( \Big( 1 - \frac{1}{q^n} \Big)^{q^n} \bigg)^{\frac{1}{n}}
\geq \prod_{n=1}^{m_n+1} \exp \Big( -\frac{1}{n} -\frac{1}{n q^n } \Big) 
=  \exp \Big( \sum_{n=1}^{m_n+1} -\frac{1}{n} -\frac{1}{n q^n } \Big) \\
\geq &\exp \Big( -\log (m_n+1) - \gamma - \frac{1}{m_n+1} - \sum_{n=1}^{m_n+1} \frac{1}{q^n} \Big) \\
\geq &\frac{e^{-\gamma}}{m_n+1} e^{-\frac{1}{m_n+1}} e^{- 2q^{-1}} 
\geq \frac{e^{-\gamma}}{m_n+1} \frac{1}{1+\frac{3}{m_n+1}} e^{- 2q^{-1}}
= \frac{e^{-\gamma}}{m_n+4} e^{- 2q^{-1}} ,
\end{align*}
where for the second to last relation we used the fact that $e^x \leq 1+3x$ for $x \in [0,1]$. Similarly,

\begin{align} 
\begin{split}\label{(1- 1/q^n)^c... product}
\prod_{n=1}^{m_n+1} \Big( 1 - \frac{1}{q^n} \Big)^{\mathfrak{c} \frac{q^\frac{n}{2}}{n}} 
\geq &\prod_{n=1}^{m_n+1} \exp \Big( -\frac{\mathfrak{c}}{n q^{\frac{n}{2}}} - \frac{\mathfrak{c}}{n q^{\frac{3n}{2}}} \Big) 
\geq \exp \Big( \sum_{n=1}^{m_n} -\frac{2 \mathfrak{c} }{q^{\frac{n}{2}}} \Big) \\
\geq &\exp \Big( -\frac{2 \mathfrak{c} }{q^{\frac{1}{2}}} \frac{1}{1-q^{- \frac{1}{2}}} \Big) 
\geq e^{-7 \mathfrak{c} q^{- \frac{1}{2}}} .
\end{split}
\end{align}
Hence, we deduce that

\begin{align*}
\frac{\phi (R_n )}{\lvert R_n \rvert}
\geq \frac{e^{-\gamma}}{m_n+4} e^{- a q^{-\frac{1}{2}}},
\end{align*}
where $a = 7 \mathfrak{c} + 2$. Finally, we apply Lemma \ref{Primorial bound} and rearrange to see that

\begin{align*}
\phi (R_n )
\geq \frac{e^{-\gamma} \lvert R_n \rvert}{\log_q \log_q \lvert  R_n \rvert + O(1)} e^{- a q^{-\frac{1}{2}}} . 
\end{align*}

\noindent For (\ref{Totient function upper bound}), we proceed in a similar fashion. We have that

\begin{align*}
\frac{\phi (R_n )}{\lvert R_n \rvert}
= \prod_{P \mid R_n} \Big( 1 - \frac{1}{\lvert P \rvert} \Big)
\leq \prod_{\degree P \leq m_n } \Big( 1 - \frac{1}{\lvert P \rvert} \Big)
\leq \prod_{n=1}^{m_n } \Big( 1 - \frac{1}{q^n} \Big)^{\frac{q^n}{n} - \mathfrak{c} \frac{q^{\frac{n}{2}}}{n}} .
\end{align*}
By Lemma \ref{(1+ 1/n)^n bounds} again, we see that

\begin{align*}
\prod_{n=1}^{m_n} \Big( 1 - \frac{1}{q^n} \Big)^{\frac{q^n}{n}} 
\leq \prod_{n=1}^{m_n} \exp \Big( -\frac{1}{n} \Big) 
\leq \exp \Big( -\log (m_n) - \gamma \Big) 
= \frac{e^{-\gamma}}{m_n}.
\end{align*}
Also, by (\ref{(1- 1/q^n)^c... product}), we see that

\begin{align*}
\prod_{n=1}^{m_n +1} \Big( 1 - \frac{1}{q^n} \Big)^{-\mathfrak{c} \frac{q^\frac{n}{2}}{n}}
\leq e^{7 \mathfrak{c} q^{- \frac{1}{2}}} .
\end{align*}
Hence,

\begin{align*}
\frac{\phi (R_n )}{\lvert R_n \rvert}
\leq \frac{e^{-\gamma}}{m_n } e^{b q^{-\frac{1}{2}}} ,
\end{align*}
where $b = 7 \mathfrak{c}$. Finally, we apply Lemma \ref{Primorial bound} and rearrange to see that

\begin{align*}
\phi (R_n )
\leq \frac{e^{-\gamma} \lvert R_n \rvert}{\log_q \log_q \lvert  R_n \rvert + O(1) } e^{b q^{-\frac{1}{2}}} .
\end{align*}
\end{proof}

\begin{proposition} \label{Number of primitive characters bound}
For all $R \in \mathcal{A}$ with $\degree R \geq 1$ we have that

\begin{align*}
\phi^* (R)
\geq \frac{e^{- \gamma} \phi (R) }{\log_q \log_q \lvert R \rvert + O(1) } e^{-c q^{-\frac{1}{2}}} ,
\end{align*}
and for infinitely many $R \in \mathcal{A}$ we have that

\begin{align*}
\phi^* (R)
\leq \frac{e^{- \gamma} \phi (R) }{\log_q \log_q \lvert R \rvert + O(1) } e^{d q^{-\frac{1}{2}}} ,
\end{align*}
where $c$ and $d$ are positive constants which are independent of $q$ and $R$.
\end{proposition}

\begin{proof} 
From Corollary \ref{Number of primitive characters}, we have that

\begin{align*}
\frac{\phi^* (R)}{\phi (R)}
= &\sum_{EF=R } \mu (E) \frac{\phi (F)}{\phi (R)}
= \sum_{EF=R} \mu (E) \bigg( \prod_{\substack{P \mid E \\ P^2 \nmid R}} \frac{1}{\lvert P \rvert -1} \bigg) \bigg( \prod_{\substack{P \mid E \\ P^2 \mid R}} \frac{1}{\lvert P \rvert} \bigg) \\
= &\bigg( \prod_{\substack{P \mid R \\ P^2 \nmid R}} 1 - \frac{1}{\lvert P \rvert - 1} \bigg) \bigg( \prod_{\substack{P \mid R \\ P^2 \nmid R}} 1 - \frac{1}{\lvert P \rvert } \bigg)
\end{align*}
Now, by similar means as in Proposition \ref{Totient function bounds}, we can obtain that

\begin{align*}
\frac{\phi^* (R)}{\phi (R)}
\geq \prod_{P \mid R} 1 - \frac{1}{\lvert P \rvert - 1}
\geq \frac{e^{-\gamma}}{\log_q \log_q \lvert  R \rvert + O(1) } e^{-c q^{-\frac{1}{2}}}
\end{align*}
for all $R \in \mathcal{A}$ with $\degree R \geq 1$, and that

\begin{align*}
\frac{\phi^* (R)}{\phi (R)}
\leq \prod_{P \mid R} 1 - \frac{1}{\lvert P \rvert }
\leq \frac{e^{-\gamma}}{\log_q \log_q \lvert  R \rvert + O(1) } e^{d q^{-\frac{1}{2}}}
\end{align*}
for infinitely many $R \in \mathcal{A}$, where $c$ and $d$ are independent of $q$ and $R$.
\end{proof}

\begin{remark} \label{remark, asymptotic equivalences of powers of phi(R)/R}
By similar, but simpler, means as in the proof of Propositions \ref{Totient function bounds} and \ref{Number of primitive characters bound}, we can show that

\begin{align*}
\prod_{P \mid R} \frac{1}{1 + 2 \lvert P \rvert^{-1}}
\asymp \prod_{P \mid R} \bigg( \frac{1}{1 + \lvert P \rvert^{-1}} \bigg)^2
\end{align*}
and

\begin{align*}
\prod_{P \mid R} \frac{1}{1 + \lvert P \rvert^{-1}}
\asymp \prod_{P \mid R} 1 - \lvert P \rvert^{-1} .
\end{align*}
\end{remark}

\noindent Finally, we prove the following four lemmas.

\begin{lemma}
For all non-negative integers $k$ we have that

\begin{align*}
\bigg( \prod_{P \mid R} 1- \lvert P \rvert^{-1} \bigg)^k \omega (R)
\gg_k 1 .
\end{align*}
Also,

\begin{align*}
\bigg( \prod_{P \mid R} 1- \lvert P \rvert^{-1} \bigg) \omega (R)
\gg \log \omega (R) .
\end{align*}
\end{lemma}

\begin{proof}
Again, it suffices to prove both results for the primorials. From previous results in this section, we can see that

\begin{align*}
\bigg( \prod_{P \mid R} 1- \lvert P \rvert^{-1} \bigg)^k \omega (R)
\gg \frac{q^{m_n}}{(m_n)^{k+1}}
\gg_k 1.
\end{align*}
For the second result, we note that it is sufficient to show that 

\begin{align*}
\bigg( \prod_{P \mid R} 1- \lvert P \rvert^{-1} \bigg) \omega (R)
\gg \sqrt{\omega (R)} ,
\end{align*}
which is equivalent to

\begin{align*}
\bigg( \prod_{P \mid R} 1- \lvert P \rvert^{-1} \bigg)^2 \omega (R)
\gg 1 ,
\end{align*}
which we have proved.
\end{proof}

\begin{lemma} \label{Sum over P divides R of deg P / (P-1)}
Let $R \in \mathcal{M}$. We have that

\begin{align*}
\sum_{P \mid R} \frac{\degree P}{\lvert P \rvert - 1} = O \big( \log \omega (R) \big).
\end{align*}
\end{lemma}

\begin{proof}
As we have done in Lemma \ref{Omega function upper bound} and Proposition \ref{Totient function bounds}, we can reduce the proof to the case where $R$ is a primorial. That is, we need only prove that

\begin{align*}
\sum_{P \mid R_n } \frac{\degree P}{\lvert P \rvert - 1} = O ( \log n ) . 
\end{align*}

\noindent To this end, we recall that $\# \mathcal{P}_m = \frac{q^m}{m} + O \Big( \frac{q^{\frac{m}{2}}}{m} \Big)$. From this, we can deduce that there is a constant $c \in (0,1)$, which is independent of $q$, such that $\# \mathcal{P}_{\leq m} \geq c q^{\frac{m}{2}}$ for all positive integers $m$. In particular, if we take $m = \lceil \frac{2}{\log q} \log \frac{n}{c} \rceil$, then $\# \mathcal{P}_{\leq m} \geq n$ . So,

\begin{align*}
\sum_{P \mid R_n} \frac{\degree P }{\lvert P \rvert - 1}
\leq \sum_{i = 1}^{\frac{2}{ \log q} \log \frac{n}{c} + 1} \sum_{\substack{P \text{ prime} \\ \degree P = i}} \frac{\degree P }{\lvert P \rvert - 1}
\ll \sum_{i = 1}^{\frac{2}{ \log q} \log \frac{n}{c} + 1} \frac{i}{q^i - 1} \frac{q^i}{i} 
\ll \log n ,
\end{align*}
where the second relation follows from the prime polynomial theorem.
\end{proof}

\begin{lemma} \label{Lemma, sum of 1/phi(N) over monic N with degree leq x}
We have that

\begin{align*}
\sum_{\substack{N \in \mathcal{M} \\ \degree N \leq x}} \frac{1}{\phi (N)}
\ll x .
\end{align*}
\end{lemma}

\begin{proof}
For all $N \in \mathcal{A}$ we have that

\begin{align*}
\sum_{E \mid N} \frac{\mu (E)^2}{\phi (E)}
= \prod_{E \mid N} 1 + \frac{1}{\lvert P \rvert -1}
= \prod_{E \mid N} \frac{1}{1-\lvert P \rvert^{-1}}
= \frac{\lvert N \rvert}{\phi (N)} .
\end{align*}
So,

\begin{align*}
\sum_{\substack{N \in \mathcal{M} \\ \degree N \leq x}} \frac{1}{\phi (N)}
= &\sum_{\substack{N \in \mathcal{M} \\ \degree N \leq x}} \frac{1}{\lvert N \rvert} \frac{\lvert N \rvert}{\phi (N)}
= \sum_{\substack{N \in \mathcal{M} \\ \degree N \leq x}} \frac{1}{\lvert N \rvert} \sum_{E \mid N} \frac{\mu (E)^2}{\phi (E)}
= \sum_{\substack{E \in \mathcal{M} \\ \degree E \leq x}} \frac{\mu (E)^2}{\phi (E)} \sum_{\substack{ N \in \mathcal{M} \\ \degree N \leq x \\ E \mid N}} \frac{1}{\lvert N \rvert} \\
\leq &x \sum_{\substack{E \in \mathcal{M} \\ \degree E \leq x}} \frac{\mu (E)^2}{\phi (E) \lvert E \rvert}
\ll x ,
\end{align*}
where the last relation uses the fact that $\phi (N) \gg \frac{\lvert N \rvert}{\log_q \log_q \lvert R \rvert}$ (see Proposition \ref{Totient function bounds}).
\end{proof}

\begin{lemma} \label{Lemma, sum of mu(N)/phi(N) over monic N with degree leq x, lower bound}
We have that

\begin{align*}
\sum_{\substack{N \in \mathcal{M} \\ \degree N \leq x}} \frac{\mu^2 (N)}{\phi (N)}
\geq x .
\end{align*}
\end{lemma}

\begin{proof}
For square-free $N$ we have that

\begin{align*}
\frac{1}{\phi (N)}
= \frac{1}{\lvert N \rvert} \prod_{P \mid N} \Big( 1 - \vert P \rvert ^{-1} \Big)^{-1}
= \frac{1}{\lvert N \rvert} \prod_{P \mid N} \bigg( 1 + \frac{1}{\vert P \rvert} + \frac{1}{\vert P \rvert^2} + \ldots \bigg)
= \sum_{\substack{M \in \mathcal{M} \\ \rad (M) =N}} \frac{1}{\lvert M \rvert} ,
\end{align*}
and so 

\begin{align*}
\sum_{\substack{N \in \mathcal{M} \\ \degree N \leq z}} \frac{\mu (N)^2}{\phi (N)}
= \sum_{\substack{N \in \mathcal{M} \\ N \text{ is square-free} \\ \degree N \leq z}} \sum_{\substack{M \in \mathcal{M} \\ \rad(M) =F}} \frac{1}{\lvert M \rvert}
\geq \sum_{\substack{M \in \mathcal{M} \\ \degree M \leq z}} \frac{1}{\lvert M \rvert}
= z.
\end{align*}
\end{proof}


\section{The Second Moment}

We now proceed to prove Theorem \ref{Second primitive moment function fields statement}.

\begin{proof}[Proof of Theorem \ref{Second primitive moment function fields statement}]
We have that

\begin{align*}
\sumstar_{\chi (\modulus R)} \Big\lvert L \Big( \frac{1}{2}, \chi \Big) \Big\rvert^2 
= &\sumstar_{\chi (\modulus R)} \sum_{\substack{A,B \in \mathcal{M} \\ \degree A , \degree B < \degree R}} \frac{\chi (A) \conj\chi (B)}{\lvert AB \rvert^{\frac{1}{2}}} \\
=&\sum_{EF=R} \mu (E) \phi (F) \sum_{\substack{A,B \in \mathcal{M}  \\ \degree A , \degree B < \degree R \\ (AB,R)=1 \\ A \equiv B (\modulus F)}} \frac{1}{\lvert AB \rvert^{\frac{1}{2}}} \\
=&\sum_{EF=R} \mu (E) \phi (F) \sum_{\substack{A \in \mathcal{M}  \\ \degree A < \degree R \\ (A,R)=1}} \frac{1}{\lvert A \rvert^{\frac{1}{2}}} \sum_{\substack{B \in \mathcal{M}  \\ \degree B < \degree R \\ B \equiv A (\modulus F)}} \frac{1}{\lvert B \rvert^{\frac{1}{2}}} 
\end{align*}
The second equality follows from Proposition \ref{Primitive chracter sum, mobius inversion}. For the last equality we note that if $R$ is square-full, $EF=R$, and $\mu (E)\neq 0$, then $F$ and $R$ have the same prime factors. Therefore, if we also have that $(A,R)=1$ and $B \equiv A (\modulus F)$, then $(B,R)=1$. Continuing,

\begin{align}
\begin{split} \label{Second primitive sum, no coprime conditions}
&\sumstar_{\chi \modulus R} \Big\lvert L \Big( \frac{1}{2}, \chi \Big) \Big\rvert^2 \\
=&\sum_{EF=R} \mu (E) \phi (F) \sum_{\substack{A \in \mathcal{M}  \\ \degree A < \degree R}} \frac{1}{\lvert A \rvert^{\frac{1}{2}}} \sum_{G \mid (A,R)} \mu (G) \sum_{\substack{B \in \mathcal{M}  \\ \degree B < \degree R \\ B \equiv A (\modulus F)}} \frac{1}{\lvert B \rvert^{\frac{1}{2}}} \\
=&\sum_{EF=R} \mu (E) \phi (F) \sum_{G \mid R} \mu (G) \sum_{\substack{A \in \mathcal{M}  \\ \degree A < \degree R \\ G \mid A}} \frac{1}{\lvert A \rvert^{\frac{1}{2}}} \sum_{\substack{B \in \mathcal{M}  \\ \degree B < \degree R \\ B \equiv A (\modulus F)}} \frac{1}{\lvert B \rvert^{\frac{1}{2}}} \\
=&\sum_{EF=R} \mu (E) \phi (F) \sum_{G \mid R} \mu (G) \sum_{\substack{K \in \mathcal{A} \\ \degree K < \degree F - \degree G \\ \text{or} \\ K=0}} \bigg( \sum_{\substack{A \in \mathcal{M}  \\ \degree A < \degree R \\ A \equiv GK (\modulus F)}} \frac{1}{\lvert A \rvert^{\frac{1}{2}}} \bigg) \bigg( \sum_{\substack{B \in \mathcal{M}  \\ \degree B < \degree R \\ B \equiv GK (\modulus F)}} \frac{1}{\lvert B \rvert^{\frac{1}{2}}} \bigg).
\end{split}
\end{align}
The last equality follows from the fact that $F$ and $R$ have the same prime factors, and so, if $\mu (G) \neq 0$, then $G \mid F$. Hence, if $G \mid A$, then $A \equiv GK (\modulus F)$ for some $K \in \mathcal{A}$ with $\degree K < \degree F - \degree G$ or $k=0$. \\

\noindent Now, we note that if $K \in \mathcal{A} \backslash \mathcal{M}$, then

\begin{align*}
\sum_{\substack{A \in \mathcal{M}  \\ \degree A < \degree R \\ A \equiv GK (\modulus F)}} \frac{1}{\lvert A \rvert^{\frac{1}{2}}} 
= \sum_{\substack{L \in \mathcal{M} \\ \degree L < \degree R -\degree F}} \frac{1}{\lvert LF+GK \rvert^{\frac{1}{2}}} 
= \sum_{\substack{L \in \mathcal{M} \\ \degree L < \degree R -\degree F}} \frac{1}{\lvert LF \rvert^{\frac{1}{2}}}
= \frac{1}{q^{\frac{1}{2}} -1} \bigg( \frac{\lvert R \rvert^{\frac{1}{2}}}{\lvert F \rvert} - \frac{1}{\lvert F \rvert^{\frac{1}{2}}} \bigg) .
\end{align*}
Whereas, if $K \in \mathcal{M}$, then

\begin{align*}
\sum_{\substack{A \in \mathcal{M}  \\ \degree A < \degree R \\ A \equiv GK (\modulus F)}} \frac{1}{\lvert A \rvert^{\frac{1}{2}}} 
= \frac{1}{\lvert G K \rvert^{\frac{1}{2}}} + \sum_{\substack{L \in \mathcal{M} \\ \degree L < \degree R -\degree F}} \frac{1}{\lvert LF+GK \rvert^{\frac{1}{2}}}
= \frac{1}{\lvert G K \rvert^{\frac{1}{2}}} + \frac{1}{q^{\frac{1}{2}} -1} \bigg( \frac{\lvert R \rvert^{\frac{1}{2}}}{\lvert F \rvert} - \frac{1}{\lvert F \rvert^{\frac{1}{2}}} \bigg) .
\end{align*}
Hence,

\begin{align*}
&\sum_{\substack{K \in \mathcal{A} \\ \degree K < \degree F - \degree G \\ \text{or} \\ K=0}} \bigg( \sum_{\substack{A \in \mathcal{M}  \\ \degree A < \degree R \\ A \equiv GK (\modulus F)}} \frac{1}{\lvert A \rvert^{\frac{1}{2}}} \bigg) \bigg( \sum_{\substack{B \in \mathcal{M}  \\ \degree B < \degree R \\ B \equiv GK (\modulus F)}} \frac{1}{\lvert B \rvert^{\frac{1}{2}}} \bigg) \\
= &\frac{1}{\big( q^{\frac{1}{2}} -1 \big)^2} \bigg( \frac{\lvert R \rvert^{\frac{1}{2}}}{\lvert F \rvert} - \frac{1}{\lvert F \rvert^{\frac{1}{2}}} \bigg)^2 \sum_{\substack{K \in \mathcal{A} \\ \degree K < \degree F - \degree G \\ \text{or} \\ K=0}} 1 \\
&+ \frac{2}{q^{\frac{1}{2}} -1} \bigg( \frac{\lvert R \rvert^{\frac{1}{2}}}{\lvert F \rvert} - \frac{1}{\lvert F \rvert^{\frac{1}{2}}} \bigg) \frac{1}{\lvert G \rvert^{\frac{1}{2}}} \sum_{\substack{K \in \mathcal{M} \\ \degree K < \degree F - \degree G}} \frac{1}{\lvert K \rvert^{\frac{1}{2}}} \\
&+\frac{1}{\lvert G \rvert} \sum_{\substack{K \in \mathcal{M} \\ \degree K < \degree F - \degree G}} \frac{1}{\lvert K \rvert} \\
= & \frac{1}{\big( q^{\frac{1}{2}} -1 \big)^2} \bigg( \frac{\lvert R \rvert}{\lvert FG \rvert} - 2 \frac{\lvert R \rvert^{\frac{1}{2}}}{\lvert F \rvert \lvert G \rvert^{\frac{1}{2}}} - \frac{1}{\lvert G \rvert} + \frac{2}{\lvert F \rvert^{\frac{1}{2}} \lvert G \rvert^{\frac{1}{2}}} \bigg) + \frac{\degree F}{\lvert G \rvert} - \frac{\degree G}{\lvert G \rvert} .
\end{align*}
By applying this to (\ref{Second primitive sum, no coprime conditions}), and using (\ref{Sum of mu(E)/E^s over E mid R}) to (\ref{Sum of mu(E)phi(F) deg F/F^s over EF=R}), we see that

\begin{align*}
&\sumstar_{\chi \modulus R} \Big\lvert L \Big( \frac{1}{2}, \chi \Big) \Big\rvert^2 \\
= &\frac{\phi (R)^3}{\lvert R \rvert^2} \degree R
+ 2 \frac{\phi (R)^3}{\lvert R \rvert^2} \sum_{P \mid R} \frac{\degree P}{\lvert P \rvert -1}
+ \frac{1}{\big( q^{\frac{1}{2}} -1 \big)^2} \bigg( - \frac{\phi (R)^3}{\lvert R \rvert^2} + 2 \frac{\phi (R)}{\lvert R \rvert^{\frac{1}{2}}} \prod_{P \mid R} \Big(1-\frac{1}{\lvert P \rvert^{\frac{1}{2}}} \Big)^2 \bigg) .
\end{align*}
\end{proof}


\section{The Brun-Titchmarsh Theorem for the Divisor Function in $\mathbb{F}_{q}[T]$}

\noindent In this section we prove a specific case of the function field analogue of the generalised Brun-Titchmarsh theorem. The generalised Brun-Titchmarsh theorem in the number field setting was proved by Shiu \cite{BrunTitchTheoMultFunc_Shiu}. It gives upper bounds for sums over short intervals and arithmetic progressions of certain multiplicative functions. We will look at the case where the multiplicative function is the divisor function in the function field setting. \\

\noindent The main results in this section are the following two theorems.

\begin{theorem} \label{Brun-Titschmarsh theorem, divisor function case in function field}
Suppose $\alpha , \beta$ are fixed and satisfy $0 < \alpha < \frac{1}{2}$ and $0 < \beta < \frac{1}{2}$. Let $X \in \mathcal{M}$ and $y$ be a positive integer satisfying $\beta \degree X < y \leq \degree X$. Also, let $A \in \mathcal{A}$ and $G \in \mathcal{M}$ satisfy $(A,G)=1$ and $\degree G < (1-\alpha ) y$. Then, we have that

\begin{align*}
\sum_{\substack{N \in \mathcal{M} \\ \degree (N-X) < y \\ N \equiv A (\modulus G)}} d(N)
\ll_{\alpha , \beta} \frac{q^y \degree X}{\phi (G)} .
\end{align*}
\end{theorem}

\noindent Intuitively, this seems to be a good upper bound. Indeed, all $N$ in the sum are of degree equal to $\degree X$, and so this suggests that the average value that the divisor function will take is $\degree X$. Also, there are $q^y \frac{1}{\vert G \rvert} \approx q^y \frac{1}{\phi (G)}$ possible values for $N$ in the sum.

\begin{theorem} \label{Brun-Titschmarsh theorem, divisor function case in function field, y equality version}
Suppose $\alpha , \beta$ are fixed and satisfy $0 < \alpha < \frac{1}{2}$ and $0 < \beta < \frac{1}{2}$. Let $X \in \mathcal{M}$ and $y$ be a positive integer satisfying $\beta \degree X < y \leq \degree X$. Also, let $A \in \mathcal{A}$ and $G \in \mathcal{M}$ satisfy $(A,G)=1$ and $\degree G < (1-\alpha ) y$. Finally, let $a \in (\mathbb{F}_q)^*$. Then, we have that

\begin{align*}
\sum_{\substack{N \in \mathcal{A} \\ \degree (N-X) = y \\ (N-X) \in \ \mathcal{M} \\ N \equiv A (\modulus G)}} d(N)
\ll_{\alpha , \beta} \frac{q^y \degree X}{\phi (G)} .
\end{align*}
\end{theorem}

\noindent Our proofs of these two theorems are based on Shiu's proof of the more general theorem in the number field setting \cite{BrunTitchTheoMultFunc_Shiu}. We begin by proving preliminary results that are needed for the main part of the proofs.\\

\noindent The Selberg sieve gives us the following result. A proof is given in \cite{SieveMethPolyRingFinField_Webb_1983}.

\begin{theorem} \label{Theorem, Selberg sieve}
Let $\mathcal{S} \subseteq \mathcal{A}$ be a finite subset. For a prime $P \in \mathcal{A}$ we define $\mathcal{S}_P = \mathcal{S} \cap P \mathcal{A} = \{ A \in \mathcal{S} : P \mid A \}$. We extend this to all square-free $D \in \mathcal{A}$: $\mathcal{S}_D = \mathcal{S} \cap D \mathcal{A}$.

Furthermore, let $\mathcal{Q} \subseteq \mathcal{A}$ be a subset of prime elements. For positive integers $z$ we define $\mathcal{Q}_z = \prod_{\substack{P \in \mathcal{Q} \\ \degree P \leq z}} P$. We also define $\mathcal{S}_{\mathcal{Q},z} := \mathcal{S} \backslash \cup_{P \mid \mathcal{Q}_z} \mathcal{S}_P$.

Suppose there exists a completely multiplicative function $\omega$ and a function $r$ such that for each $D \mid \mathcal{Q}_z$ we have $\# \mathcal{S}_D = \frac{\omega (D)}{\lvert D \rvert} \# \mathcal{S}_D + r(D)$ and $0 < \omega (D) < \lvert D \rvert$. Also, define $\psi$ multiplicatively by $\psi (P) = \frac{\lvert P \rvert}{\omega (P)} - 1$ and $\psi (P^e ) = 0$ for $e \geq 2$. \\

We then have that

\begin{align*}
\# \mathcal{S}_{\mathcal{Q},z}
:= &\# \Big( \mathcal{S} \backslash \cup_{P \mid \mathcal{Q}_z} \mathcal{S}_P \Big)
= \# \{ A \in \mathcal{S} : (P \mid A \text{ and } P \in \mathcal{Q}) \Rightarrow \degree P > z \} \\
\leq &\frac{\# \mathcal{S}}{\sum_{\substack{F \in \mathcal{M} \\ \degree F \leq z \\ F \mid \mathcal{Q}_z}} \frac{\mu^2 (F)}{\psi (F)}} + \sum_{\substack{D,E \in \mathcal{M} \\ \degree D, \degree E \leq z \\ D,E \mid \mathcal{Q}_z}} \lvert r([D,E]) \rvert
\end{align*}
\end{theorem}

\begin{corollary} \label{Selberg sieve application, sum of 1 over short interval, arithmetic progression, large prime divisor}
Let $X \in \mathcal{M}$ and $y$ be a positive integer satisfying $y \leq \degree X$. Also, let $K \in \mathcal{M}$ and $A \in \mathcal{A}$ satisfy $(A,K)=1$. Finally, let $z$ be a positive integer such that $\degree K + z \leq y$. Then,

\begin{align*}
\sum_{\substack{N \in \mathcal{M} \\ \degree (N-X) < y \\ N \equiv A (\modulus K) \\  p_{-} (N) > z}} 1
\leq \frac{q^y}{\phi (K) z} + O \Big( q^{2z} \Big) .
\end{align*}
\end{corollary}

\begin{proof}
Let us define

\begin{align*}
\mathcal{S}
= \{ N \in \mathcal{M} : \degree (N-X) < y , N \equiv A (\modulus K) \}
\end{align*}
and

\begin{align*}
\mathcal{Q}
= \{ P \text{ prime} : \degree P \leq z , P \nmid K \} .
\end{align*}
Then, we have that

\begin{align*}
\# \mathcal{S}_{\mathcal{Q} , z} 
= \sum_{\substack{N \in \mathcal{M} \\ \degree (N-X) < y \\ N \equiv A (\modulus K) \\  p_{-} (N) > z}} 1 ,
\end{align*}
which is what we want to bound. \\

\noindent For $D \mid \mathcal{Q}_z$ with $\degree D \leq z$ we have that

\begin{align*}
\# \mathcal{S}_D
= \# \{ N \in \mathcal{M} : \degree (N-X) < y , N \equiv A (\modulus K) , N \equiv 0 (\modulus D) \}
= \frac{q^y}{\lvert KD \rvert} .
\end{align*}
This follows from the fact that $K$ and $D$ are coprime and that $\degree K + \degree D \leq \degree K + z \leq y$. For $D \mid \mathcal{Q}_z$ with $\degree D > z$ we have that

\begin{align*}
\# \mathcal{S}_D
= \frac{q^y}{\lvert KD \rvert} + c_D
\end{align*}
where $\lvert c_D \rvert \leq 1$. Therefore, we have $\omega (D) =1$ and $ \lvert r(D) \rvert \leq 1$ for all $D \mid \mathcal{Q}_z$. We also have that $\psi (D) = \phi (D)$ for square-free $D$. \\

\noindent We can now see that

\begin{align*}
\sum_{\substack{F \in \mathcal{M} \\ \degree F \leq z \\ F \mid \mathcal{Q}_z}} \frac{\mu^2 (F)}{\psi (F)}
= \sum_{\substack{F \in \mathcal{M} \\ \degree F \leq z \\ (F,K)=1}} \frac{\mu (F)^2}{\phi (F)} ,
\end{align*}
and we have that

\begin{align*} 
\sum_{\substack{F \in \mathcal{M} \\ \degree F \leq z \\ (F,K)=1}} \frac{\mu (F)^2}{\phi (F)} \sum_{E \mid K} \frac{\mu (E)^2}{\phi (E)}
\geq \sum_{\substack{F \in \mathcal{M} \\ \degree F \leq z}} \frac{\mu (F)^2}{\phi (F)} .
\end{align*}
To this we apply Lemma \ref{Lemma, sum of mu(N)/phi(N) over monic N with degree leq x, lower bound} and the fact that

\begin{align*}
\sum_{E \mid K} \frac{\mu (E)^2}{\phi (E)}
= \prod_{P \mid K} 1 + \frac{1}{\lvert P \rvert -1}
= \prod_{P \mid K} \Big( 1 - \vert P \rvert ^{-1} \Big)^{-1}
= \frac{\lvert K \rvert}{\phi (K)} ,
\end{align*}
to obtain

\begin{align*}
\sum_{\substack{F \in \mathcal{M} \\ \degree F \leq z \\ (F,K)=1}} \frac{\mu (F)^2}{\phi (F)} \geq \frac{\phi (K)}{\lvert K \rvert} z.
\end{align*}

\noindent Also, we have that

\begin{align*}
\sum_{\substack{D,E \in \mathcal{M} \\ \degree D, \degree E \leq z \\ D,E \mid \mathcal{P}_z}} \lvert r([D,E]) \rvert
\leq \bigg( \sum_{\substack{D \in \mathcal{M} \\ \degree D \leq z}} 1 \bigg)^2
\ll q^{2z} .
\end{align*}

\noindent The result now follows by applying Theorem \ref{Theorem, Selberg sieve}.
\end{proof}

\noindent The proof of the following corollary is almost identical to the proof above.

\begin{corollary} \label{Selberg sieve application, sum of 1 over short interval, arithmetic progression, large prime divisor, y equality case}
Let $X \in \mathcal{M}$ and $y$ be a positive integer satisfying $y \leq \degree X$. Also, let $K \in \mathcal{M}$ and $A \in \mathcal{A}$ satisfy $(A,K)=1$. Finally, let $z$ be a positive integer such that $\degree K + z \leq y$, and let $a \in (\mathbb{F}_q )^*$. Then,

\begin{align*}
\sum_{\substack{N \in \mathcal{A} \\ \degree (N-X) = y \\ (N-X) \in a \mathcal{M} \\ N \equiv A (\modulus K) \\  p_{-} (N) > z}} 1
\leq \frac{q^y}{\phi (K) z} + O \Big( q^{2z} \Big) .
\end{align*}
\end{corollary}

\begin{lemma} \label{Sum over primes of 1/degree P}
Let $w$ be a positive integer. We have that

\begin{align*}
\sum_{\degree P \leq w} \frac{1}{\degree P} \ll \frac{q^w}{w^2} ,
\end{align*}
where the implied constant can be taken to be independent of $q$ and will be denoted by $\mathfrak{d}$.
\end{lemma}

\begin{proof}
By using the prime polynomial theorem, we have that

\begin{align*}
\sum_{\degree P \leq w} \frac{1}{\degree P}
= \sum_{n=1}^{w} \frac{1}{n} \bigg( \frac{q^n}{n} + O \Big( \frac{q^{\frac{n}{2}}}{n} \Big) \bigg) .
\end{align*}
Now,

\begin{align*}
\sum_{n=1}^{w} \frac{q^n}{n^2}
= &\frac{1}{q-1} \bigg( \sum_{n=1}^{w} \frac{q^{n+1}}{n^2} - \frac{q^n}{n^2} \bigg) \\
= &\frac{1}{q-1} \bigg( \sum_{n=1}^{w} \frac{q^{n+1}}{(n+1)^2} - \frac{q^n}{n^2} \bigg) + \frac{1}{q-1} \bigg( \sum_{n=1}^{w} \frac{1}{n} \Big(2 + \frac{1}{n} \Big) \frac{q^{n+1}}{(n+1)^2} \bigg) \\
\leq &\frac{q}{q-1} \bigg( \frac{q^{w}}{(w+1)^2} - 1 \bigg) + \frac{3q}{q-1} \bigg( \sum_{n=1}^{\frac{w}{2}} \frac{q^{n}}{n^3} + \sum_{n=\frac{w}{2}}^{w} \frac{q^{n}}{n^3} \bigg) \\
\leq &\frac{q}{q-1} \bigg( \frac{q^{w}}{(w+1)^2} - 1 \bigg) + \frac{3q}{q-1} \bigg( \sum_{n=1}^{\frac{w}{2}} q^{n} + \frac{8}{w^3} \sum_{n=1}^{w} q^{n} \bigg) \\
\ll &\frac{q^w}{w^2} ,
\end{align*}
Also, we can easily see that

\begin{align*}
\sum_{n=1}^{w} O \Big( \frac{q^{\frac{n}{2}}}{n^2} \Big)
\ll q^{\frac{w}{2}}
\ll \frac{q^w}{w^2}
\end{align*}
The result now follows.
\end{proof}

\begin{lemma} \label{Sum of 1 over N monic, degree N leq z, largest divisor less than w}
Let $0 < \alpha , \beta < \frac{1}{2}$, let $z > q$ be an integer, and let

\begin{align*}
w(z) := \log_q z .
\end{align*}
Then,

\begin{align*}
\sum_{\substack{N \in \mathcal{M} \\ \degree N \leq z \\  p_{+} (N) \leq w(z)}} 1
\leq q^{\sqrt{\mathfrak{d}} \frac{z}{(\log z)}}
\end{align*}
as $z \rightarrow \infty$, where $\mathfrak{d}$ is as in Lemma \ref{Sum over primes of 1/degree P}. In particular, this implies that

\begin{align*}
\sum_{\substack{N \in \mathcal{M} \\ \degree N \leq z \\  p_{+} (N) \leq w(z)}} 1
\ll q^{\frac{z}{4}}
\end{align*}
(under the condition that $z >q$).
\end{lemma}

\begin{proof}
Let $\delta > 0$. We will optimise on the value of $\delta$ later. We have that

\begin{align*}
\sum_{\substack{N \in \mathcal{M} \\ \degree N \leq z \\  p_{+} (N) \leq w(z)}} 1
\leq &q^{\delta z} \sum_{\substack{N \in \mathcal{M} \\ \degree N \leq z \\  p_{+} (N) \leq w(z)}} \lvert N \rvert^{-\delta}
\leq q^{\delta z} \sum_{\substack{N \in \mathcal{M} \\  p_{+} (N) \leq w(z)}} \lvert N \rvert^{-\delta}
= q^{\delta z} \prod_{\degree P \leq w(z)} \bigg( 1 + \vert P \rvert^{-\delta} + \vert P \rvert^{-2 \delta} + \ldots \bigg) \\
= &q^{\delta z} \prod_{\degree P \leq w(z)} \bigg( 1 + \frac{1}{\lvert P \rvert^{\delta} -1} \bigg)
\leq q^{\delta z} \prod_{\degree P \leq w(z)} \bigg( \exp \Big( \frac{1}{\lvert P \rvert^{\delta} -1}\Big) \bigg) \\
\leq &q^{\delta z} \prod_{\degree P \leq w(z)} \bigg( \exp \Big( \frac{1}{\delta \log \lvert P \rvert}\Big) \bigg) ,
\end{align*}
where the last two relations follow from the Taylor series for the exponential function. Continuing,

\begin{align*}
\sum_{\substack{N \in \mathcal{M} \\ \degree N \leq z \\  p_{+} (N) \leq w(z)}} 1
\leq \exp \bigg( (\delta \log q) z + \frac{1}{\delta \log q} \sum_{\degree P \leq w(z)} \frac{1}{\degree P} \bigg)
\leq \exp \bigg( (\delta \log q) z + \frac{1}{\delta \log q} \frac{\mathfrak{d} q^{w(z)}}{w(z)^2} \bigg) ,
\end{align*}
where the last inequality follows from Lemma \ref{Sum over primes of 1/degree P}. By using the definition of $w(z)$, we have that

\begin{align*}
\frac{\mathfrak{d} q^{w(z)}}{w(z)^2}
= \frac{\mathfrak{d} z}{(\log_q z)^2} ;
\end{align*}
and if we take

\begin{align*}
\delta
= \frac{\sqrt{\mathfrak{d}}}{\log z} ,
\end{align*}
then

\begin{align*}
\sum_{\substack{N \in \mathcal{M} \\ \degree N \leq z \\  p_{+} (N) \leq w(z)}} 1
\leq \exp \bigg( \frac{\sqrt{\mathfrak{d}} (\log q) z}{\log z} + \frac{\sqrt{\mathfrak{d}} (\log z) z}{(\log q) (\log_q z)^2} \bigg)
\leq q^{\sqrt{\mathfrak{d}} \frac{z}{(\log z)}} .
\end{align*}
\end{proof}

\begin{lemma} \label{Lemma, sum of d(N)/N over N monic, degree N geq z/2, largest divisor leq z/r}
Let $z$ and $r$ be a positive integers satisfying $r \log_q r \leq z$. Then,

\begin{align*}
\sum_{\substack{N \in \mathcal{M} \\ \degree N \geq \frac{z}{2} \\ P_{+} (N) \leq \frac{z}{r}}} \frac{d(N)}{\vert N \rvert}
\ll z^2 \exp \bigg( -\frac{r \log r}{9} \bigg) .
\end{align*}
\end{lemma}

\begin{proof}
Let $\frac{3}{4} \leq \delta <1$. We will optimise on the value of $\delta$ later. We have that

\begin{align}
\begin{split} \label{Sum of d(N)/N over N monic, degree N geq z/2, and largest divisor leq z/r}
\sum_{\substack{N \in \mathcal{M} \\ \degree N \geq \frac{z}{2} \\ P_{+} (N) \leq \frac{z}{r}}} \frac{d(N)}{\vert N \rvert}
\leq &q^{(\delta -1) \frac{z}{2}} \sum_{\substack{N \in \mathcal{M} \\ \degree N \geq \frac{z}{2} \\ P_{+} (N) \leq \frac{z}{r}}} \frac{d(N)}{\vert N \rvert^{\delta} }
\leq q^{(\delta -1) \frac{z}{2}} \sum_{\substack{N \in \mathcal{M} \\ P_{+} (N) \leq \frac{z}{r}}} \frac{d(N)}{\vert N \rvert^{\delta} } \\
\leq &q^{(\delta -1) \frac{z}{2}} \prod_{\degree P \leq \frac{z}{r}} \bigg( 1 + \frac{2}{\vert P \rvert^{\delta}} + \sum_{l=2}^{\infty} \frac{l+1}{\lvert P \rvert^{l \delta}} \bigg) \\
\leq &\exp \bigg( (\log q) (\delta -1) \frac{z}{2} + 2 \sum_{\degree P \leq \frac{z}{r}} \frac{1}{\lvert P \rvert^{\delta}} + \sum_{\degree P \leq \frac{z}{r}} \sum_{l=2}^{\infty} \frac{l+1}{\lvert P \rvert^{l \delta}} \bigg)
\end{split}
\end{align}
where the last relation uses the Taylor series for the exponential function.

\noindent Note that

\begin{align}
\begin{split} \label{Sum of (l+1)/P^(l delta) over degree P leq z/r and l=2 to infinity}
\sum_{\degree P \leq \frac{z}{r}} \sum_{l=2}^{\infty} \frac{l+1}{\lvert P \rvert^{l \delta}}
\leq &\sum_{P \text{ prime}} \frac{3}{\vert P \rvert^{2 \delta}} \sum_{l=0}^{\infty} \frac{l+1}{\lvert P \rvert^{l \delta}}
= \sum_{P \text{ prime}} \frac{3}{\vert P \rvert^{2 \delta}} \bigg( \frac{1}{1 - \frac{1}{\vert P \rvert^{\delta}}} \bigg)^2
= 3 \sum_{P \text{ prime}} \bigg( \frac{1}{ \vert P \rvert^{\delta} -1} \bigg)^2 \\
= &O(1) ,
\end{split}
\end{align}
where the last relation uses the fact that $\delta \geq \frac{3}{4}$. Also, we can write $\frac{1}{\vert P \rvert^{\delta}} = \frac{1}{\vert P \rvert} + \frac{1}{\vert P \rvert} \Big( \lvert P \rvert^{1 -\delta} -1 \Big)$. We have that

\begin{align}
\begin{split} \label{Sum of 1/P over degree P leq z/r}
\sum_{\degree P \leq \frac{z}{r}} \frac{1}{\lvert P \rvert}
= \sum_{n=1}^{\frac{z}{r}} \frac{1}{q^n} \bigg( \frac{q^n}{n} + O \Big( \frac{q^{\frac{n}{2}}}{n} \Big) \bigg)
\leq \log z - \log r + O(1)
\leq \log (z) + O(1),
\end{split}
\end{align}
and that

\begin{align}
\begin{split} \label{Sum of 1/P (P^(1-delta) -1) over degree P leq z/r}
\sum_{\degree P \leq \frac{z}{r}} \frac{1}{\lvert P \rvert} \Big( \lvert P \rvert^{1 -\delta} -1 \Big)
= &\sum_{\degree P \leq \frac{z}{r}} \frac{1}{\lvert P \rvert} \sum_{n=1}^{\infty} \frac{\big( (1-\delta) \log \vert P \rvert \big)^n}{n!} \\
\leq &\sum_{n=1}^{\infty} \frac{(1-\delta)^n \big( (\log q) \frac{z}{r} \big)^{n-1}}{n!} \sum_{ \degree P \leq \frac{z}{r}} \frac{(\log q) \degree P}{\lvert P \rvert} \\
\leq &(1 + \mathfrak{c}) \sum_{n=1}^{\infty} \frac{(1-\delta)^n \big( (\log q) \frac{z}{r} \big)^{n}}{n!} 
= (1 + \mathfrak{c}) q^{(1-\delta) \frac{z}{r}} ,
\end{split}
\end{align}
where the second-to-last relation follows from a similar calculation as (\ref{Sum of 1/P over degree P leq z/r}). \\

\noindent We substitute (\ref{Sum of (l+1)/P^(l delta) over degree P leq z/r and l=2 to infinity}), (\ref{Sum of 1/P over degree P leq z/r}), and (\ref{Sum of 1/P (P^(1-delta) -1) over degree P leq z/r}) into (\ref{Sum of d(N)/N over N monic, degree N geq z/2, and largest divisor leq z/r}) to obtain

\begin{align*}
\sum_{\substack{N \in \mathcal{M} \\ \degree N \geq \frac{z}{2} \\ P_{+} (N) \leq \frac{z}{r}}} \frac{d(N)}{\vert N \rvert}
\ll z^2 \exp \bigg( \log q (\delta -1) \frac{z}{2} + 2 (1 + \mathfrak{c}) q^{(1-\delta) \frac{z}{r}} \bigg) .
\end{align*}
We can now take $\delta = 1 - \frac{r \log_q r}{4z}$ (by the conditions on $r$ given in theorem, we have that $\frac{3}{4} \leq \delta < 1$, as required). Then,

\begin{align*}
\sum_{\substack{N \in \mathcal{M} \\ \degree N \geq \frac{z}{2} \\ P_{+} (N) \leq \frac{z}{r}}} \frac{d(N)}{\vert N \rvert}
\ll z^2 \exp \bigg( -\frac{r \log r}{8} + 2 (1 + \mathfrak{c})  r^{\frac{1}{4}} \bigg) 
\ll z^2 \exp \bigg( -\frac{r \log r}{9} \bigg) .
\end{align*}
\end{proof}

\begin{proof}[Proof of Theorem \ref{Brun-Titschmarsh theorem, divisor function case in function field}]
We will need to break the sum into four parts. First, we define $z := \frac{\alpha}{10} y$. Now, for any $N$ in the summation range, we can write

\begin{align} \label{Prime decomposition of N}
N
= {P_1}^{e_1} \ldots {P_j}^{e_j} {P_{j+1}}^{e_{j+1}} \ldots {P_n}^{e_n}
\end{align}
where $\degree P_1 \leq \degree P_2 \leq \ldots \leq \degree P_n$ and $j \geq 0$ is chosen such that

\begin{align*}
\degree \Big( {P_1}^{e_1} \ldots {P_j}^{e_j} \Big) \leq z < \degree \Big( {P_1}^{e_1} \ldots {P_j}^{e_j} {P_{j+1}}^{e_{j+1}} \Big) .
\end{align*}
For convenience, we write

\begin{align*}
B_N := &{P_1}^{e_1} \ldots {P_j}^{e_j} , \\
D_N := &{P_{j+1}}^{e_{j+1}} \ldots {P_n}^{e_n} .
\end{align*}
We will consider the following cases:

\begin{enumerate}
\item $ p_{-} (D_N) > \frac{1}{2} z$ ;
\item $ p_{-} (D_N) \leq  \frac{1}{2} z$ and $\degree B_N \leq \frac{1}{2} z$ ;
\item $ p_{-} (D_N) < w(z)$ and $\degree B_N > \frac{1}{2} z$ ;
\item $w(z) \leq  p_{-} (D_N) \leq \frac{1}{2} z$ and $\degree B_N > \frac{1}{2} z$ ;
\end{enumerate}
where 

\begin{align*}
w(z)
:= \begin{cases}
1 &\text{ if $z \leq q$} \\
\log_q (z) &\text{ if $z > q$} .
\end{cases}
\end{align*}

\noindent \textbf{\underline{Case 1:}} We have that

\begin{align*}
\sum_{\substack{N \in \mathcal{M} \\ \degree (N-X) < y \\ N \equiv A (\modulus G) \\  p_{-} (D_N) > \frac{1}{2} z}} d(N)
= \sum_{\substack{N \in \mathcal{M} \\ \degree (N-X) < y \\ N \equiv A (\modulus G) \\  p_{-} (D_N) > \frac{1}{2} z}} d(B_N ) d(D_N )
\leq \sum_{\substack{B \in \mathcal{M} \\ \degree B \leq z \\ (B,G)=1}} d(B) \sum_{\substack{D \in \mathcal{M} \\ \degree (D-X_B ) < y - \degree B \\ D \equiv A_B (\modulus G) \\  p_{-} (D) > \frac{1}{2} z}} d(D) ,
\end{align*}
where $X_B$ is a monic polynomial of degree $\degree X - \degree B$ such that $\degree \big( X - B X_B \big) < y$, and $A_B$ is a polynomial satisfying $A_B B \equiv A (\modulus G)$. \\

\noindent We note that

\begin{align*}
\Omega (D) \leq \frac{\degree D}{p_{-} (D)} \leq \frac{y}{\frac{1}{2} z} = \frac{20}{\alpha } ,
\end{align*}
and so $d(D)  \leq 2^{\frac{20}{\alpha }}$ . Hence,

\begin{align*}
\sum_{\substack{N \in \mathcal{M} \\ \degree (N-X) < y \\ N \equiv A (\modulus G) \\  p_{-} (D_N) > \frac{1}{2} z}} d(N)
\ll_{\alpha} \sum_{\substack{B \in \mathcal{M} \\ \degree B \leq z \\ (B,G)=1}} d(B) \sum_{\substack{D \in \mathcal{M} \\ \degree (D-X_B ) < y - \degree B \\ D \equiv A_B (\modulus G) \\  p_{-} (D) > \frac{1}{2} z}} 1 .
\end{align*}
We can now apply Corollary \ref{Selberg sieve application, sum of 1 over short interval, arithmetic progression, large prime divisor} to obtain

\begin{align}
\begin{split} \label{Case 1, proof of Brun-Titschmarsh theorem, divisor function case in function field}
\sum_{\substack{N \in \mathcal{M} \\ \degree (N-X) < y \\ N \equiv A (\modulus G) \\  p_{-} (D_N) > \frac{1}{2} z}} d(N)
\ll_{\alpha} & \frac{ q^y }{\phi (G) z} \sum_{\substack{B \in \mathcal{M} \\ \degree B \leq z \\ (B,G)=1}} \frac{d(B)}{\lvert B \rvert} + q^{z} \sum_{\substack{B \in \mathcal{M} \\ \degree B \leq z \\ (B,G)=1}} d(B)
\leq \bigg( \frac{ 2 q^y }{\phi (G) z} +q^{2z} \bigg) \sum_{\substack{B \in \mathcal{M} \\ \degree B \leq z \\ (B,G)=1}} \frac{d(B)}{\lvert B \rvert} \\
\leq &\bigg( \frac{ 2 q^y }{\phi (G) z} +q^{2z} \bigg) z^2 
\ll \frac{ q^y z }{\phi (G)}
\leq \frac{ q^y \degree X }{\phi (G)} ,
\end{split}
\end{align}
where the second-to-last relation uses the fact that $\degree G \leq (1-\alpha ) y$ and $z= \frac{\alpha }{10}y$. \\

\noindent \textbf{\underline{Case 2:}} Suppose $N$ satisfies case 2. Then, the associated $P_{j+1}$ (from (\ref{Prime decomposition of N})) satisfies ${P_{j+1}}^{e_{j+1}} \mid N$, $\degree P_{j+1} \leq \frac{1}{2} z$, and $\degree {P_{j+1}}^{e_{j+1}} > \frac{1}{2} z$. For a general prime $P$ with $\degree P \leq \frac{1}{2} z$ we denote $e_P \geq 2$ to be the smallest integer such that $\degree P^{e_P} > \frac{1}{2} z$.  We will need to note for later that

\begin{align*}
\sum_{\degree P \leq \frac{1}{2} z} \frac{1}{\lvert P \rvert^{e_P}}
\leq \sum_{\degree P \leq \frac{1}{4} z} q^{-\frac{1}{2} z}+ \sum_{\frac{1}{4} z < \degree P \leq \frac{1}{2} z} \frac{1}{\lvert P \rvert^{2}} 
\ll q^{-\frac{1}{4} z} .
\end{align*}

\noindent Let us also note that for $N$ with $\degree N \leq \degree X$ we have that

\begin{align*}
d(N)
\ll_{\alpha , \beta} \lvert N \rvert^{\frac{\alpha \beta }{80}}
\leq \lvert X \rvert^{\frac{\alpha \beta }{80}}
\leq q^{\frac{\alpha }{80}y}
= q^{\frac{1}{8} z} .
\end{align*}

\noindent So,

\begin{align}
\begin{split} \label{Case 2, proof of Brun-Titschmarsh theorem, divisor function case in function field}
\sum_{\substack{N \in \mathcal{M} \\ \degree (N-X) < y \\ N \equiv A (\modulus G) \\  p_{-} (D_N) \leq \frac{1}{2} z \\ \degree B_N \leq \frac{1}{2} z}} d(N)
\leq &\sum_{\substack{\degree P \leq \frac{1}{2} z \\ (P,G)=1}} \sum_{\substack{ N \in \mathcal{M} \\ \degree (N-X) < y \\ N \equiv A (\modulus G) \\ N \equiv 0 (\modulus P^{e_P})}} d(N)
\ll_{\alpha , \beta} q^{\frac{1}{8}z} \sum_{\substack{\degree P \leq \frac{1}{2} z \\ (P,G)=1}} \sum_{\substack{ N \in \mathcal{M} \\ \degree (N-X) < y \\ N \equiv A (\modulus G) \\ N \equiv 0 (\modulus P^{e_P})}} 1 \\
\leq &q^{\frac{1}{8}z} \sum_{\substack{\degree P \leq \frac{1}{2} z \\ (P,G)=1}} \bigg( q^{y} \frac{1}{\lvert G P^{e_P} \rvert} + O(1) \bigg)
\leq q^{y} \frac{1}{\lvert G \rvert} q^{\frac{1}{8}z} \sum_{\degree P \leq \frac{1}{2} z} \frac{1}{\lvert P^{e_P} \rvert} \; + O \big( q^{\frac{5}{8} z} \big) \\
\ll & q^{y} \frac{1}{\lvert G \rvert} q^{-\frac{1}{8}z} + O \big( q^{\frac{5}{8} z} \big)
\ll q^{y} \frac{1}{\lvert G \rvert} q^{-\frac{1}{8}z} ,
\end{split}
\end{align}
where the last relation follows from the fact that $z=\frac{\alpha}{10} y$ and $\degree G \leq (1- \alpha ) y$.\\

\noindent \textbf{\underline{Case 3:}} Suppose $N$ satisfies case 3. For the case where $z \leq q$ we have that $w(z) =1$, meaning that the only possible value $N$ could take is $1$. At most this contributes $O(1)$. \\

\noindent So, suppose that $z > q$, and so $w(z) = \log_q z$. Case 3 tells us that $\frac{1}{2} z < \degree B_N \leq z$ and

\begin{align*}
p_{+} (B_N )
\leq p_{-} (D_N )
< w(z) .
\end{align*}
Hence,

\begin{align}
\begin{split} \label{Case 3, proof of Brun-Titschmarsh theorem, divisor function case in function field}
\sum_{\substack{N \in \mathcal{M} \\ \degree (N-X) < y \\ N \equiv A (\modulus G) \\  p_{-} (D_N) < w(z) \\ \frac{1}{2} z <\degree B_N \leq z}} d(N) 
\ll_{\alpha , \beta} \; &q^{\frac{1}{8} z} \sum_{\substack{N \in \mathcal{M} \\ \degree (N-X) < y \\ N \equiv A (\modulus G) \\  p_{-} (D_N) < w(z) \\ \frac{1}{2} z <\degree B_N \leq z}} 1 \\
\leq &q^{\frac{1}{8} z} \sum_{\substack{B \in \mathcal{M} \\ \frac{1}{2} z < \degree B \leq z \\ (B,G)=1 \\ p_{+} (B ) < w(z)}} \sum_{\substack{N \in \mathcal{M} \\ \degree (N-X) < y \\ N \equiv A (\modulus G) \\ N \equiv 0 (\modulus B)}} 1 \\
\leq &q^{\frac{1}{8} z} \sum_{\substack{B \in \mathcal{M} \\ \frac{1}{2} z < \degree B \leq z \\ p_{+} (B ) < w(z)}} \bigg( \frac{q^y}{\lvert G B \rvert} + O(1) \bigg) \\
\leq &\bigg( \frac{q^y}{\lvert G \rvert} q^{-\frac{3}{8} z} \sum_{\substack{B \in \mathcal{M} \\ \frac{1}{2} z < \degree B \leq z \\ p_{+} (B ) < w(z)}} 1 \bigg) + O \big( q^{\frac{9}{8} z} \big) \\
\ll &\bigg( \frac{q^y}{\lvert G \rvert} q^{-\frac{3}{8} z} q^{\frac{1}{4} z} \bigg) + O \big( q^{\frac{9}{8} z} \big) \\
\ll &\frac{q^y}{\lvert G \rvert} q^{-\frac{1}{8} z} 
\end{split}
\end{align}
as $z \longrightarrow \infty$, where the second-to-last relation follows from Lemma \ref{Sum of 1 over N monic, degree N leq z, largest divisor less than w}, and the last relation uses the fact that $\degree G \leq (1-\alpha ) y$ and $z=\frac{\alpha }{10} y$. \\

\noindent \textbf{\underline{Case 4:}} The case $z<1$ is trivial, and so we proceed under the assumption that $z \geq 1$. We have that

\begin{align} \label{Case 4, equation 1}
\sum_{\substack{N \in \mathcal{M} \\ \degree (N-X) < y \\ N \equiv A (\modulus G) \\ w(z) \leq  p_{-} (D_N ) \leq \frac{1}{2} z \\ \frac{1}{2} z < \degree B_N \leq z}} d(N)
= \sum_{\substack{ B \in \mathcal{M} \\ \frac{1}{2} z < \degree B \leq z \\ (B,G)=1}} d(B) \sum_{\substack{N \in \mathcal{M} \\ \degree (N-X) < y \\ N \equiv A (\modulus G) \\ w(z) \leq  p_{-} (D_N ) \leq \frac{1}{2} z \\  B_N = B \\ p_{-} (D_N ) \geq p_{+} (B_N )}} d( D_N ) .
\end{align}
We now divide $p_{-} (D_N)$ into the blocks $ \frac{1}{r+1} z <  p_{-} (D_N ) \leq \frac{1}{r} z$ for $r = 2, 3, \ldots , r_1$ where

\begin{align*}
r_1 = \Big\lfloor \frac{z}{w(z)} \Big\rfloor .
\end{align*}
For $D_N$ satisfying $ \frac{1}{r+1} z <  p_{-} (D_N ) \leq \frac{1}{r} z$ we have that

\begin{align*}
\Omega (D_N )
\leq \frac{\degree X}{ p_{-} (D_N )}
\leq \frac{\degree X}{\frac{1}{r+1} z}
\leq \frac{10 (r+1)}{\alpha \beta}
\leq \frac{20 r}{\alpha \beta} ,
\end{align*}
and so

\begin{align*}
d (D_N )
\leq 2^{\frac{20 r}{\alpha \beta}}
= a^r ,
\end{align*}
where $a = 2^{\frac{20 }{\alpha \beta}}$. So, continuing from (\ref{Case 4, equation 1}),

\begin{align}
\begin{split} \label{Case 4, equation 2}
\sum_{\substack{N \in \mathcal{M} \\ \degree (N-X) < y \\ N \equiv A (\modulus G) \\ w(z) \leq  p_{-} (D_N ) \leq \frac{1}{2} z \\ \degree B_N > \frac{1}{2} z}} d(N)
\leq &\sum_{r=2}^{r_1} a^r \sum_{\substack{B \in \mathcal{M} \\ \frac{1}{2} z < \degree B \leq z \\ (B,G)=1 \\ p_{+} (B) \leq \frac{1}{r} z}} d(B) \sum_{\substack{N \in \mathcal{M} \\ \degree (N-X) < y \\ N \equiv A (\modulus G) \\ N \equiv 0 (\modulus B) \\ \frac{1}{r+1} z <  p_{-} (D_N ) \leq \frac{1}{r} z }} 1 \\
\leq &\sum_{r=2}^{r_1} a^r \sum_{\substack{B \in \mathcal{M} \\ \frac{1}{2} z < \degree B \leq z \\ (B,G)=1 \\ p_{+} (B) \leq \frac{1}{r} z}} d(B) \sum_{\substack{D \in \mathcal{M} \\ \degree (N-X_B ) < y - \degree B \\ N \equiv A_B (\modulus G) \\  p_{-} (D ) > \frac{1}{r+1} z }} 1 ,
\end{split}
\end{align}
where $X_B$ is a monic polynomial of degree $\degree X - \degree B$ such that $\degree X - B X_B < y$, and $A_B$ is a polynomial satisfying $A_B B \equiv A (\modulus G)$. \\

\noindent Corollary \ref{Selberg sieve application, sum of 1 over short interval, arithmetic progression, large prime divisor} tells us that

\begin{align*}
\sum_{\substack{D \in \mathcal{M} \\ \degree (N-X_B ) < y - \degree B \\ N \equiv A_B (\modulus G) \\  p_{-} (D ) \geq \frac{1}{r+1} z }} 1
\leq \frac{q^y}{\phi (G) \vert B \rvert} \frac{r+1}{z} + q^{\frac{2z}{r+1}}
\ll \frac{q^y}{\phi (G) \vert B \rvert} \frac{r+1}{z} ,
\end{align*}
where the last relation follows from the fact that $\degree B \leq z$, $z = \frac{\alpha}{10} y$, and $\degree G \leq (1-\alpha ) y$. Hence, continuing from (\ref{Case 4, equation 2}):

\begin{align*}
\sum_{\substack{N \in \mathcal{M} \\ \degree (N-X) < y \\ N \equiv A (\modulus G) \\ v <  p_{-} (D_N ) \leq \frac{1}{2} z \\ \degree B_N > \frac{1}{2} z}} d(N)
\ll \frac{q^y}{\phi (G)} \frac{1}{z} \sum_{r=2}^{r_1} (r+1) a^r \sum_{\substack{B \in \mathcal{M} \\ \frac{1}{2} z < \degree B \leq z \\ (B,G)=1 \\ p_{+} (B) \leq \frac{1}{r} z}} \frac{d(B)}{\lvert B \rvert} .
\end{align*}
Finally, we wish to apply Lemma \ref{Lemma, sum of d(N)/N over N monic, degree N geq z/2, largest divisor leq z/r}. This requires that $r \log_q r \leq z$. Now, when $1 \leq z \leq q$ we have that $w(z) =1$ and $r_1 =z$. Hence, $r \log_q r \leq z \log_q q = z$. When $z > q$ we have that $w(z) = \log_q z$ and $r_1 = \Big\lfloor \frac{z}{w(z)} \Big\rfloor$. Hence, $r \log_q r \leq \frac{z}{\log_q z} (\log_q z - \log_q \log_q z) \leq z$, since $z>q$. Hence,

\begin{align}
\begin{split} \label{Case 4, proof of Brun-Titschmarsh theorem, divisor function case in function field}
\sum_{\substack{N \in \mathcal{M} \\ \degree (N-X) < y \\ N \equiv A (\modulus G) \\ v <  p_{-} (D_N ) \leq \frac{1}{2} z \\ \degree B_N > \frac{1}{2} z}} d(N)
\ll \frac{q^y}{\phi (G)} z \sum_{r=2}^{r_1} (r+1) a^r \exp \bigg( -\frac{r \log r}{9} \bigg)
\ll \frac{q^y}{\phi (G)} z
\ll \frac{q^y}{\phi (G)} \degree X . \\
\end{split}
\end{align}

\noindent The proof  now follows from (\ref{Case 1, proof of Brun-Titschmarsh theorem, divisor function case in function field}), (\ref{Case 2, proof of Brun-Titschmarsh theorem, divisor function case in function field}), (\ref{Case 3, proof of Brun-Titschmarsh theorem, divisor function case in function field}), and (\ref{Case 4, proof of Brun-Titschmarsh theorem, divisor function case in function field}).
\end{proof}

\begin{proof}[Proof of Theorem \ref{Brun-Titschmarsh theorem, divisor function case in function field, y equality version}]
The proof of this theorem is almost identical to the proof of Theorem \ref{Brun-Titschmarsh theorem, divisor function case in function field}. Where we applied Corollary \ref{Selberg sieve application, sum of 1 over short interval, arithmetic progression, large prime divisor}, we should instead apply Corollary \ref{Selberg sieve application, sum of 1 over short interval, arithmetic progression, large prime divisor, y equality case}. Also, the calculations

\begin{align*}
\sum_{\substack{N \in \mathcal{M} \\ \degree (N-X) <y \\ N \equiv A (\modulus G) \\ N \equiv 0 (\modulus P^{e_P})}} 1 = \frac{q^y}{\lvert G P^{e_P} \rvert} + O(1)
\quad \quad \quad \text{ and } \quad \quad \quad
\sum_{\substack{N \in \mathcal{M} \\ \degree (N-X) <y \\ N \equiv A (\modulus G) \\ N \equiv 0 (\modulus B)}} 1 = \frac{q^y}{\lvert G B \rvert} + O(1) 
\end{align*}
should be replaced by

\begin{align*}
\sum_{\substack{N \in \mathcal{A} \\ \degree (N-X) =y \\ (N-X) \in a\mathcal{M} \\ N \equiv A (\modulus G) \\ N \equiv 0 (\modulus P^{e_P})}} 1 = \frac{q^y}{\lvert G P^{e_P} \rvert} + O(1)
\quad \quad \quad \text{ and } \quad \quad \quad
\sum_{\substack{N \in \mathcal{A} \\ \degree (N-X) =y \\ (N-X) \in a\mathcal{M} \\ N \equiv A (\modulus G) \\ N \equiv 0 (\modulus B)}} 1 = \frac{q^y}{\lvert G B \rvert} + O(1) ,
\end{align*}
respectively.
\end{proof}


\section{Further Preliminary Results}

\begin{lemma} \label{integral over Re s = c of (y/n)^s / s^3}
Let $c$ be a positive real number, and let $k \geq 2$ be an integer. Then,

\begin{align*}
\int_{c- i \infty}^{c+ i \infty} \frac{y^s}{s^k} \mathrm{d} s = \begin{cases} 0 &\text{ if $0 \leq y < 1$} \\ \frac{2\pi i}{(k-1)!} (\log y)^{k-1} &\text{ if $y \geq 1$} \end{cases} .
\end{align*}
\end{lemma}

\begin{proof}
We will first look at the case when $y \geq 1$. Let $n$ be a positive integer, and define the following curves: 

\begin{align*}
l_1 (n) := &[c-ni , c+ni] , \\
l_2 (n) := &[c+ni , ni] , \\
l_3 (n) := &\Big\{ n e^{it} : t \in \Big[ \frac{\pi}{2} , \frac{3 \pi}{2} \Big] \Big\} \text{ (orientated anticlockwise)} , \\
l_4 (n) := &[-ni , c-ni] , \\
L (n) := &l_1 (n) \cup l_2 (n) \cup l_3 (n) \cup l_4 (n) .
\end{align*}
We can see that

\begin{align*}
\int_{c- i \infty}^{c+ i \infty} \frac{y^s}{s^k} \mathrm{d} s = \lim_{n \rightarrow \infty} \Bigg( \int_{L(n)} \frac{y^s}{s^k} \mathrm{d} s - \int_{l_2 (n)} \frac{y^s}{s^k} \mathrm{d} s - \int_{l_3 (n)} \frac{y^s}{s^k} \mathrm{d} s - \int_{l_4 (n)} \frac{y^s}{s^k} \mathrm{d} s \Bigg).
\end{align*}

\noindent For the first integral we apply the residue theorem to obtain that

\begin{align*}
\lim_{n \rightarrow \infty} \int_{L (n)} \frac{y^s}{s^k} \mathrm{d} s = \frac{ 2\pi i}{(k-1)!} (\log y)^{k-1} .
\end{align*}
For $j \in \{ 2,4 \}$ we have that

\begin{align*}
\lim_{n \rightarrow \infty} \Bigg\lvert \int_{l_j (n)} \frac{y^s}{s^k} \mathrm{d} s \bigg\rvert \leq \lim_{n \rightarrow \infty} \frac{y^c}{n^k} \int_{l_j (n)} 1 \mathrm{d} s = \lim_{n \rightarrow \infty} \frac{c y^c}{n^k} = 0 .
\end{align*}
For the third integral we note that when $s \in l_3 (n)$ we have $\lvert y^s \rvert \leq 1$ (since $\Re s \leq 0$ and $y \geq 1$). Hence,

\begin{align*}
\lim_{n \rightarrow \infty} \bigg\lvert \int_{l_3 (n)} \frac{y^s}{s^k} \mathrm{d} s \bigg\rvert \leq \lim_{n \rightarrow \infty} \frac{1}{n^k} \int_{l_3 (n)} 1 \mathrm{d} s = \lim_{n \rightarrow \infty} \frac{\pi}{n^{k-1}} = 0.
\end{align*}
So, for $y \geq 1$ we deduce that

\begin{align*}
\int_{c- i \infty}^{c+ i \infty} \frac{y^s}{s^k} \mathrm{d} s = \frac{2\pi i}{(k-1)!} (\log y)^{k-1}.
\end{align*}

\noindent Now we will look at the case when $0 \leq y < 1$. Again, let $n$ be a positive integer, and define the following curves: 

\begin{align*}
l_1 (n) := &[c-ni , c+ni] , \\
l_3 (n) := &\Big\{ c + n e^{it} : t \in \Big[ \frac{-\pi}{2} , \frac{\pi}{2} \Big] \Big\} \text{ (orientated clockwise)} ,\\
L (n) := &l_1 (n) \cup l_2 (n) .
\end{align*}
We can see that

\begin{align*}
\int_{c- i \infty}^{c+ i \infty} \frac{y^s}{s^k} \mathrm{d} s = \lim_{n \rightarrow \infty} \Bigg( \int_{L (n)} \frac{y^s}{s^k} \mathrm{d} s - \int_{l_2 (n)} \frac{y^s}{s^k} \mathrm{d} s \Bigg).
\end{align*}

\noindent The limit of the first integral is equal to zero by the residue theorem, because there are no poles inside $L (n)$. The limit of the second integral is also zero, and this can be shown by a method similar to that applied for the curve $l_3 (n)$ in the case $y \geq 1$. So, for $0 \leq y < 1$ we deduce that

\begin{align*}
\int_{c- i \infty}^{c+ i \infty} \frac{y^s}{s^k} \mathrm{d} s = 0.
\end{align*}
\end{proof}

\noindent We now have the following proposition.

\begin{proposition} \label{Sum over (A,R)=1, deg A <= x of 1/A}
Let $R \in \mathcal{M}$ and let $x$ be a positive integer. Then,

\begin{align*}
\sum_{\substack{ A \in \mathcal{M} \\ \degree A \leq x \\ (A , R) = 1}} \frac{1}{\lvert A \rvert}
= \begin{cases}
\frac{\phi (R)}{\lvert R \rvert} x + O \big(\log \omega (R) \big)  &\text{ if $x \geq \degree R$} \\
\frac{\phi (R)}{\lvert R \rvert} x + O \big(\log \omega (R) \big) + O \Big( \frac{2^{\omega (R)} x}{q^x} \Big) &\text{ if $x < \degree R$}
\end{cases} .
\end{align*}
\end{proposition}

\begin{proof}
For all positive integers $x$ we have that

\begin{align*}
\sum_{\substack{ A \in \mathcal{M} \\ \degree A \leq x \\ (A , R) = 1}} \frac{1}{\lvert A \rvert}
&= \sum_{\substack{ A \in \mathcal{M} \\ \degree F \leq x}} \frac{1}{\lvert A \rvert} \sum_{E \mid (A,R)} \mu (E)
= \sum_{E \mid R} \mu (E) \sum_{\substack{ A \in \mathcal{M} \\ \degree A \leq x \\ E \mid A}} \frac{1}{\lvert A \rvert}
= \sum_{\substack{E \mid R \\ \degree E \leq x}} \frac{\mu (E)}{\lvert E \rvert} \sum_{\substack{ A \in \mathcal{M} \\ \degree A \leq x - \degree E}} \frac{1}{\lvert A \rvert} \\
= &\sum_{\substack{E \mid R \\ \degree E \leq x}} \frac{\mu (E)}{\lvert E \rvert} ( x - \degree E )
= \sum_{E \mid R} \frac{\mu (E)}{\lvert E \rvert} ( x - \degree E ) - \sum_{\substack{E \mid R \\ \degree E > x}} \frac{\mu (E)}{\lvert E \rvert} ( x - \degree E ) .
\end{align*}
By (\ref{Sum of mu(E)/E^s over E mid R}), (\ref{Sum of mu(E) deg E/E^s over E mid R}), and Lemma \ref{Sum over P divides R of deg P / (P-1)}, we see that

\begin{align*}
\sum_{E \mid R} \frac{\mu (E)}{\lvert E \rvert} ( x - \degree E ) = \frac{\phi (R)}{\lvert R \rvert} x + O \bigg( \frac{\phi (R)}{\lvert R \rvert} \log \omega (R) \bigg) .
\end{align*}
When $x \geq \degree R$, it is clear that

\begin{align*}
\sum_{\substack{E \mid R \\ \degree E > x}} \frac{\mu (E)}{\lvert E \rvert} ( x - \degree E ) = 0 .
\end{align*}
Whereas, when $x < \degree R$, we have that

\begin{align*}
\sum_{\substack{E \mid R \\ \degree E > x}} \frac{\mu (E)}{\lvert E \rvert} ( x - \degree E )
\ll \sum_{\substack{E \mid R \\ \degree E > x}} \frac{\lvert \mu (E) \rvert \degree E}{\lvert E \rvert}
\ll \frac{x}{q^x} \sum_{\substack{E \mid R \\ \degree E > x}} \lvert \mu (E) \rvert \ll \frac{2^{\omega (R)} x}{q^x} .
\end{align*}
The proof follows.
\end{proof}

\begin{lemma} \label{Prod of P div R of P-1/P+1 derivatives bounds}
For all $R \in \mathcal{A}$ and all $s \in \mathbb{C}$ with $\Re (s) > -1$ we define

\begin{align*}
f_R (s) := \prod_{P \mid R} \frac{1 - \lvert P \rvert^{-s-1}}{1 + \lvert P \rvert^{-s-1}} .
\end{align*}
Then, for all $R \in \mathcal{A}$ and $j=1,2,3,4$ we have that

\begin{align*}
f_R^{(j)} (0) \ll \big( \log_q \log_q \lvert R \rvert \big)^{j} \prod_{P \mid R} \frac{1 - \lvert P \rvert^{-1}}{1 + \lvert P \rvert^{-1}} .
\end{align*}
\end{lemma}

\begin{remark}
We must mention that, in the lemma and the proof, the implied constants may depend on $j$, for example; but because there are only finitely many cases of $j$ that we are interested in, we can take the implied constants to be independent.
\end{remark}

\begin{proof}
First, we note that

\begin{align} \label{Prod of P div R of P-1/P+1 first derivative}
f_R '(s)
= g_R (s) f_R (s) ,
\end{align}
where

\begin{align*}
g_R (s) 
:= \sum_{P \mid R} 2 \log \lvert P \rvert \bigg( \frac{1}{\lvert P \rvert^{s+1} + 1} + \frac{1}{\lvert P \rvert^{2s+2} - 1} \bigg) .
\end{align*}
We note further that

\begin{align}
\begin{split} \label{Prod of P div R of P-1/P+1 second, third, fourth derivative}
f_R ''(s) = &\Big( g_R (s)^2 + g_R ' (s) \Big) f_R (s) , \\
f_R '''(s) = &\Big( g_R (s)^3 + 3 g_R (s) g_R ' (s) + g_R '' (s) \Big) f_R (s) , \\
f_R '''' (s) = &\Big( g_R (s)^4 + 6 g_R (s)^2 g_R ' (s) + 4 g_R (s) g_R '' (s) + 3 g_R ' (s)^2 + g_R ''' (s) \Big) f_R (s) .
\end{split}
\end{align}

\noindent For all $R \in \mathcal{A}$ and $k=0,1,2,3$ it is not difficult to deduce that

\begin{align} \label{Prod of P div R of P-1/P+1 , g derivatives bounds}
g_R^{(k)} (0)
\ll \sum_{P \mid R} \frac{ \big( \log \lvert P \rvert \big)^{k+1} }{\lvert P \rvert - 1} .
\end{align}
The function $\frac{\big( \log x \big)^{k+1}}{x-1}$ is decreasing at large enough $x$, and the limit as $x \longrightarrow \infty$ is $0$. Therefore, there exist an independent constant $c \geq 1$ such that for $k=0,1,2,3$ and all $A,B \in \mathcal{A}$ with $\degree A \leq \degree B$ we have that

\begin{align*}
c \frac{\big( \log \lvert A \rvert \big)^{k+1}}{\lvert A \rvert -1}
\geq \frac{\big( \log \lvert B \rvert \big)^{k+1}}{\lvert B \rvert -1} .
\end{align*}
Hence, taking $n=\omega (R)$, we see that

\begin{align}
\begin{split} \label{sum P div R of (log P)^(k+1) / P-1 bound}
\sum_{P \mid R} \frac{ \big( \log \lvert P \rvert \big)^{k+1} }{\lvert P \rvert - 1}
\ll &\sum_{P \mid R_n } \frac{ \big( \log \lvert P \rvert \big)^{k+1} }{\lvert P \rvert - 1}
\ll \sum_{r=1}^{m_n +1 } \frac{q^r}{r} \frac{r^{k+1}}{q^r - 1}
\ll \sum_{n=1}^{m_n +1 } r^k \\
\ll &(m_n + 1 )^{k+1}
\ll \big( \log_q \log_q \lvert R_n \rvert \big)^{k+1}
\ll \big( \log_q \log_q \lvert R \rvert \big)^{k+1} ,
\end{split}
\end{align}
where we have used the prime polynomial theorem and Lemma \ref{Primorial bound}. \\

\noindent So, by (\ref{Prod of P div R of P-1/P+1 first derivative})--(\ref{sum P div R of (log P)^(k+1) / P-1 bound}) and the fact that

\begin{align*}
f_R (0)
= \prod_{P \mid R} \frac{1 - \lvert P \rvert^{-1}}{1 + \lvert P \rvert^{-1}} ,
\end{align*}
we deduce that

\begin{align*}
f_R^{(j)} (0) \ll \big( \log_q \log_q \lvert R \rvert \big)^{j} \prod_{P \mid R} \frac{1 - \lvert P \rvert^{-1}}{1 + \lvert P \rvert^{-1}} .
\end{align*}
\end{proof}

\begin{proposition} \label{Proposition Sum of 2^omega(N) log(|Q|/|N|)^2 / |N|}
Let $R \in \mathcal{A}$, and define ${z_R}' := \degree R - \log_q 9^{\omega (R)}$. We have that

\begin{align*}
\sum_{\substack{N \in \mathcal{M} \\ \degree N \leq {z_R}' \\ (N,R) = 1}} \frac{2^{\omega (N)}}{\lvert N \rvert} ( {z_R}' - \degree N )^2
= &\frac{(1-q^{-1})}{12} \bigg( \prod_{P \mid R}  \frac{1 - \lvert P \rvert^{-1}}{1 + \lvert P \rvert^{-1}} \bigg) (\degree R)^4 \\
&+ O \Bigg( \bigg( \prod_{P \mid R}  \frac{1 - \lvert P \rvert^{-1}}{1 + \lvert P \rvert^{-1}} \bigg) \Big( (\degree R)^3 \omega (R) + (\degree R)^3 \log \degree R \Big) \Bigg) .
\end{align*}
\end{proposition}

\begin{proof}

\textbf{\underline{STEP 1:}} Let us define the function $F$ for $\Re s > 1$ by

\begin{align*}
F(s)
= \sum_{\substack{N \in \mathcal{M} \\ (N,R) = 1}} \frac{2^{\omega (N)}}{\lvert N \rvert^s} .
\end{align*}
We can see that

\begin{align*}
F(s)
&= \prod_{\substack{P \text{ prime} \\ P \nmid R}} \bigg( 1 + \frac{2}{\lvert P \rvert^s} + \frac{2}{\lvert P \rvert^{2s}} + \frac{2}{\lvert P \rvert^{3s}} \ldots \bigg)
= \prod_{\substack{P \text{ prime} \\ P \nmid R}} \bigg( \frac{2}{1 - \lvert P \rvert^{-s}} - 1 \bigg)\\
&= \prod_{P \text{ prime}} \bigg( \frac{1 + \lvert P \rvert^{-s}}{1 - \lvert P \rvert^{-s}} \bigg) \prod_{P \mid R} \bigg( \frac{1 - \lvert P \rvert^{-s}}{1 + \lvert P \rvert^{-s}} \bigg)
= \frac{\zeta_{\mathcal{A}} (s)^2}{\zeta_{\mathcal{A}} (2s)} \prod_{P \mid R} \bigg( \frac{1 - \lvert P \rvert^{-s}}{1 + \lvert P \rvert^{-s}} \bigg) . 
\end{align*}

Now, let $c$ be a positive real number, and define $y_R := q^{{z_R}' }$. On the one hand, we have that

\begin{align}
\begin{split} \label{Inside Int F(1+s) y^s / s^3}
\frac{1}{\pi i} \int_{c - i \infty}^{c+ i \infty} F(1+s) \frac{{y_R}^s}{s^3} \mathrm{d} s 
= &\frac{1}{\pi i} \sum_{\substack{N \in \mathcal{M} \\ (N,R) = 1}} \frac{2^{\omega (N)}}{\lvert N \rvert} \int_{c - i \infty}^{c+ i \infty} \frac{{y_R}^s}{\lvert N \rvert ^s s^3} \mathrm{d} s
= \sum_{\substack{N \in \mathcal{M} \\ \degree N \leq {z_R}' \\ (N,R) = 1}} \frac{2^{\omega (N)}}{\lvert N \rvert} \log \Big( \frac{y_R }{\lvert N \rvert} \Big)^2 \\ 
= &( \log q )^2 \sum_{\substack{N \in \mathcal{M} \\ \degree N \leq {z_R}' \\ (N,R) = 1}} \frac{2^{\omega (N)}}{\lvert N \rvert} ( {z_R}' - \degree N )^2 ,
\end{split}
\end{align}
where the second equality follows from Lemma \ref{integral over Re s = c of (y/n)^s / s^3}. On the other hand, for all positive integers $n$ define the following curves: 

\begin{align*}
l_1 (n) := &\Big[ c - \frac{(2n+1) \pi i}{\log q} , c + \frac{(2n+1) \pi i}{\log q} \Big] \\
l_2 (n) := &\Big[ c + \frac{(2n+1) \pi i}{\log q} , -\frac{1}{4} + \frac{(2n+1) \pi i}{\log q} \Big] \\
l_3 (n) := &\Big[ -\frac{1}{4} + \frac{(2n+1) \pi i}{\log q} , -\frac{1}{4} - \frac{(2n+1) \pi i}{\log q} \Big] \\
l_4 (n) := &\Big[ -\frac{1}{4} - \frac{(2n+1) \pi i}{\log q} , c - \frac{(2n+1) \pi i}{\log q} \Big] \\
L (n) := &l_1 (n) \cup l_2 (n) \cup l_3 (n) \cup l_4 (n) .
\end{align*}
Then, we have that

\begin{align}
\begin{split}\label{Outside Int F(1+s) y^s / s^3}
&\frac{1}{\pi i} \int_{c - i \infty}^{c+ i \infty} F(1+s) \frac{{y_R}^s}{s^3} \mathrm{d} s\\
= & \frac{1}{\pi i} \lim_{n \rightarrow \infty} \Bigg( \int_{L (n)} F(1+s) \frac{{y_R}^s}{s^3} \mathrm{d} s - \int_{l_2 (n)} F(1+s) \frac{{y_R}^s}{s^3} \mathrm{d} s - \int_{l_3 (n)} F(1+s) \frac{{y_R}^s}{s^3} \mathrm{d} s - \int_{l_4 (n)} F(1+s) \frac{{y_R}^s}{s^3} \mathrm{d} s \bigg) . 
\end{split}
\end{align}

\noindent \textbf{\underline{STEP 2:}} For the first integral in (\ref{Outside Int F(1+s) y^s / s^3}) we note that $F(1+s) \frac{{y_R}^s}{s^3}$ has a fifth-order pole at $s=0$ and double poles at $s = \frac{2m \pi i}{\log q}$ for $m = \pm 1 , \pm 2 , \ldots , \pm n$. By applying Cauchy's residue theorem we see that

\begin{align} \label{Proposition Sum of 2^omega(N) log(|Q|/|N|)^2 / |N| , residues}
\lim_{n \rightarrow \infty} \frac{1}{\pi i} \int_{L(n)} F(1+s) \frac{{y_R}^s}{s^3} \mathrm{d} s
= 2 \residue_{s=0} F(s+1)\frac{{y_R}^s}{s^3} + 2\sum_{\substack{m \in \mathbb{Z} \\ m \neq 0}} \residue_{s=\frac{2m \pi i}{\log q}} F(1+s)\frac{{y_R}^s}{s^3} . 
\end{align}

\noindent \textbf{\underline{STEP 2.1:}} For the first residue term we have that

\begin{align} \label{Proposition Sum of 2^omega(N) log(|Q|/|N|)^2 / |N| , residue at 0}
\residue_{s=0} F(s+1)\frac{{y_R}^s}{s^3}
= \frac{1}{4!} \lim_{s \longrightarrow 0} \frac{\mathrm{d}^4}{\mathrm{d}s^4} \Bigg( \zeta_{\mathcal{A}} (s +1)^2 s^2 \frac{1}{\zeta_{\mathcal{A}} (2s+2)} \prod_{P \mid R} \bigg( \frac{1 - \lvert P \rvert^{-s-1}}{1 + \lvert P \rvert^{-s-1}} \bigg) {y_R}^s \Bigg) .
\end{align}
If we apply the product rule for differentiation, then one of the terms will be

\begin{align*} 
& \frac{1}{4!} \lim_{s \longrightarrow 0} \Bigg( \zeta_{\mathcal{A}} (s +1)^2 s^2 \frac{1}{\zeta_{\mathcal{A}} (2s+2)} \prod_{P \mid R} \bigg( \frac{1 - \lvert P \rvert^{-s-1}}{1 + \lvert P \rvert^{-s-1}} \bigg) \frac{\mathrm{d}^4}{\mathrm{d}s^4} {y_R}^s \Bigg) \\
= & \frac{(1-q^{-1}) (\log q )^2}{24} \prod_{P \mid R} \bigg( \frac{1 - \lvert P \rvert^{-1}}{1 + \lvert P \rvert^{-1}} \bigg) ({z_R}')^4 \\
= & \frac{(1-q^{-1}) (\log q )^2}{24} \prod_{P \mid R} \bigg( \frac{1 - \lvert P \rvert^{-1}}{1 + \lvert P \rvert^{-1}} \bigg) (\degree R)^4 + O \Bigg( \log q  \prod_{P \mid R} \bigg( \frac{1 - \lvert P \rvert^{-1}}{1 + \lvert P \rvert^{-1}} \bigg) (\degree R)^3 \omega (R) \Bigg) .
\end{align*}

\noindent Now we look at the remaining terms that arise from the product rule. By using the fact that $\zeta_{\mathcal{A}} (1+s) = \frac{1}{1-q^{-s}}$ and the Taylor series for $q^{-s}$, we have for $k=0,1,2,3,4$ that

\begin{align} \label{Proposition Sum of 2^omega(N) log(|Q|/|N|)^2 / |N| , zeta(s+1) s derivatives at 0}
\lim_{s \rightarrow 0} \frac{1}{( \log q)^{k-1}}\frac{\mathrm{d}^k}{\mathrm{d} s^k} \zeta (s+1) s
= O (1) ,
\end{align}
Similarly,

\begin{align} \label{Proposition Sum of 2^omega(N) log(|Q|/|N|)^2 / |N| , zeta(2s+2)^-1 derivatives at 0}
\lim_{s \rightarrow 0} \frac{\mathrm{d}^k}{\mathrm{d} s^k}\zeta (2s+2)^{-1}
= \lim_{s \rightarrow 0} \frac{\mathrm{d}^k}{\mathrm{d} s^k} \Big( 1 - q^{-1-2s} \Big)
= O(1) ,
\end{align}
By (\ref{Proposition Sum of 2^omega(N) log(|Q|/|N|)^2 / |N| , zeta(s+1) s derivatives at 0}), (\ref{Proposition Sum of 2^omega(N) log(|Q|/|N|)^2 / |N| , zeta(2s+2)^-1 derivatives at 0}), and Lemma \ref{Prod of P div R of P-1/P+1 derivatives bounds}, we see that the remaining terms are of order

\begin{align*}
( \log q )^2 \prod_{P \mid R} \bigg( \frac{1 - \lvert P \rvert^{-1}}{1 + \lvert P \rvert^{-1}} \bigg) (\degree R)^3 \log \degree R .
\end{align*}

\noindent Hence,

\begin{align} 
\begin{split} \label{Proposition Sum of 2^omega(N) log(|Q|/|N|)^2 / |N| , residue at 0 calculated}
2 \residue_{s=0} F(s+1)\frac{{y_R}^s}{s^3} 
= &\frac{(1-q^{-1}) (\log q )^2}{12} \bigg( \prod_{P \mid R}  \frac{1 - \lvert P \rvert^{-1}}{1 + \lvert P \rvert^{-1}} \bigg) (\degree R)^4 \\
&+ O \Bigg( ( \log q )^2 \bigg( \prod_{P \mid R}  \frac{1 - \lvert P \rvert^{-1}}{1 + \lvert P \rvert^{-1}} \bigg) \Big( (\degree R)^3 \omega (R) + (\degree R)^3 \log \degree R \Big) \Bigg) .
\end{split}
\end{align}

\noindent \textbf{\underline{STEP 2.2:}} Now we look at the remaining residue terms in (\ref{Proposition Sum of 2^omega(N) log(|Q|/|N|)^2 / |N| , residues}). By similar (but simpler) means as above can show that

\begin{align*}
\residue_{s=\frac{2m \pi i}{\log q}} F(1+s)\frac{{y_R}^s}{s^3}
= O \Bigg( \frac{1}{m^3} (\log q )^2 \bigg( \prod_{P \mid R}  \frac{1 - \lvert P \rvert^{-1}}{1 + \lvert P \rvert^{-1}} \bigg) \degree R \Bigg) ,
\end{align*}
and so

\begin{align} \label{Proposition Sum of 2^omega(N) log(|Q|/|N|)^2 / |N| , other residues calculated}
\sum_{\substack{m \in \mathbb{Z} \\ m \neq 0}} \residue_{s=\frac{2m \pi i}{\log q}} F(1+s)\frac{{y_R}^s}{s^3}
= O \Bigg( (\log q )^2 \bigg( \prod_{P \mid R}  \frac{1 - \lvert P \rvert^{-1}}{1 + \lvert P \rvert^{-1}} \bigg) \degree R \Bigg) .
\end{align}

\noindent \textbf{\underline{STEP 2.3:}} By (\ref{Proposition Sum of 2^omega(N) log(|Q|/|N|)^2 / |N| , residues}), (\ref{Proposition Sum of 2^omega(N) log(|Q|/|N|)^2 / |N| , residue at 0 calculated}) and (\ref{Proposition Sum of 2^omega(N) log(|Q|/|N|)^2 / |N| , other residues calculated}), we see that

\begin{align}
\begin{split} \label{Proposition Sum of 2^omega(N) log(|Q|/|N|)^2 / |N| , L(n) integral}
\lim_{n \rightarrow \infty} \frac{1}{\pi i} \int_{L(n)} F(1+s) \frac{{y_R}^s}{s^3} \mathrm{d} s 
= &\frac{(1-q^{-1}) (\log q )^2}{12} \bigg( \prod_{P \mid R}  \frac{1 - \lvert P \rvert^{-1}}{1 + \lvert P \rvert^{-1}} \bigg) (\degree R)^4 \\
&+ O \Bigg( ( \log q )^2 \bigg( \prod_{P \mid R}  \frac{1 - \lvert P \rvert^{-1}}{1 + \lvert P \rvert^{-1}} \bigg) \Big( (\degree R)^3 \omega (R) + (\degree R)^3 \log \degree R \Big) \Bigg) .
\end{split}
\end{align}

\noindent \textbf{\underline{STEP 3:}} We now look at the integrals over $l_2 (n)$ and $l_4 (n)$. For all positive integers $n$ and all $s \in l_2 (n) , l_4 (n)$ we have that $F(s+1) {y_R}^s = O_{q,R} (1)$. One can now easily deduce for $i = 2,4$ that

\begin{align} \label{Proposition Sum of 2^omega(N) log(|Q|/|N|)^2 / |N| , l_2 (n) and l_4 (n) integrals}
\lim_{n \rightarrow \infty} \Big\lvert \frac{1}{\pi i} \int_{l_i (n)} F(1+s) \frac{{y_R}^s}{s^3} \mathrm{d} s \Big\rvert = 0 .
\end{align}

\noindent \textbf{\underline{STEP 4:}} We now look at the integral over $l_3 (n)$. For all positive integers $n$ and all $s \in l_3 (n)$ we have that

\begin{align*}
\frac{\zeta_{\mathcal{A}} (s+1)^2}{\zeta_{\mathcal{A}} (2s+2)}
= O(1) 
\end{align*}
and

\begin{align*}
\bigg\lvert \bigg( \prod_{P \mid R} \frac{1 - \lvert P \rvert^{-s-1}}{1 + \lvert P \rvert^{-s-1}} \bigg) {y_R}^{s} \bigg\lvert
\ll &\bigg( \prod_{P \mid R} \frac{1 + \lvert P \rvert^{-\frac{3}{4}}}{1 - \lvert P \rvert^{-\frac{3}{4}}} \bigg) \bigg( \prod_{P \mid R} 9^{\frac{1}{4}} \bigg) \lvert R \rvert^{-\frac{1}{4}} \\
\ll &\bigg( \prod_{P \mid R} 1 + \frac{2}{2^{\frac{3}{4}} -1} \bigg) \bigg( \prod_{P \mid R} 2 \bigg) \lvert R \rvert^{-\frac{1}{4}} \\
\ll &\bigg( \prod_{P \mid R} 8 \bigg) \lvert R \rvert^{-\frac{1}{4}}
\ll \bigg( \prod_{P \mid R} \frac{8}{\lvert P \rvert^{\frac{1}{4}}} \bigg)
\ll 1 .
\end{align*}
We can now easily deduce that

\begin{align} \label{Proposition Sum of 2^omega(N) log(|Q|/|N|)^2 / |N| , l_3 (n) integral}
\lim_{n \rightarrow \infty} \Big\lvert \frac{1}{\pi i} \int_{l_3 (n)} F(1+s) \frac{{y_R}^s}{s^3} \mathrm{d} s \Big\rvert
= O(1) .
\end{align}

\noindent \textbf{\underline{STEP 5:}} By (\ref{Inside Int F(1+s) y^s / s^3}), (\ref{Outside Int F(1+s) y^s / s^3}), (\ref{Proposition Sum of 2^omega(N) log(|Q|/|N|)^2 / |N| , L(n) integral}), (\ref{Proposition Sum of 2^omega(N) log(|Q|/|N|)^2 / |N| , l_2 (n) and l_4 (n) integrals}) and (\ref{Proposition Sum of 2^omega(N) log(|Q|/|N|)^2 / |N| , l_3 (n) integral}), we deduce that

\begin{align*}
\sum_{\substack{N \in \mathcal{M} \\ \degree N \leq {z_R}' \\ (N,R) = 1}} \frac{2^{\omega (N)}}{\lvert N \rvert} ( {z_R}' - \degree N )^2
= &\frac{(1-q^{-1})}{12} \bigg( \prod_{P \mid R}  \frac{1 - \lvert P \rvert^{-1}}{1 + \lvert P \rvert^{-1}} \bigg) (\degree R)^4 \\
&+ O \Bigg( \bigg( \prod_{P \mid R}  \frac{1 - \lvert P \rvert^{-1}}{1 + \lvert P \rvert^{-1}} \bigg) \Big( (\degree R)^3 \omega (R) + (\degree R)^3 \log \degree R \Big) \Bigg) .
\end{align*}
\end{proof}

\begin{lemma} \label{Proposition Sum of 2^omega(N) / |N|}
We have that

\begin{align*}
\sum_{\substack{N \in \mathcal{M} \\ \degree N \leq x }} \frac{2^{\omega(N)}}{\lvert N \rvert}
= \frac{q-1}{2q} x^2 + \frac{3q+1}{2q} x + 1 .
\end{align*}
From this we easily deduce that

\begin{align*}
\sum_{\substack{N \in \mathcal{M} \\ \degree N \leq x }} \frac{2^{\omega(N)}}{\lvert N \rvert}
= O (x^2 ) .
\end{align*}
\end{lemma}

\begin{proof}
For $s>1$ we define

\begin{align*}
F(s)
:= \sum_{N \in \mathcal{M}} \frac{2^{\omega(N)}}{\lvert N \rvert^{s+1}}.
\end{align*}
We can see that

\begin{align*}
F(s)
= &\prod_{P \text{ prime}} \bigg( 1 + \frac{2}{\lvert P \rvert^{s+1}} + \frac{2}{\lvert P \rvert^{2(s+1)}} + \frac{2}{\lvert P \rvert^{3(s+1)}} + \ldots \bigg) 
= \prod_{P \text{ prime}} \bigg( \frac{2}{1 - \frac{1}{\lvert P \rvert^{s+1}}} - 1 \bigg) \\
= &\prod_{P \text{ prime}} \frac{1 - \frac{1}{\lvert P \rvert^{2(s+1)}}}{\Big( 1 - \frac{1}{\lvert P \rvert^{s+1}} \Big)^2} 
= \frac{\zeta (s+1)^2}{\zeta (2s+2)}
= \bigg( \sum_{n=0}^{\infty} q^{-ns} \bigg)^2 \Big( 1 - q^{-1-2s} \Big) . 
\end{align*}
By comparing the coefficients of powers of $q^{-s}$, we see that

\begin{align*}
\sum_{\substack{N \in \mathcal{M} \\ \degree N \leq x }} \frac{2^{\omega(N)}}{\lvert N \rvert}
= \bigg( \sum_{n=0}^{x} n+1 \bigg) - \frac{1}{q} \bigg( \sum_{n=2}^{x} n-1 \bigg)
= \frac{q-1}{2q} x^2 + \frac{3q+1}{2q} x + 1 .
\end{align*}
\end{proof}

\begin{lemma} \label{Lemma, Sum of 2^omega(N) / |N| with (N,R)=1}
Let $R \in \mathcal{M}$. We have that

\begin{align*}
\sum_{\substack{N \in \mathcal{M} \\ \degree N \leq \degree R \\ (N,R)=1 }} \frac{2^{\omega(N)}}{\lvert N \rvert}
\ll \bigg( \prod_{P \mid R} \frac{1}{1 + 2 \lvert P \rvert^{-1}} \bigg) \big( \degree R \big)^2
\asymp \bigg( \prod_{P \mid R} \frac{1- \lvert P \rvert^{-1}}{1 + \lvert P \rvert^{-1}} \bigg) \big( \degree R \big)^2 .
\end{align*}
\end{lemma}

\begin{proof}
We have that

\begin{align*}
\bigg( \sum_{\substack{N \in \mathcal{M} \\ \degree N \leq \degree R \\ (N,R)=1 }} \frac{2^{\omega(N)}}{\lvert N \rvert} \bigg) \bigg( \sum_{E \mid R} \frac{2^{\omega(E)}}{\lvert E \rvert} \bigg)
\leq \sum_{\substack{N \in \mathcal{M} \\ \degree N \leq 2 \degree R }} \frac{2^{\omega(N)}}{\lvert N \rvert}
\ll \big( \degree R \big)^2 .
\end{align*}
where the last relations follows from Lemma \ref{Proposition Sum of 2^omega(N) / |N|}. We also note that

\begin{align*}
\sum_{E \mid R} \frac{2^{\omega(E)}}{\lvert E \rvert}
\geq \sum_{E \mid R} \frac{\mu (E)^2 2^{\omega(E)}}{\lvert E \rvert}
= \prod_{P \mid R} 1 +\frac{2}{\lvert P \rvert}.
\end{align*}
This proves the first relation in the lemma. The second relation follows from Lemma \ref{remark, asymptotic equivalences of powers of phi(R)/R}.
\end{proof}

\begin{lemma} \label{Double divisor sum}
Let $F,K \in \mathcal{M}$, $x \geq 0$, and $a \in \big( \mathbb{F}_q \big)^*$. Suppose also that $\frac{1}{2} x < \degree KF \leq \frac{3}{4} x$. Then,

\begin{align*}
\sum_{\substack{N \in \mathcal{M} \\ \degree N = x - \degree KF \\ (N,F) = 1}} d(N) d(KF+ aN)
\ll q^x x^2 \frac{1}{\lvert KF \rvert} \sum_{\substack{H \mid K \\ \degree H \leq \frac{x-\degree KF}{2}}} \frac{d(H)}{\lvert H \rvert} .
\end{align*}
\end{lemma}

\begin{proof}
We have that,

\begin{align*}
&\sum_{\substack{N \in \mathcal{M} \\ \degree N = x - \degree KF \\ (N,F) = 1}} d(N) d(KF+ aN) \\
\leq 2 &\sum_{\substack{N \in \mathcal{M} \\ \degree N = x - \degree KF \\ (N,F) = 1}} \sum_{\substack{G \mid N \\ \degree G \leq \frac{x-\degree KF}{2}}} d(KF+ aN)\\
\ll &\sum_{\substack{ G \in \mathcal{M} \\ \degree G \leq \frac{x-\degree KF}{2} \\ (G,F) = 1}} \; \sum_{\substack{N \in \mathcal{M} \\ \degree N = x - \degree KF \\ G \mid N}} d(KF+ aN) \\
= &\sum_{\substack{H \mid K \\ \degree H \leq \frac{x-\degree KF}{2}}} \; \sum_{\substack{ G \in \mathcal{M} \\ \degree G \leq \frac{x-\degree KF}{2} \\ (G,F) = 1 \\ (G,K)=H}} \; \sum_{\substack{N \in \mathcal{M} \\ \degree N = x - \degree KF \\ G \mid N}} d(KF+ aN) \\
= &\sum_{\substack{H \mid K \\ \degree H \leq \frac{x-\degree KF}{2}}} \; \sum_{\substack{ G \in \mathcal{M} \\ \degree G \leq \frac{x-\degree KF}{2} \\ (G,F) = 1 \\ (G,K)=H}} \; \sum_{\substack{N' \in \mathcal{M} \\ \degree N' = x - \degree KF - \degree H \\ G' \mid N'}} d(HK'F+ aHN')
\end{align*}
where $N' , G', K'$ are defined by $HN'=N , HG'=G, HK'=K$. Continuing, we have that

\begin{align*}
&\sum_{\substack{N \in \mathcal{M} \\ \degree N \leq x - \degree KF \\ (N,F) = 1}} d(N) d(KF+ aN) \\
\ll &\sum_{\substack{H \mid K \\ \degree H \leq \frac{x-\degree KF}{2}}} d(H) \sum_{\substack{ G \in \mathcal{M} \\ \degree G \leq \frac{x-\degree KF}{2} \\ (G,F) = 1 \\ (G,K)=H}} \; \sum_{\substack{N' \in \mathcal{M} \\ \degree N' = x - \degree KF - \degree H \\ G' \mid N'}} d(K'F+ aN') \\
\leq &\sum_{\substack{H \mid K \\ \degree H \leq \frac{x-\degree KF}{2}}} d(H) \sum_{\substack{ G \in \mathcal{M} \\ \degree G \leq \frac{x-\degree KF}{2} \\ (G,F) = 1 \\ (G,K)=H}} \; \sum_{\substack{M \in \mathcal{M} \\ \degree (M-K'F) = x - \degree KF - \degree H \\ (M-K'F) \in a \mathcal{M} \\ M \equiv K' F (\modulus G')}} d(M) \\
\ll &q^x x \frac{1}{\lvert K F \rvert} \sum_{\substack{H \mid K \\ \degree H \leq \frac{x-\degree KF}{2}}} \frac{d(H)}{\lvert H \rvert} \sum_{\substack{ G \in \mathcal{M} \\ \degree G \leq \frac{x-\degree KF}{2} \\ (G,F) = 1 \\ (G,K)=H}} \frac{1}{\phi (G')}\\
\ll &q^x x^2 \frac{1}{\lvert K F \rvert} \sum_{\substack{H \mid K \\ \degree H \leq \frac{x-\degree KF}{2}}} \frac{d(H)}{\lvert H \rvert} .
\end{align*}
The third relation holds by Theorem \ref{Brun-Titschmarsh theorem, divisor function case in function field, y equality version} with $\beta = \frac{1}{6}$ and $\alpha = \frac{1}{4}$ (one may wish to note that $(K'F,G')=1$ and that the other conditions of the theorem are satisfied because $\frac{1}{2} x < \degree KF \leq \frac{3}{4} x$). The last relation follows from Lemma \ref{Lemma, sum of 1/phi(N) over monic N with degree leq x}.
\end{proof}

\begin{lemma} \label{Double divisor sum, number 2}
Let $F,K \in \mathcal{M}$ and $x \geq 0$ satisfy $\degree KF < x$. Then,

\begin{align*}
\sum_{\substack{N \in \mathcal{M} \\ \degree N = x \\ (N,F) = 1}} d(N) d(KF+N)
\ll q^x x^2  \sum_{\substack{H \mid K \\ \degree H \leq \frac{x}{2}}} \frac{d(H)}{\lvert H \rvert} .
\end{align*}
\end{lemma}

\begin{proof}
The proof is similar to the proof of Lemma \ref{Double divisor sum}. We have that

\begin{align*}
\sum_{\substack{N \in \mathcal{M} \\ \degree N = x \\ (N,F) = 1}} d(N) d(KF+N)
\leq &2 \sum_{\substack{N \in \mathcal{M} \\ \degree N = x \\ (N,F) = 1}} \sum_{\substack{G \mid N \\ \degree G \leq \frac{x}{2}}} d(KF+N)
\ll \sum_{\substack{G \in \mathcal{M} \\ \degree G \leq \frac{x}{2} \\ (G,F) =1}} \sum_{\substack{N \in \mathcal{M} \\ \degree N = x \\ G \mid N}} d(KF+N) \\
= &\sum_{\substack{H \mid K \\ \degree H \leq \frac{x}{2}}} \sum_{\substack{G \in \mathcal{M} \\ \degree G \leq \frac{x}{2} \\ (G,F) =1 \\ (G,K) = H}} \sum_{\substack{N \in \mathcal{M} \\ \degree N = x \\ G \mid N}} d(KF+N) \\
= &\sum_{\substack{H \mid K \\ \degree H \leq \frac{x}{2}}} \sum_{\substack{G \in \mathcal{M} \\ \degree G \leq \frac{x}{2} \\ (G,F) =1 \\ (G,K) = H}} \sum_{\substack{N' \in \mathcal{M} \\ \degree N' = x-\degree H \\ G' \mid N'}} d(HK'F+HN') ,
\end{align*}
where $N' , G', K'$ are defined by $HN'=N , HG'=G, HK'=K$. Continuing, we have that

\begin{align*}
\sum_{\substack{N \in \mathcal{M} \\ \degree N = x \\ (N,F) = 1}} d(N) d(KF+N)
\ll &\sum_{\substack{H \mid K \\ \degree H \leq \frac{x}{2}}} d(H) \sum_{\substack{G \in \mathcal{M} \\ \degree G \leq \frac{x}{2} \\ (G,F) =1 \\ (G,K) = H}} \sum_{\substack{N' \in \mathcal{M} \\ \degree N' = x-\degree H \\ G' \mid N'}} d(K'F+N') \\
\leq &\sum_{\substack{H \mid K \\ \degree H \leq \frac{x}{2}}} d(H) \sum_{\substack{G \in \mathcal{M} \\ \degree G \leq \frac{x}{2} \\ (G,F) =1 \\ (G,K) = H}} \sum_{\substack{M \in \mathcal{M} \\ \degree (M-X) < x - \degree H \\ M \equiv K'F (\modulus G')}} d(M) ,
\end{align*}
where we define $X := T^{x - \degree H}$. We can now apply Theorem \ref{Brun-Titschmarsh theorem, divisor function case in function field} to obtain that

\begin{align*}
\sum_{\substack{N \in \mathcal{M} \\ \degree N = x \\ (N,F) = 1}} d(N) d(KF+N)
\ll &q^x x \sum_{\substack{H \mid K \\ \degree H \leq \frac{x}{2}}} \frac{d(H)}{\lvert H \rvert} \sum_{\substack{G \in \mathcal{M} \\ \degree G \leq \frac{x}{2} \\ (G,F) =1 \\ (G,K) = H}} \frac{1}{\phi (G')}
\ll &q^x x^2 \sum_{\substack{H \mid K \\ \degree H \leq \frac{x}{2}}} \frac{d(H)}{\lvert H \rvert} .
\end{align*}
\end{proof}

\begin{lemma} \label{Lemma, off diagonal z_1 z_2 sum}
Let $F \in \mathcal{M}$ and $z_1 , z_2$ be non-negative integers. Then, for all $\epsilon > 0$ we have that

\begin{align*}
\sum_{\substack{A,B,C,D \in \mathcal{M} \\ \degree AB = z_1 \\ \degree CD = z_2 \\ (ABCD,F)=1 \\ AC \equiv BD (\modulus F) \\ AC \neq BD}} 1
\begin{cases}
\ll_{\epsilon} \frac{1}{\lvert F \rvert} \Big( q^{z_1} q^{z_2} \Big)^{1+\epsilon} &\text{ if $z_1 + z_2 \leq \frac{19}{10} \degree F$} \\
\ll \frac{1}{\phi (F)} q^{z_1} q^{z_2} (z_1 + z_2)^3 &\text{ if $z_1 + z_2 > \frac{19}{10} \degree F$} .
\end{cases}
\end{align*}
\end{lemma}

\begin{proof}
We can split the sum into the cases $\degree AC > \degree BD$, $\degree AC < \degree BD$, and $\degree AC = \degree BD$ with $AC \neq BD$. The first two cases are identical by symmetry. \\

\noindent When $\degree AC > \degree BD$, we have that $AC = KF + BD$ where $K \in \mathcal{M}$ and $\degree KF > \degree BD$. Furthermore,

\begin{align*}
2 \degree KF = 2 \degree AC > \degree AC + \degree BD = \degree AB + \degree CD = z_1 + z_2,
\end{align*}
from which we deduce that $\frac{z_1 + z_2}{2} < \degree KF \leq z_1 + z_2$; and

\begin{align*}
\degree KF + \degree BD = \degree AC + \degree BD = z_1 + z_2,
\end{align*}
from which we deduce that $\degree BD = z_1 + z_2 - \degree KF$. \\

\noindent When $\degree AC = \degree BD$, we must have that $\degree AC = \degree BD = \frac{z_1 + z_2}{2}$ (in particular, this case applies only when $z_1 + z_2 $ is even). Also, we can write $AC = KF +BD$, where $\degree KF < \degree BD = \frac{z_1 + z_2}{2}$ and $K$ need not be monic. \\

\noindent So, writing $N = BD$, we have that

\begin{align}
\begin{split} \label{z_1 z_2 off diagonal divisor split}
\sum_{\substack{A,B,C,D \in \mathcal{M} \\ \degree AB = z_1 \\ \degree CD = z_2 \\ (ABCD,F)=1 \\ AC \equiv BD (\modulus F) \\ AC \neq BD}} 1
\leq &2 \sum_{\substack{K \in \mathcal{M} \\ \frac{z_1 + z_2}{2} < \degree KF \leq z_1 + z_2}} \sum_{\substack{N \in \mathcal{M} \\ \degree N = z_1 + z_2 - \degree KF \\ (N,F)=1}} d(N) d(KF+N) \\
&+ \sum_{\substack{K \in \mathcal{A} \\ \degree KF < \frac{z_1 + z_2}{2}}} \sum_{\substack{N \in \mathcal{M} \\ \degree N = \frac{z_1 + z_2}{2} \\ (N,F)=1}} d(N) d(KF+N) \\
\end{split}
\end{align}

\textbf{\underline{STEP 1:}} Let us consider the case when $z_1 + z_2 \leq \frac{19}{10} \degree F$. By using well known bounds on the divisor function, we have that

\begin{align*}
&\sum_{\substack{K \in \mathcal{M} \\ \frac{z_1 + z_2}{2} < \degree KF \leq z_1 + z_2}} \sum_{\substack{N \in \mathcal{M} \\ \degree N = z_1 + z_2 - \degree KF \\ (N,F)=1}} d(N) d(KF+N) \\
\ll_{\epsilon} &\Big( q^{z_1} q^{z_2} \Big)^{\frac{\epsilon}{2}} \sum_{\substack{K \in \mathcal{M} \\ \frac{z_1 + z_2}{2} < \degree KF \leq z_1 + z_2}} \sum_{\substack{N \in \mathcal{M} \\ \degree N = z_1 + z_2 - \degree KF \\ (N,F)=1}} 1 \\
\leq &\Big( q^{z_1} q^{z_2} \Big)^{1 + \frac{\epsilon}{2}} \sum_{\substack{K \in \mathcal{M} \\ \frac{z_1 + z_2}{2} < \degree KF \leq z_1 + z_2}} \frac{1}{\lvert KF \rvert} \\
\leq &\Big( q^{z_1} q^{z_2} \Big)^{1 + \frac{\epsilon}{2}} \frac{z_1 + z_2}{\lvert F \rvert} \\
\ll_{\epsilon} &\Big( q^{z_1} q^{z_2} \Big)^{1 + \epsilon} \frac{1}{\lvert F \rvert} .
\end{align*}

\noindent As for the sum

\begin{align*}
\sum_{\substack{K \in \mathcal{A} \\ \degree KF < \frac{z_1 + z_2}{2}}} \sum_{\substack{N \in \mathcal{M} \\ \degree N = \frac{z_1 + z_2}{2} \\ (N,F)=1}} d(N) d(KF+N) ,
\end{align*}
we note that it does not apply to this case where $z_1 + z_2 \leq \frac{19}{10} \degree F$ because $\degree KF \geq \degree F \geq \frac{20}{19} \frac{z_1 + z_2}{2}$, which does not overlap with range $\degree KF < \frac{z_1 + z_2}{2}$ in the sum. \\

\noindent Hence,

\begin{align*}
\sum_{\substack{A,B,C,D \in \mathcal{M} \\ \degree AB = z_1 \\ \degree CD = z_2 \\ (ABCD,F)=1 \\ AC \equiv BD (\modulus F) \\ AC \neq BD}} 1 
\ll_{\epsilon} \Big( q^{z_1} q^{z_2} \Big)^{1 + \epsilon} \frac{1}{\lvert F \rvert} .
\end{align*}

\noindent \textbf{\underline{STEP 2:}} We now consider the case when $z_1 + z_2 > \frac{19}{10} \degree F$.\\

\noindent \textbf{\underline{STEP 2.1:}} We consider the subcase where $\frac{z_1 + z_2}{2} < \degree KF \leq \frac{3(z_1 + z_2)}{4}$. This allows us to apply Lemma \ref{Double divisor sum} for the first relation below.

\begin{align*}
&\sum_{\substack{K \in \mathcal{M} \\ \frac{z_1 + z_2}{2} < \degree KF \leq \frac{3(z_1 + z_2)}{4}}} \sum_{\substack{N \in \mathcal{M} \\ \degree N = z_1 + z_2 - \degree KF \\ (N,F)=1}} d(N) d(KF+N) \\
\ll &q^{z_1} q^{z_2} (z_1 + z_2)^2 \frac{1}{\lvert F \rvert} \sum_{\substack{K \in \mathcal{M} \\ \frac{z_1 + z_2}{2} < \degree KF \leq \frac{3(z_1 + z_2)}{4}}} \frac{1}{\lvert K \rvert} \sum_{\substack{H \mid K \\ \degree H \leq \frac{z_1 + z_2 - \degree KF}{2}}} \frac{d(H)}{\lvert H \rvert} \\
\leq &q^{z_1} q^{z_2} (z_1 + z_2)^2 \frac{1}{\lvert F \rvert} \sum_{\substack{K \in \mathcal{M} \\ \degree K \leq z_1 + z_2}} \frac{1}{\lvert K \rvert} \sum_{H \mid K} \frac{d(H)}{\lvert H \rvert} \\
= &q^{z_1} q^{z_2} (z_1 + z_2)^2 \frac{1}{\lvert F \rvert} \sum_{\substack{H \in \mathcal{M} \\ \degree H \leq z_1 + z_2}} \frac{d(H)}{\lvert H \rvert} \sum_{\substack{K \in \mathcal{A} \\ \degree K \leq z_1 + z_2 \\ H \mid K}} \frac{1}{\lvert K \rvert} \\
\leq &q^{z_1} q^{z_2} (z_1 + z_2)^3 \frac{1}{\lvert F \rvert} \sum_{\substack{H \in \mathcal{M} \\ \degree H \leq z_1 + z_2}} \frac{d(H)}{\lvert H \rvert^2} \\
\ll &q^{z_1} q^{z_2} (z_1 + z_2)^3 \frac{1}{\lvert F \rvert} .
\end{align*}

\noindent \textbf{\underline{STEP 2.2:}} Now we consider the subcase where $\frac{3(z_1 + z_2)}{4} < \degree KF \leq z_1 + z_2$. We have that

\begin{align*}
&\sum_{\substack{K \in \mathcal{M} \\ \frac{3(z_1 + z_2)}{4} < \degree KF \leq z_1 + z_2}} \sum_{\substack{N \in \mathcal{M} \\ \degree N = z_1 + z_2 - \degree KF \\ (N,F)=1}} d(N) d(KF+N) \\
= &\sum_{\substack{N \in \mathcal{M} \\ \degree N < \frac{z_1 + z_2}{4} \\ (N,F)=1}} \sum_{\substack{K \in \mathcal{M} \\ \degree KF = z_1 + z_2 - \degree N}} d(N) d(KF+N) \\
\leq &\sum_{\substack{N \in \mathcal{M} \\ \degree N < \frac{z_1 + z_2}{4} \\ (N,F)=1}} d(N) \sum_{\substack{M \in \mathcal{M} \\ \degree (M - X_{(N)} ) < z_1 + z_2 - \degree N \\ M \equiv N (\modulus F)}} d(M)
\end{align*}
where we define $X_{(N)} = T^{z_1 + z_2 - \degree N}$. We can now apply Theorem \ref{Brun-Titschmarsh theorem, divisor function case in function field}. One may wish to note that

\begin{align*}
y = z_1 + z_2 - \degree N \geq \frac{3}{4}(z_1 +z_2) \geq \frac{3}{4} \frac{19}{10} \degree F
\end{align*}
and so

\begin{align*}
\degree F \leq \frac{40}{57} y = (1- \alpha) y
\end{align*}
where $0 < \alpha < \frac{1}{2}$, as required. Hence, we have that

\begin{align*}
&\sum_{\substack{K \in \mathcal{M} \\ \frac{3(z_1 + z_2)}{4} < \degree KF \leq z_1 + z_2}} \sum_{\substack{N \in \mathcal{M} \\ \degree N = z_1 + z_2 - \degree KF \\ (N,F)=1}} d(N) d(KF+N) \\
\ll &q^{z_1} q^{z_2} (z_1 + z_2) \frac{1}{\phi (F)} \sum_{\substack{N \in \mathcal{M} \\ \degree N \leq \frac{z_1 + z_2}{4} \\ (N,F)=1}} \frac{d(N)}{\lvert N \rvert} \\
\leq &q^{z_1} q^{z_2} (z_1 + z_2) \frac{1}{\phi (F)} \bigg( \sum_{\substack{N \in \mathcal{M} \\ \degree N \leq z_1 + z_2}} \frac{1}{\lvert N \rvert} \bigg)^2 \\
\ll &q^{z_1} q^{z_2} (z_1 + z_2)^3 \frac{1}{\phi (F)} . \\
\end{align*}

\noindent \textbf{\underline{STEP 2.3:}} We now look at the sum

\begin{align*}
\sum_{\substack{K \in \mathcal{A} \\ \degree KF < \frac{z_1 + z_2}{2}}} \sum_{\substack{N \in \mathcal{M} \\ \degree N = \frac{z_1 + z_2}{2} \\ (N,F)=1}} d(N) d(KF+N) .
\end{align*}
By Lemma \ref{Double divisor sum, number 2} we have that

\begin{align*}
\sum_{\substack{K \in \mathcal{A} \\ \degree KF < \frac{z_1 + z_2}{2}}} \sum_{\substack{N \in \mathcal{M} \\ \degree N = \frac{z_1 + z_2}{2} \\ (N,F)=1}} d(N) d(KF+N)
\ll &q^{\frac{z_1 + z_2 }{2}} (z_1 + z_2 )^2 \sum_{\substack{K \in \mathcal{A} \\ \degree KF < \frac{z_1 + z_2}{2}}} \sum_{H \mid K} \frac{d(H)}{\lvert H \rvert} \\
\leq &q^{z_1 + z_2 -1} (z_1 + z_2 )^2 \frac{1}{\lvert F \rvert} \sum_{\substack{K \in \mathcal{A} \\ \degree KF < \frac{z_1 + z_2}{2}}} \frac{1}{\lvert K \rvert} \sum_{H \mid K} \frac{d(H)}{\lvert H \rvert} \\
\leq &q^{z_1 + z_2 } (z_1 + z_2 )^3 \frac{1}{\lvert F \rvert} ,
\end{align*}
where the last relation uses a similar calculation as that in Step 2.1. \\

\noindent \textbf{\underline{STEP 2.4:}} We apply steps 2.1, 2.2, and 2.3 to (\ref{z_1 z_2 off diagonal divisor split}) and we see that

\begin{align*}
\sum_{\substack{A,B,C,D \in \mathcal{M} \\ \degree AB = z_1 \\ \degree CD = z_2 \\ (ABCD,F)=1 \\ AC \equiv BD (\modulus F) \\ AC \neq BD}} 1 
\ll q^{z_1} q^{z_2} (z_1 + z_2)^3 \frac{1}{\phi (F)}
\end{align*}
for $z_1 + z_2 \geq \frac{19}{10} \degree F$.
\end{proof}

\noindent In fact, we can prove the following, more general Lemma:

\begin{lemma} \label{Lemma, off diagonal z_1 z_2 sum with alpha}
Let $F \in \mathcal{M}$, $z_1 , z_2$ be non-negative integers, and let $a \in \big( \mathbb{F}_q \big)^*$. Then, for all $\epsilon > 0$ we have that

\begin{align*}
\sum_{\substack{A,B,C,D \in \mathcal{M} \\ \degree AB = z_1 \\ \degree CD = z_2 \\ (ABCD,F)=1 \\ AC \equiv a BD (\modulus F) \\ AC \neq BD}} 1
\begin{cases}
\ll_{\epsilon} \frac{1}{\lvert F \rvert} \Big( q^{z_1} q^{z_2} \Big)^{1+\epsilon} &\text{ if $z_1 + z_2 \leq \frac{19}{10} \degree F$} \\
\ll \frac{1}{\phi (F)} q^{z_1} q^{z_2} (z_1 + z_2)^3 &\text{ if $z_1 + z_2 > \frac{19}{10} \degree F$} ,
\end{cases}
\end{align*}
\end{lemma}

\begin{proof}
The case where $a=1$ is just Lemma \ref{Lemma, off diagonal z_1 z_2 sum}. The proof of the case where $a \neq 1$  is very similar to the proof of Lemma \ref{Lemma, off diagonal z_1 z_2 sum}. In fact it is easier, because the the case where $\degree AC = \degree BD$ cannot exist: We would require that $AC$ and $BD$ are both monic, but also require that at least one of $AC$ and $BD$ have leading coefficient equal to $a \neq 1$.
\end{proof}

\begin{proposition} \label{Proposition, off diagonal leq z_R terms}
Let $R \in \mathcal{M}$ and define $z_R := \degree R - \log_q 2^{\omega (R)}$. Also, let $a \in \big( \mathbb{F}_q \big)^*$. Then, 

\begin{align*}
\sum_{EF=R} \mu (E) \phi (F) \sum_{\substack{A,B,C,D \in \mathcal{M} \\ \degree AB \leq z_R \\ \degree CD \leq z_R \\ (ABCD,R)=1 \\ AC \equiv aBD (\modulus F) \\ AC \neq BD}} \frac{1}{\lvert ABCD \rvert^{\frac{1}{2}}}
\ll \lvert R \rvert \big( \degree R \big)^3 .
\end{align*}
\end{proposition}

\begin{proof}
We apply Lemma \ref{Lemma, off diagonal z_1 z_2 sum with alpha} with $\epsilon = \frac{1}{50}$ to deduce that

\begin{align*}
\sum_{\substack{A,B,C,D \in \mathcal{M} \\ \degree AB \leq z_R \\ \degree CD \leq z_R \\ (ABCD,R)=1 \\ AC \equiv aBD (\modulus F) \\ AC \neq BD}} \frac{1}{\lvert ABCD \rvert^{\frac{1}{2}}}
\ll &\frac{1}{\lvert F \rvert} \sum_{\substack{z_1 , z_2 \leq z_R \\ z_1 + z_2 \leq \frac{19}{10} \degree F}} \Big( q^{z_1} q^{z_2} \Big)^{\frac{1}{2} + \epsilon}
+ \frac{1}{\phi (F)} \sum_{\substack{z_1 , z_2 \leq z_R \\ \frac{19}{10} \degree F < z_1 + z_2 \leq 2 \degree R}} q^{\frac{z_1}{2}} q^{\frac{z_2}{2}} (z_1 + z_2)^3 \\
\ll &\frac{1}{\lvert F \rvert^{\frac{1}{20} - 2 \epsilon}}
+ \frac{1}{\phi (F)} \big( \degree R \big)^3 \sum_{z_1 < z_R} \sum_{z_2 < z_R} q^{\frac{z_1}{2}} q^{\frac{z_2}{2}} \\
\ll &\frac{1}{\lvert F \rvert^{\frac{1}{20} - 2 \epsilon}}
+ \frac{1}{\lvert F \rvert} q^{z_R} \big( \degree R \big)^3 .
\end{align*}
So,

\begin{align*}
\sum_{EF=R} \mu (E) \phi (F) \sum_{\substack{A,B,C,D \in \mathcal{M} \\ \degree AB \leq z_R \\ \degree CD \leq z_R \\ (ABCD,R)=1 \\ AC \equiv BD (\modulus F) \\ AC \neq BD}} \frac{1}{\lvert ABCD \rvert^{\frac{1}{2}}}
\ll & q^{z_R} \big( \degree R \big)^3 \sum_{EF=R} \lvert \mu (E) \rvert + \sum_{EF=R} \lvert \mu (E) \rvert \frac{\phi (F)}{\lvert F \rvert^{\frac{1}{20} - 2 \epsilon}} \\
\ll &q^{z_R} \big( \degree R \big)^3 2^{\omega (R)} + \lvert R \rvert \\
\ll &\lvert R \rvert \big( \degree R \big)^3 ,
\end{align*}
where the second-to-last relation uses the following:

\begin{align*}
\sum_{EF=R} \lvert \mu (E) \rvert \frac{\phi (F)}{\lvert F \rvert^{\frac{1}{20} - 2 \epsilon}}
\leq &\sum_{EF=R} \lvert \mu (E) \rvert \phi (F)
= \phi (R) \sum_{EF=R} \lvert \mu (E) \rvert \bigg( \prod_{\substack{P \mid E \\ P^2 \mid R}} \frac{1}{\lvert P \rvert} \bigg) \bigg( \prod_{\substack{P \mid E \\ P^2 \nmid R}} \frac{1}{\lvert P \rvert -1} \bigg) \\
\leq & \phi (R) \sum_{EF=R} \lvert \mu (E) \rvert \prod_{P \mid E} \frac{1}{\lvert P \rvert -1}
= \phi (R) \prod_{P \mid R} 1 + \frac{1}{\lvert P \rvert -1}
= \phi (R) \frac{\lvert R \rvert}{\phi (R)}
= \lvert R \rvert.
\end{align*}
\end{proof}


\section{The Fourth Moment}

\noindent We now proceed to prove Theorem \ref{Fourth primitive moment function fields statement}. In the proof we implicitly state that some terms are of lower order than the main term and that is easy to check. We do not give the justification explicitly, although all the results one needs for a rigorous justification are given in Section \ref{Section, Results on Well-known Multiplicative Functions}.

\begin{proof}[Proof of Theorem \ref{Fourth primitive moment function fields statement}]
Let $\chi$ be a Dirichlet character of modulus $R$. By Propositions \ref{Proposition, short sum for odd squared L-function} and \ref{Proposition, short sum for even squared L-function}, we have that

\begin{align*}
\Big\lvert L \Big( \frac{1}{2} , \chi \Big) \Big\lvert^2
= 2 \sum_{\substack{A,B \in \mathcal{M} \\ \degree AB < \degree R}} \frac{\chi (A) \conj\chi (B)}{\lvert AB \rvert^{\frac{1}{2}}} + c (\chi ) 
= 2 a(\chi ) + 2 b(\chi ) + c (\chi ) ,
\end{align*}
where

\begin{align*}
z_R := &\degree R - \log_q (2^{\omega (Q)}) , \\
a (\chi) := &\sum_{\substack{A,B \in \mathcal{M} \\ \degree AB \leq z_R }} \frac{\chi (A) \conj\chi (B)}{\lvert AB \rvert^{\frac{1}{2}}} , \\
b (\chi) := &\sum_{\substack{A,B \in \mathcal{M} \\  z_R < \degree AB < \degree R}} \frac{\chi (A) \conj\chi (B)}{\lvert AB \rvert^{\frac{1}{2}}} ,
\end{align*}
and $c(\chi )$ is defined as in (\ref{Definition of c (chi)}). Then,

\begin{align*}
\sumstar_{\chi \modulus R} \Big\lvert L \Big( \frac{1}{2} , \chi \Big) \Big\lvert^4
= \sumstar_{\chi \modulus R} \Big( 2 a(\chi ) + 2 b(\chi ) + c (\chi ) \Big)^2 .
\end{align*}
We will show that $\sumstar_{\chi \modulus R} \lvert a (\chi ) \rvert^2$ has an asymptotic main term of higher order than $\sumstar_{\chi \modulus R} \lvert b (\chi ) \rvert^2$ and $\sumstar_{\chi \modulus R} \lvert c (\chi ) \rvert^2$. From this and the Cauchy-Schwarz inequality, we deduce that $\sumstar_{\chi \modulus R} \lvert a (\chi ) \rvert^2$ gives the leading term in the asymptotic formula. \\

\noindent \textbf{\underline{STEP 1:}} We have that

\begin{align*}
\sumstar_{\chi \modulus R} \lvert a (\chi ) \rvert^2
= &\sumstar_{\chi \modulus R} \sum_{\substack{A,B,C,D \in \mathcal{M} \\ \degree AB \leq z_R \\  \degree CD \leq z_R}} \frac{\chi (AC) \conj\chi (BD)}{\lvert ABCD \rvert^{\frac{1}{2}}}
= \sum_{\substack{A,B,C,D \in \mathcal{M} \\ \degree AB \leq z_R \\  \degree CD \leq z_R }} \frac{1}{\lvert ABCD \rvert^{\frac{1}{2}}} \sumstar_{\chi \modulus R} \chi (AC) \conj\chi (BD)\\
= &\sum_{\substack{A,B,C,D \in \mathcal{M} \\ \degree AB \leq z_R \\  \degree CD \leq z_R \\ (ABCD,R) = 1}} \frac{1}{\lvert ABCD \rvert^{\frac{1}{2}}} \sum_{\substack{ EF = R \\ F \mid (AC - BD)}} \mu (E) \phi (F) ,
\end{align*}
where the last equality follows from Proposition \ref{Primitive chracter sum, mobius inversion}. Continuing,

\begin{align}
\begin{split} \label{alpha chi primitive sum split into diagonal and off diagonal terms}
&\sumstar_{\chi \modulus R} \lvert a (\chi ) \rvert^2 = \sum_{EF = R} \mu (E) \phi (F) \sum_{\substack{A,B,C,D \in \mathcal{M} \\ \degree AB \leq z_R \\  \degree CD \leq z_R \\ (ABCD,R) = 1 \\  F \mid (AC - BD)}} \frac{1}{\lvert ABCD \rvert^{\frac{1}{2}}} \\
= &\sum_{EF = R} \mu (E) \phi (F) \sum_{\substack{A,B,C,D \in \mathcal{M} \\ \degree AB \leq z_R \\  \degree CD \leq z_R \\ (ABCD,R) = 1 \\  F \mid (AC - BD) \\ AC = BD}} \frac{1}{\lvert ABCD \rvert^{\frac{1}{2}}}
+ \sum_{EF = R} \mu (E) \phi (F) \sum_{\substack{A,B,C,D \in \mathcal{M} \\ \degree AB \leq z_R \\  \degree CD \leq z_R \\ (ABCD,R) = 1 \\  F \mid (AC - BD) \\ AC \neq BD}} \frac{1}{\lvert ABCD \rvert^{\frac{1}{2}}} \\
= &\bigg( \sum_{EF = R} \mu (E) \phi (F) \bigg) \bigg( \sum_{\substack{A,B,C,D \in \mathcal{M} \\ \degree AB \leq z_R \\  \degree CD \leq z_R \\ (ABCD,R) = 1 \\ AC = BD}} \frac{1}{\lvert ABCD \rvert^{\frac{1}{2}}} \bigg)
+ \sum_{EF = R} \mu (E) \phi (F) \sum_{\substack{A,B,C,D \in \mathcal{M} \\ \degree AB \leq z_R \\  \degree CD \leq z_R \\ (ABCD,R) = 1 \\  F \mid (AC - BD) \\ AC \neq BD}} \frac{1}{\lvert ABCD \rvert^{\frac{1}{2}}}
\end{split}
\end{align}

\noindent \textbf{\underline{STEP 1.1:}} We will look at the first term on the far-RHS. Since $AC = BD$, we can write $A = FU , B = FV , C = GV , D = GU$, where $F,G,U,V$ are monic and $U,V$ are coprime. Let us write $N = UV$, and note that there are $2^{\omega (N)}$ ways of writing $N = UV$ with $U,V$ being coprime. Then,

\begin{align}
\begin{split} \label{Sum of A,B,C,D monic , deg AB, deg CD < Z , ABCD coprime to Q , AC=BD to FGRS}
&\sum_{\substack{A,B,C,D \in \mathcal{M} \\ \degree AB \leq z_R \\  \degree CD \leq z_R \\ (ABCD,R) = 1 \\ AC = BD}} \frac{1}{\lvert ABCD \rvert^{\frac{1}{2}}} \\
= &\sum_{\substack{F,G,U,V \in \mathcal{M} \\ (U,V)=1 \\ \degree F^2 UV \leq z_R \\  \degree G^2 UV \leq z_R \\ (FGUV,R) = 1 }} \frac{1}{\lvert FGUV \rvert}
= \sum_{\substack{N \in \mathcal{M} \\ \degree N \leq z_R \\ (N,R) = 1}} \frac{2^{\omega(N)}}{\lvert N \rvert} \bigg( \sum_{\substack{ F \in \mathcal{M} \\ \degree F \leq \frac{z_R -\degree N}{2} \\ (F , R) = 1}} \frac{1}{\lvert F \rvert} \bigg)^2 \\
= &\sum_{\substack{N \in \mathcal{M} \\ \degree N \leq {z_R}' \\ (N,R) = 1}} \frac{2^{\omega(N)}}{\lvert N \rvert} \bigg( \sum_{\substack{ F \in \mathcal{M} \\ \degree F \leq \frac{z_R -\degree N}{2} \\ (F , R) = 1}} \frac{1}{\lvert F \rvert} \bigg)^2
+ \sum_{\substack{N \in \mathcal{M} \\ {z_R}' < \degree N \leq z_R \\ (N,R) = 1}} \frac{2^{\omega(N)}}{\lvert N \rvert} \bigg( \sum_{\substack{ F \in \mathcal{M} \\ \degree F \leq \frac{z_R -\degree N}{2} \\ (F , R) = 1}} \frac{1}{\lvert F \rvert} \bigg)^2 ,
\end{split}
\end{align}
where ${z_R}' := \degree R - \log_q 9^{\omega (R)}$.

Let us look at the first term on the far-RHS of (\ref{Sum of A,B,C,D monic , deg AB, deg CD < Z , ABCD coprime to Q , AC=BD to FGRS}). We apply Proposition \ref{Sum over (A,R)=1, deg A <= x of 1/A}. When $x = \frac{z_R - \degree N}{2}$ and $\degree N \leq {z_R}'$, we have that $\frac{2^{\omega (R)} x}{q^x} = O(1)$. Hence,

\begin{align}
\begin{split} \label{degree N leq z_R' sum}
&\sum_{\substack{N \in \mathcal{M} \\ \degree N \leq {z_R}' \\ (N,R) = 1}} \frac{2^{\omega(N)}}{\lvert N \rvert} \bigg( \sum_{\substack{ F \in \mathcal{M} \\ \degree F \leq \frac{z_R -\degree N}{2} \\ (F , R) = 1}} \frac{1}{\lvert F \rvert} \bigg)^2 \\
= &\bigg( \frac{\Phi (R)}{2 \lvert R \rvert} \bigg)^2 \sum_{\substack{N \in \mathcal{M}\\ \degree N \leq {z_R}' \\ (N,R) = 1}} \frac{2^{\omega(N)}}{\lvert N \rvert} \Bigg( {z_R}' - \degree N + O \big( \log \omega (R) \big) \Bigg)^2 \\
= &\bigg( \frac{\Phi (R)}{2 \lvert R \rvert} \bigg)^2 \sum_{\substack{N \in \mathcal{M} \\ \degree N \leq {z_R}' \\ (N,R) = 1}} \frac{2^{\omega(N)}}{\lvert N \rvert} \Bigg( \big( {z_R}' - \degree N \big)^2 + O \big( \degree R \log \omega (R) \big) \Bigg) \\
= &\frac{1-q^{-1}}{48} \prod_{\substack{P \text{ prime} \\ P \mid R}} \bigg( \frac{ \big( 1 - \lvert P \rvert^{-1} \big)^3}{1 + \lvert P \rvert^{-1}} \bigg) (\degree R)^4 \\
&\quad \quad \quad + O \Bigg( \bigg( \prod_{\substack{P \text{ prime} \\ P \mid R}}  \frac{\big( 1 - \lvert P \rvert^{-1} \big)^3 }{1 + \lvert P \rvert^{-1}} \bigg) \Big( (\degree R)^3 \omega (R)+ (\degree R)^3 \log \degree R \Big) \Bigg) ,
\end{split}
\end{align}
where the last equality follows from Proposition \ref{Proposition Sum of 2^omega(N) log(|Q|/|N|)^2 / |N|} and Lemma \ref{Lemma, Sum of 2^omega(N) / |N| with (N,R)=1}. \\

\noindent Now we look at the second term on the far-RHS of (\ref{Sum of A,B,C,D monic , deg AB, deg CD < Z , ABCD coprime to Q , AC=BD to FGRS}). Because ${z_R}' < \degree N \leq z_R$, we have that $\degree F \leq \log_q \big( \frac{3}{\sqrt{2}} \big)^{\omega (R)} $. Using this and Proposition \ref{Sum over (A,R)=1, deg A <= x of 1/A}, we have that

\begin{align*}
\sum_{\substack{ F \in \mathcal{M} \\ \degree F \leq \frac{z_R -\degree N}{2} \\ (F , R) = 1}} \frac{1}{\lvert F \rvert}
\leq \sum_{\substack{ F \in \mathcal{M} \\ \degree F \leq \log_q \big( \frac{3}{\sqrt{2}}^{\omega (R)} \big) \\ (F , R) = 1}} \frac{1}{\lvert F \rvert}
\ll \frac{\phi (R)}{\lvert R \rvert} \omega (R) .
\end{align*}
Also, by similar means as in Lemma \ref{Proposition Sum of 2^omega(N) / |N|}, we can see that

\begin{align*}
\sum_{\substack{N \in \mathcal{M} \\ {z_R}' \leq \degree N \leq z_R \\ (N,R) = 1}} \frac{2^{\omega(N)}}{\lvert N \rvert}
\leq \sum_{\substack{N \in \mathcal{M} \\ {z_R}' \leq \degree N \leq z_R }} \frac{2^{\omega(N)}}{\lvert N \rvert}
\ll \omega(R) \degree R .
\end{align*}
Hence,

\begin{align} \label{z_R' leq degree N leq z_R' sum}
\sum_{\substack{N \in \mathcal{M} \\ {z_R}' \leq \degree N \leq z_R \\ (N,R) = 1}} \frac{2^{\omega(N)}}{\lvert N \rvert} \bigg( \sum_{\substack{ F \in \mathcal{M} \\ \degree F \leq \frac{z_R -\degree N}{2} \\ (F , R) = 1}} \frac{1}{\lvert F \rvert} \bigg)^2
\ll \bigg( \frac{\phi (R)}{\lvert R \rvert} \bigg)^2 \big( \omega (R) \big)^3 \degree R
\end{align}

\noindent By (\ref{Sum of A,B,C,D monic , deg AB, deg CD < Z , ABCD coprime to Q , AC=BD to FGRS}), (\ref{degree N leq z_R' sum}) and (\ref{z_R' leq degree N leq z_R' sum}), we have that

\begin{align*}
&\bigg( \sum_{EF = R} \mu (E) \phi (F) \bigg) \bigg( \sum_{\substack{A,B,C,D \in \mathcal{M} \\ \degree AB \leq z_R \\  \degree CD \leq z_R \\ (ABCD,R) = 1 \\ AC = BD}} \frac{1}{\lvert ABCD \rvert^{\frac{1}{2}}} \bigg)\\
= &\frac{1-q^{-1}}{48} \phi^* (R) \prod_{\substack{P \text{ prime} \\ P \mid R}} \bigg( \frac{ \big( 1 - \lvert P \rvert^{-1} \big)^3}{1 + \lvert P \rvert^{-1}} \bigg) (\degree R)^4 \\
&\quad \quad \quad + O \Bigg( \phi^* (R) \bigg( \prod_{\substack{P \text{ prime} \\ P \mid R}}  \frac{\big( 1 - \lvert P \rvert^{-1} \big)^3 }{1 + \lvert P \rvert^{-1}} \bigg) \Big( (\degree R)^3 \omega (R)+ (\degree R)^3 \log \degree R \Big) \Bigg) . \\
\end{align*}

\noindent \textbf{\underline{STEP 1.2:}}
For the second term on the far-RHS of (\ref{alpha chi primitive sum split into diagonal and off diagonal terms}) we simply apply Proposition \ref{Proposition, off diagonal leq z_R terms}. From this, Step 1.1, and (\ref{alpha chi primitive sum split into diagonal and off diagonal terms}), we deduce that

\begin{align*}
&\sumstar_{\chi \modulus R} \lvert a (\chi ) \rvert^2 \\
= &\frac{1-q^{-1}}{48} \phi^* (R) \prod_{\substack{P \text{ prime} \\ P \mid R}} \bigg( \frac{ \big( 1 - \lvert P \rvert^{-1} \big)^3}{1 + \lvert P \rvert^{-1}} \bigg) (\degree R)^4 \\
&\quad \quad \quad + O \Bigg( \phi^* (R) \bigg( \prod_{\substack{P \text{ prime} \\ P \mid R}}  \frac{\big( 1 - \lvert P \rvert^{-1} \big)^3 }{1 + \lvert P \rvert^{-1}} \bigg) \Big( (\degree R)^3 \omega (R)+ (\degree R)^3 \log \degree R \Big) \Bigg) . \\
\end{align*}

\noindent \textbf{\underline{STEP 2:}} We will now look at $\sumstar_{\chi (\modulus R)} \lvert b (\chi) \rvert^2$. We have that

\begin{align}
\begin{split} \label{ B(X)^2 sum split into Equal and Nonequal} 
\sumstar_{\chi \modulus R} \lvert b (\chi) \rvert^2
\leq  &\sum_{\chi \modulus R} \lvert b (\chi) \rvert^2
= \phi (R) \sum_{\substack{A,B,C,D \in \mathcal{M} \\ z_R < \degree AB < \degree R \\ z_R < \degree CD <\degree R \\ (ABCD,R) = 1 \\ AC \equiv BD (\modulus R)}} \frac{1}{\lvert ABCD \rvert^{\frac{1}{2}}}\\
= &\phi (R) \sum_{\substack{A,B,C,D \in \mathcal{M} \\ z_R < \degree AB < \degree R \\ z_R < \degree CD <\degree R \\ (ABCD,R) = 1 \\ AC = BD}} \frac{1}{\lvert ABCD \rvert^{\frac{1}{2}}}
+ \phi (R) \sum_{\substack{A,B,C,D \in \mathcal{M} \\ z_R < \degree AB < \degree R \\ z_R < \degree CD <\degree R \\ (ABCD,R) = 1 \\ AC \equiv BD (\modulus R) \\ AC \neq BD}} \frac{1}{\lvert ABCD \rvert^{\frac{1}{2}}}
\end{split}
\end{align}

\noindent \textbf{\underline{STEP 2.1:}} Looking at the first term on the far-RHS, we apply the same technique as in (\ref{Sum of A,B,C,D monic , deg AB, deg CD < Z , ABCD coprime to Q , AC=BD to FGRS}) to obtain

\begin{align}
\begin{split} \label{z_R < deg AB, split with z_R'}
&\phi (R) \sum_{\substack{A,B,C,D \in \mathcal{M} \\ z_R < \degree AB < \degree R \\ z_R < \degree CD <\degree R \\ (ABCD,R) = 1 \\ AC = BD}} \frac{1}{\lvert ABCD \rvert^{\frac{1}{2}}} \\
= &\phi (R) \sum_{\substack{N \text{ monic} \\ \degree N < \degree R \\ (N,R) = 1}} \frac{2^{\omega(N)}}{\lvert N \rvert} \bigg( \sum_{\substack{ F \text{ monic} \\ \frac{z_R -\degree N}{2} < \degree F < \frac{\degree R -\degree N}{2} \\ (F , R) = 1}} \frac{1}{\lvert F \rvert} \bigg)^2 \\
\leq &\phi (R) \sum_{\substack{N \text{ monic} \\ \degree N \leq {z_R}' \\ (N,R) = 1}} \frac{2^{\omega(N)}}{\lvert N \rvert} \bigg( \sum_{\substack{ F \text{ monic} \\ \frac{z_R -\degree N}{2} < \degree F < \frac{\degree R -\degree N}{2} \\ (F , R) = 1}} \frac{1}{\lvert F \rvert} \bigg)^2
+ \phi (R) \sum_{\substack{N \text{ monic} \\ {z_R}' <\degree N < \degree R \\ (N,R) = 1}} \frac{2^{\omega(N)}}{\lvert N \rvert} \bigg( \sum_{\substack{ F \text{ monic} \\ \degree F < \frac{\degree R -\degree N}{2} \\ (F , R) = 1}} \frac{1}{\lvert F \rvert} \bigg)^2 ,
\end{split}
\end{align}
where ${z_R}' := \degree R - \log_q 9^{\omega (R)}$.\\

\noindent We look at the first term on the far-RHS:

\begin{align*}
&\phi (R) \sum_{\substack{N \text{ monic} \\ \degree N \leq {z_R}' \\ (N,R) = 1}} \frac{2^{\omega(N)}}{\lvert N \rvert} \bigg( \sum_{\substack{ F \text{ monic} \\ \frac{z_R -\degree N}{2} < \degree F < \frac{\degree R -\degree N}{2} \\ (F , R) = 1}} \frac{1}{\lvert F \rvert} \bigg)^2 \\
= &\phi (R) \sum_{\substack{N \text{ monic} \\ \degree N \leq {z_R}' \\ (N,R) = 1}} \frac{2^{\omega(N)}}{\lvert N \rvert} \bigg( \sum_{\substack{ F \text{ monic} \\ \degree F < \frac{\degree R -\degree N}{2} \\ (F , R) = 1}} \frac{1}{\lvert F \rvert} - \sum_{\substack{ F \text{ monic} \\ \degree F \leq \frac{z_R -\degree N}{2} \\ (F , R) = 1}} \frac{1}{\lvert F \rvert} \bigg)^2 \\
\ll &\phi (R) \sum_{\substack{N \text{ monic} \\ \degree N \leq {z_R}' \\ (N,R) = 1}} \frac{2^{\omega(N)}}{\lvert N \rvert} \bigg( \frac{\phi (R)}{\lvert R \rvert} \omega (R) + \log \omega (R) \bigg)^2 \\
\ll &\lvert R \rvert \bigg( \frac{\phi (R)}{\lvert R \rvert} \bigg)^3 \big( \omega (R) \big)^2 \sum_{\substack{N \text{ monic} \\ \degree N \leq {z_R}'' \\ (N,R) = 1}} \frac{2^{\omega(N)}}{\lvert N \rvert} \\
\ll &\lvert R \rvert \bigg( \frac{\phi (R)}{\lvert R \rvert} \bigg)^3 \big( \omega (R) \big)^2 \bigg( \prod_{P \mid R} \frac{1 - \lvert P \rvert^{-1}}{1 + \lvert P \rvert^{-1}} \bigg) \big( \degree R \big)^2 ,
\end{align*}
where for the second relation we applied Proposition \ref{Sum over (A,R)=1, deg A <= x of 1/A} twice. For the use of this proposition one may wish to note that, because $\degree N \leq {z_R}'$, we have that $\frac{\degree R -\degree N}{2} \geq \frac{z_R -\degree N}{2} \geq \log_q \big( \frac{3}{\sqrt{2}}\big)^{\omega (R)}$, and so when $x = \frac{\degree R -\degree N}{2}$ or $x = \frac{z_R -\degree N}{2}$ we have that $\frac{2^{\omega (R)} x}{q^x} = O (1)$. For the last relation we applied Lemma \ref{Lemma, Sum of 2^omega(N) / |N| with (N,R)=1}. \\

\noindent Now we look at the second term on the far-RHS of (\ref{z_R < deg AB, split with z_R'}). Because ${z_R}' <\degree N < \degree R$, we have that $\frac{\degree R - \degree N}{2} < \log_q 9^{\frac{\omega (R)}{2}}$. Hence,

\begin{align*}
\phi (R) \sum_{\substack{N \text{ monic} \\ {z_R}' <\degree N < \degree R \\ (N,R) = 1}} \frac{2^{\omega(N)}}{\lvert N \rvert} \bigg( \sum_{\substack{ F \text{ monic} \\ \degree F < \frac{\degree R -\degree N}{2} \\ (F , R) = 1}} \frac{1}{\lvert F \rvert} \bigg)^2
\leq &\phi (R) \sum_{\substack{N \text{ monic} \\ \degree N < \degree R \\ (N,R) = 1}} \frac{2^{\omega(N)}}{\lvert N \rvert} \bigg( \sum_{\substack{ F \text{ monic} \\ \degree F < \log_q 9^{\frac{\omega (R)}{2}} \\ (F , R) = 1}} \frac{1}{\lvert F \rvert} \bigg)^2 \\
\ll &\lvert R \rvert \bigg( \frac{\phi (R)}{\lvert R \rvert} \bigg)^3 \big( \omega (R) \big)^2 \sum_{\substack{N \text{ monic} \\ \degree N \leq {z_R}' \\ (N,R) = 1}} \frac{2^{\omega(N)}}{\lvert N \rvert} \\
\ll &\lvert R \rvert \bigg( \frac{\phi (R)}{\lvert R \rvert} \bigg)^3 \big( \omega (R) \big)^2 \bigg( \prod_{P \mid R} \frac{1 - \lvert P \rvert^{-1}}{1 + \lvert P \rvert^{-1}} \bigg) \big( \degree R \big)^2 ,
\end{align*}
where, again, we have used Propositions \ref{Sum over (A,R)=1, deg A <= x of 1/A} and Lemma \ref{Lemma, Sum of 2^omega(N) / |N| with (N,R)=1}. \\

\noindent Hence,

\begin{align*}
\phi (R) \sum_{\substack{A,B,C,D \in \mathcal{M} \\ z_R < \degree AB < \degree R \\ z_R < \degree CD <\degree R \\ (ABCD,R) = 1 \\ AC = BD}} \frac{1}{\lvert ABCD \rvert^{\frac{1}{2}}}
\ll &\lvert R \rvert \bigg( \frac{\phi (R)}{\lvert R \rvert} \bigg)^3 \big( \omega (R) \big)^2 \bigg( \prod_{P \mid R} \frac{1 - \lvert P \rvert^{-1}}{1 + \lvert P \rvert^{-1}} \bigg) \big( \degree R \big)^2 \\
\ll & \phi^* (R) \bigg( \prod_{\substack{P \text{ prime} \\ P \mid R}}  \frac{\big( 1 - \lvert P \rvert^{-1} \big)^3 }{1 + \lvert P \rvert^{-1}} \bigg) (\degree R)^3 \omega (R) .
\end{align*}

\noindent \textbf{\underline{STEP 2.2:}} We now look at the second term on the far right-hand-side of (\ref{ B(X)^2 sum split into Equal and Nonequal}):

\begin{align*}
\phi (R) \sum_{\substack{A,B,C,D \text{ monic} \\ z_R < \degree AB < \degree R \\ z_R < \degree CD < \degree R \\ (ABCD,R) = 1 \\ AC \equiv BD (\modulus R) \\ AC \neq BD}} \frac{1}{\lvert ABCD \rvert^{\frac{1}{2}}}
= & \phi (R) \sum_{z_R < z_1 , z_2 < \degree R} \frac{1}{(q^{z_1 + z_2})^{\frac{1}{2}}} \sum_{\substack{A,B,C,D \text{ monic} \\ \degree AB = z_1 \\ \degree CD = z_2 \\ (ABCD,R) = 1 \\ AC \equiv BD (\modulus R) \\ AC \neq BD}} 1\\
= & \phi (R) \frac{1}{\phi (R)} \sum_{z_R < z_1 , z_2 < \degree R} q^{\frac{z_1 + z_2}{2}} (z_1 + z_2)^3. \\
\ll & \lvert R \rvert \big( \degree R \big)^3
\ll \phi^* (R) \bigg( \prod_{\substack{P \text{ prime} \\ P \mid R}}  \frac{\big( 1 - \lvert P \rvert^{-1} \big)^3 }{1 + \lvert P \rvert^{-1}} \bigg) (\degree R)^3 \omega (R)
\end{align*}
as $\degree R \rightarrow \infty$. The second relation follows from Lemma \ref{Lemma, off diagonal z_1 z_2 sum} with $F := R$. This can be applied because

\begin{align*}
z_1 + z_2 \geq 2z_R = 2 \degree R - 2\log_q 2^{\omega (R)} > \frac{19}{10} \degree R
\end{align*}
for large enough $\degree R$. \\

\noindent \textbf{\underline{STEP 2.3:}} Hence, we see that

\begin{align*}
\sumstar_{\chi \modulus R} \lvert b (\chi) \rvert^2
\ll & \phi^* (R) \bigg( \prod_{\substack{P \text{ prime} \\ P \mid R}}  \frac{\big( 1 - \lvert P \rvert^{-1} \big)^3 }{1 + \lvert P \rvert^{-1}} \bigg) (\degree R)^3 \omega (R) .
\end{align*}

\noindent \textbf{\underline{STEP 3:}} We will now look at $\sum_{\chi \modulus R}^* \lvert c (\chi) \rvert^2$. We have that

\begin{align*}
\sumstar_{\chi \modulus R} \lvert c (\chi) \rvert^2
\leq \sum_{\chi \modulus R} \lvert c (\chi) \rvert^2
= \sum_{\chi \modulus R} \lvert c_o (\chi) \rvert^2
-\sum_{\substack{\chi \modulus R \\ \chi \text{ even}}} \lvert c_o (\chi) \rvert^2
+\sum_{\substack{\chi \modulus R \\ \chi \text{ even}}} \lvert c_e (\chi) \rvert^2 .
\end{align*}

\noindent Now,

\begin{align*}
\sum_{\chi \modulus R} \lvert c_o (\chi) \rvert^2
= &\sum_{\chi \modulus R} \sum_{\substack{A,B,C,D \in \mathcal{M} \\ \degree AB = (\degree R) -1 \\ \degree CD = (\degree R) -1}} \frac{\chi (AC) \conj\chi (BD)}{\lvert ABCD \rvert^{\frac{1}{2}}} \\
= &\phi (R) \sum_{\substack{A,B,C,D \in \mathcal{M} \\ \degree AB = (\degree R) -1 \\ \degree CD = (\degree R) -1 \\ (ABCD,R) = 1 \\ AC = BD}} \frac{1}{\lvert ABCD \rvert^{\frac{1}{2}}}
+ \phi (R) \sum_{\substack{A,B,C,D \in \mathcal{M} \\ \degree AB = (\degree R) -1 \\ \degree CD = (\degree R) -1 \\ (ABCD,R) = 1 \\ AC \equiv BD \\ AC \neq BD}} \frac{1}{\lvert ABCD \rvert^{\frac{1}{2}}} .
\end{align*}
For the first term on the far-RHS we have that

\begin{align*}
\sum_{\substack{A,B,C,D \in \mathcal{M} \\ \degree AB = (\degree R) -1 \\ \degree CD = (\degree R) -1 \\ (ABCD,R) = 1 \\ AC = BD}} \frac{1}{\lvert ABCD \rvert^{\frac{1}{2}}}
\leq &\sum_{\substack{N \in \mathcal{M} \\ \degree N \leq (\degree R) -1}} \frac{2^{\omega (N)}}{\lvert N \rvert} \bigg( \sum_{\substack{F \in \mathcal{M} \\ \degree F = \frac{\degree R - \degree N -1}{2}}} \frac{1}{\lvert F \rvert} \bigg)^2 \\
= &\sum_{\substack{N \in \mathcal{M} \\ \degree N \leq (\degree R) -1}} \frac{2^{\omega (N)}}{\lvert N \rvert}
\ll \big( \degree R \big)^2 .
\end{align*}
For the second term we have that

\begin{align*}
\sum_{\substack{A,B,C,D \in \mathcal{M} \\ \degree AB = (\degree R) -1 \\ \degree CD (\degree R) -1 \\ (ABCD,R) = 1 \\ AC \equiv BD \\ AC \neq BD}} \frac{1}{\lvert ABCD \rvert^{\frac{1}{2}}}
= \frac{q}{\lvert R \rvert} \sum_{\substack{A,B,C,D \in \mathcal{M} \\ \degree AB = (\degree R) -1 \\ \degree CD (\degree R) -1 \\ (ABCD,R) = 1 \\ AC \equiv BD \\ AC \neq BD}} 1
\ll \frac{\lvert R \rvert }{\Phi (R)} \big( \degree R \big)^3 ,
\end{align*}
where we have used Lemma \ref{Lemma, off diagonal z_1 z_2 sum}. Hence, 

\begin{align*}
\sum_{\chi \modulus R} \lvert c_o (\chi) \rvert^2
\ll \lvert R \rvert \big( \degree R \big)^3 
\ll \phi^* (R) \bigg( \prod_{\substack{P \text{ prime} \\ P \mid R}}  \frac{\big( 1 - \lvert P \rvert^{-1} \big)^3 }{1 + \lvert P \rvert^{-1}} \bigg) (\degree R)^3 \omega (R)
\end{align*}

\noindent Similarly, by using Lemma \ref{Lemma, off diagonal z_1 z_2 sum with alpha} for the even case, we can show, for $a=0,1,2,3$, that

\begin{align*}
\sum_{\chi \modulus R} \sum_{\substack{A,B,C,D \in \mathcal{M} \\ \degree AB = \degree R - a \\ \degree CD = \degree R - a}} \frac{\chi (AC) \conj\chi (BD)}{\lvert ABCD \rvert^{\frac{1}{2}}} 
\ll &\lvert R \rvert \big( \degree R \big)^3 
\ll \phi^* (R) \bigg( \prod_{\substack{P \text{ prime} \\ P \mid R}}  \frac{\big( 1 - \lvert P \rvert^{-1} \big)^3 }{1 + \lvert P \rvert^{-1}} \bigg) (\degree R)^3 \omega (R), \\
\sum_{\substack{\chi \modulus R \\ \chi \text{ even}}} \sum_{\substack{A,B,C,D \in \mathcal{M} \\ \degree AB = \degree R - a \\ \degree CD = \degree R - a}} \frac{\chi (AC) \conj\chi (BD)}{\lvert ABCD \rvert^{\frac{1}{2}}} 
\ll &\lvert R \rvert \big( \degree R \big)^3 
\ll \phi^* (R) \bigg( \prod_{\substack{P \text{ prime} \\ P \mid R}}  \frac{\big( 1 - \lvert P \rvert^{-1} \big)^3 }{1 + \lvert P \rvert^{-1}} \bigg) (\degree R)^3 \omega (R) .
\end{align*}
Hence, by using the Cauchy-Schwarz inequality, we can deduce that

\begin{align*}
\sumstar_{\chi \modulus R} \lvert c (\chi) \rvert^2
\ll \phi^* (R) \bigg( \prod_{\substack{P \text{ prime} \\ P \mid R}}  \frac{\big( 1 - \lvert P \rvert^{-1} \big)^3 }{1 + \lvert P \rvert^{-1}} \bigg) (\degree R)^3 \omega (R) . \\
\end{align*}

\noindent \textbf{\underline{STEP 4:}} From steps 1 to 3, and the use of the Cauchy-Schwarz inequality (as described at the start of the proof), the result follows.
\end{proof}

\vspace{1cm}


\bibliography{YiasemidesBibliography1}{}
\bibliographystyle{bibstyle1}

\end{document}